\documentclass[12pt, twoside]{article}
\usepackage{amsmath,amsthm,amssymb}
\usepackage{times}
\usepackage{enumerate}

\def\titlerunning#1{\gdef\titrun{#1}}
\makeatletter
\def\author#1{\gdef\autrun{\def\and{\unskip, }#1}\gdef\@author{#1}}
\def\address#1{{\def\and{\\\hspace*{18pt}}\renewcommand{\thefootnote}{}%
\footnote {#1}}%
\markboth{\autrun}{\titrun}}
\makeatother
\def\email#1{e-mail: #1}
\def\subjclass#1{{\renewcommand{\thefootnote}{}%
\footnote{\emph{Mathematics Subject Classification (2010):} #1}}}
\def\keywords#1{\par\medskip
\noindent\textbf{Keywords.} #1}

\usepackage{amssymb}
\usepackage{hyperref} 
\usepackage{mathrsfs,bbold}
\usepackage{amsmath,amsthm, amsfonts, amssymb}
\usepackage{enumerate}
\usepackage{dsfont}
\usepackage{color}
\usepackage[a4paper]{geometry}
\usepackage[utf8]{inputenc}
\usepackage{stmaryrd}
\usepackage{titlesec}
\usepackage{mathabx}
\usepackage{appendix}
\usepackage{todonotes, fancyhdr}
\usepackage{upgreek}

\numberwithin{equation}{section}

\renewcommand{\epsilon}{\varepsilon}

\def\XXint#1#2#3{{\setbox0=\hbox{$#1{#2#3}{\int}$}
     \vcenter{\hbox{$#2#3$}}\kern-.5\wd0}}

\numberwithin{subsection}{section}
\numberwithin{subsubsection}{subsection}

\newtheorem{thm}{Theorem}[section]

\newtheorem{lem}[thm]{Lemma}

\newtheorem{prop}[thm]{Proposition}%[chapter]
\newtheorem{defn}[thm]{Definition}%[chapter]

\newtheorem{rem}[thm]{Remark}
\newtheorem{Example}[thm]{Example}

\numberwithin{equation}{section}

\newtheorem*{definition*}{Definition}
\newtheorem*{theorem*}{Theorem}
\newtheorem*{Remark*}{Remark}
\numberwithin{equation}{section} % Une �quation p.q. est la q �me �quation de la section p.

\newcommand\restr[2]{{% we make the whole thing an ordinary symbol
  \left.\kern-\nulldelimiterspace % automatically resize the bar with \right
  #1 % the function
  \vphantom{\big|} % pretend it's a little taller at normal size
  \right|_{#2} % this is the delimiter
  }}

\def\WW{{\mathbb W}}
\def\EE{{\mathbb E}}
\def\RR{{\mathbb R}}

\def\LL{{\mathbb L}}

\def\PP{{\mathbb P}}

\def\cL{{\mathcal L}}

\def\cC{{\mathcal C}}
\def\cP{{\mathcal P}}

\newcommand\indep{\protect\mathpalette{\protect\independenT}{\perp}}
\def\independenT#1#2{\mathrel{\rlap{$#1#2$}\mkern2mu{#1#2}}}

\newcommand\indeprev{\protect\mathpalette{\protect\independenT}{\top}}
\def\independenT#1#2{\mathrel{\rlap{$#1#2$}\mkern2mu{#1#2}}}

\makeatletter
\DeclareFontFamily{OMX}{MnSymbolE}{}
\DeclareSymbolFont{MnLargeSymbols}{OMX}{MnSymbolE}{m}{n}
\SetSymbolFont{MnLargeSymbols}{bold}{OMX}{MnSymbolE}{b}{n}
\DeclareFontShape{OMX}{MnSymbolE}{m}{n}{
    <-6>  MnSymbolE5
   <6-7>  MnSymbolE6
   <7-8>  MnSymbolE7
   <8-9>  MnSymbolE8
   <9-10> MnSymbolE9
  <10-12> MnSymbolE10
  <12->   MnSymbolE12
}{}
\DeclareFontShape{OMX}{MnSymbolE}{b}{n}{
    <-6>  MnSymbolE-Bold5
   <6-7>  MnSymbolE-Bold6
   <7-8>  MnSymbolE-Bold7
   <8-9>  MnSymbolE-Bold8
   <9-10> MnSymbolE-Bold9
  <10-12> MnSymbolE-Bold10
  <12->   MnSymbolE-Bold12
}{}

\let\llangle\@undefined
\let\rrangle\@undefined
\DeclareMathDelimiter{\llangle}{\mathopen}%
                     {MnLargeSymbols}{'164}{MnLargeSymbols}{'164}
\DeclareMathDelimiter{\rrangle}{\mathclose}%
                     {MnLargeSymbols}{'171}{MnLargeSymbols}{'171}
\makeatother

\begin{document}

\baselineskip=17pt

\titlerunning{Mean field rough equations}
\title{Solving mean field rough differential equations}

\author{I. Bailleul \and R. Catellier \and F. Delarue}

\maketitle

\address{I. Bailleul: Univ. Rennes, CNRS, IRMAR - UMR 6625, F-35000 Rennes, France; \email{ismael.bailleul@univ-rennes1.fr}
\and R. Catellier: Univ. C\^ote d'Azur, CNRS, Laboratoire J.A. Dieudonn\'e, 06108 Nice, France; \email{remi.catellier@unice.fr}
\and F. Delarue: Univ. C\^ote d'Azur, CNRS, Laboratoire J.A. Dieudonn\'e, 06108 Nice, France; \email{francois.delarue@unice.fr}
}

\subjclass{Primary 60H10, 60G99}

\begin{abstract}
We provide in this work a robust solution theory for random rough differential equations of mean field type
$$
dX_t = V\big( X_t,\cL(X_t)\big)dt + \textrm{F}\bigl( X_t,\cL(X_t)\bigr) dW_t,
$$
where $W$ is a random rough path and $\cL(X_t)$ stands for the law of $X_t$, with mean field interaction in both the drift and diffusivity.  The analysis requires the introduction of a new rough path-like setting and an associated notion of controlled path. We use crucially Lions' approach to differential calculus on Wasserstein space along the way.

\keywords{Random rough differential equations, controlled paths, mean field interaction}
\end{abstract}

%-------------------------%
\section{Introduction}	
%-------------------------%

The first works on mean field stochastic dynamics and interacting diffusions/Markov processes have their roots in Kac's simplified approach to kinetic theory \cite{Kac} and McKean's work \cite{McKean} on nonlinear parabolic equations. They provide the description of evolutions $(\mu_t)_{t\geq 0}$ in the space of probability measures  under the form of a pathspace random dynamics 
\begin{equation}
\label{eq:1:1}
\begin{split}
dX_t(\omega) &= V\bigl(X_t(\omega),\mu_t\bigr)dt + \textrm{F}\bigl(X_t(\omega), \mu_t\bigr) dW_t(\omega) \quad ;  \quad \mu_t := \cL(X_t),
\end{split}
\end{equation}
(where $\cL(A)$ stands for the law of a random variable $A$) and relate it to the empirical behaviour of large systems of interacting dynamics. The main emphasis of subsequent works has been on proving propagation of chaos and other limit theorems, and giving stochastic representations of solutions to nonlinear parabolic equations under more and more general settings; see \cite{Sznitman,Tanaka,Gartner,DawsonGartner,DawsonVaillancourt,meleard,jourdain-meleard,BudhirajaDupuisFischer,BudhirajaWu} for a tiny sample. Classical stochastic calculus makes sense of equation \eqref{eq:1:1}, in a probabilistic setting $(\Omega,\mathcal{F},\PP)$, only when the process $W$ is a semi-martingale under $\PP$, for some filtration, and the integrand is predictable. However, this setting happens to be too restrictive in a number of situations, especially when the diffusivity is random. 
This prompted several authors to address equation \eqref{eq:1:1} by means of rough paths theory. Indeed, one may understand rough paths theory as a natural framework for providing probabilistic models of interacting populations, beyond the realm of It\^o calculus. Cass and Lyons \cite{CassLyons} did the first study of mean field random rough differential equations and proved the well-posed character of equation \eqref{eq:1:1}, and propagation of chaos for an associated system of interacting particles, under the assumption that   there is no mean field interaction in the diffusivity, \textit{i.e.} $\textrm{F}(x,\mu) = \textrm{F}(x)$,   
and the drift depends linearly on the mean field interaction, \textit{i.e.}
  $V(x,\mu) = \int V(x,y)\,\mu(dy)$, 
   for some function $V(\cdot,\cdot)$ on $\RR^d\times\RR^d$.   
   
The method of proof of Cass and Lyons depends crucially on both assumptions. Bailleul extended partly these results in \cite{BailleulRMI} by proving well-posedness of the mean field rough differential equation \eqref{eq:1:1} in the case where the drift depends nonlinearly on the interaction term and the diffusivity is still independent of the interaction, and by proving an existence result when the diffusivity depends on the interaction. The naive approach to showing well-posedness of equation \eqref{eq:1:1} in its general form consists in treating the measure argument as a time argument. However, this is of a rather limited scope since, in this generality, one cannot expect the time dependence in F to be better than $\frac{1}{p}$-H\"older if the rough path $W$ is itself $\frac{1}{p}$-H\"older. Clearly, such a time regularity is not sufficient to make sense of the rough integral $\int \textrm{F}(\cdots)\,dW$ in the case $p \geq 2$. This serious issue explains why, so far in the literature, the coefficient F has been assumed to be a function of  the sole variable $x$.

Including the time component as one of the components of $W$ brings back the study of equation \eqref{eq:1:1} to the study of equation 
\begin{equation}
\label{EqRDE}
\begin{split}
dX_t(\omega) &= \textrm{F} \bigl( X_t(\omega),\cL(X_t)\bigr) dW_t(\omega) \quad ; \quad 
\mu_t := \cL(X_t),
\end{split}
\end{equation}
this is the precise purpose of the present paper. Treating the drift as part of the diffusivity has the drawback that we shall impose on $V$ some regularity conditions stronger than needed. Our method accommodates the general case but we leave the reader the pleasure of optimizing the details and concentrate on the new features of our approach, working on equation \eqref{EqRDE}. The raw driver $\big(W_t(\omega)\big)_{t\geq 0}$ will be assumed to take values in some $\RR^m$ and to be $\tfrac1p$-H\"older continuous, for $p \in [2,3)$, and the one form F will be an ${\mathscr L}(\RR^m,\RR^d)$-valued function on $\RR^d\times \mathcal{P}_2(\RR^d)$, where 
${\mathscr L}(\RR^m,\RR^d)$ is the space of linear mappings from $\RR^m$ to $\RR^d$ and 
$\cP_{2}(\RR^d)$ is the so-called Wasserstein space of probability measures $\mu$ with a finite second-order moment. Inspired by Lions' approach \cite{Lions,LionsCardialiaguet,CarmonaDelarue_book_I} to differential calculus on $\cP_{2}(\RR^d)$, one of the key point in our \textcolor{black}{analysis} is to lift the function F into a function $\widehat{\textrm{F}}$ defined on the space $\RR^d \times \LL^2\big(\Omega,{\mathcal F},\PP;\RR^d\big)$, given by the formula
\begin{equation}
\label{eq:lifting}
\widehat{\textrm{F}}\bigl(x,Z\bigr) = \textrm{F}(x,\cL(Z)\bigr),
\end{equation}
for $x \in \RR^d$ and $Z \in \LL^2(\Omega,{\mathcal F},\PP;\RR^d)$. So, we may rewrite equation \eqref{EqRDE} as
\begin{equation}
\label{eq:1:2}
dX_t(\omega) = \widehat{\textrm{F}}\bigl( X_t(\omega), X_t(\cdot) \bigr) dW_{t}(\omega).
\end{equation}
We used the notation $X_t(\cdot)$ to distinguish the realization $X_t(\omega)$ of the random variable $X_t$ at point $\omega$ from the random variable itself, seen as an element of the space ${\mathbb L}^2\big(\Omega,{\mathcal F},\PP;\RR^d\big)$. So, $X_t(\cdot)$ is a random variable, and thus an infinite-dimensional object, whilst $X_t(\omega)$ is a finite-dimensional vector. We feel that this writing is sufficiently explicit to remove the hat over F.

Our main well-posedness result is stated below, in a preliminary form only. The precise statement requires additional ingredients that we introduce later on in the text. In this first formulation
\begin{itemize}
   \item the quantity $w(\cdot,\cdot) = \big(w(s,t)\big)_{0 \leq s <t}$ is a random control function that is used to quantify the regularity of the solution path on subintervals $[s,t]$ of a given finite interval $[0,T]$, using some associated notion of $p$-variation for the same $p$ as above, 
   
   \item the quantity $N([0,T],\alpha)$ is some local accumulated variation of the `rough lift' of $W$ that counts the increments of $w$ of size $\alpha$ over a bounded interval $[0,T]$ for a given $\alpha>0$;
\end{itemize}
see Section \ref{SectionRoughStructure} for the set-up. The regularity assumptions on the diffusivity F are spelled-out in 
Subsection 
\ref{SubsectionStability}
and in 
Section \ref{SectionSolving}, see \textbf{\textbf{\bf Regularity assumptions 1}} and \textbf{\textbf{\bf Regularity assumptions 2}} therein.

\begin{thm}
\label{ThmMain}
Let \emph{F} satisfy the regularity assumptions \textbf{\textbf{\bf Regularity assumptions 1}} and \textbf{\textbf{\bf Regularity assumptions 2}}. Assume there exists a positive time horizon $T$ such that the random variables $w(0,T)$ and \textcolor{black}{$\bigl(N\big((0,T),\alpha\big)\bigr)_{\alpha >0}$} have sub and super exponential tails, respectively,
\begin{itemize}
   \item $\PP \bigl( w(0,T) \geq t \bigr) \leq c_{1} \exp \bigl( - t^{\varepsilon_1} \bigr)$,
   
   \item $\PP \bigl( N\big([0,T],\alpha\big) \geq t \bigr) \leq c_{2}(\alpha) \exp \bigl( - t^{1+\varepsilon_2(\alpha)} \bigr)$, \quad $\alpha >0$,
\end{itemize}
\textcolor{black}{for some positive constants $c_{1}$ and $\epsilon_1$ and possibly $\alpha$-dependent positive constants $c_{2}(\alpha)$ and $\epsilon_2(\alpha)$}. Then for any $d$-dimensional square-integrable random variable $X_{0}$, the mean field rough differential equation 
$$
dX_t = \textrm{\emph{F}}\big( X_t, \mathcal{L}(X_t)\big)\,dW_t
$$
has a unique solution defined on the whole interval $[0,T]$.
\end{thm}

Results of that form seem out of reach of the methods used in \cite{CassLyons,BailleulRMI}. Theorem \ref{ThmMain} applies in particular to mean field rough differential equations driven by some fractional Brownian motion with Hurst parameter greater than $\frac{1}{3}$, other Gaussian processes or some Markovian rough paths; see Section \ref{SectionRoughStructure}. Importantly, the solution is shown to depend continuously on the driving `rough path', in a quantitative sense detailed in Theorem \ref{ThmContinuity}. As an example \textcolor{black}{that fits our regularity assumptions}, one can solve the above mean field rough differential equation with 
$\textrm{F}(x,\mu) = \int f(x,y)\mu(dy)$,
for some fuction $f$ of class $C^3_b$ (meaning that $f$ is bounded and has bounded derivatives of order 1, 2 and 3), or with
$\textrm{F}(x,\mu) = g\bigl(x, \int_{\RR^d} y\mu(dy)\bigr)$,
for some function $g$ of class $C^3_b$. The Curie-Weiss model, where F is of the form $\textrm{F}(x,\mu) = \nabla U(x) + \int (x-y)\mu(dy)$, falls outside the scope of what is written here, because of the linear growth rate in $x$, but is within reach of our method.

One of the difficulties in solving equation \eqref{EqRDE} comes from the fact that it happens not to be sufficient to consider each signal $W_\bullet(\omega)$ as the first level of a rough path; one somehow needs to consider the whole family $\big(W_\bullet(\omega)\big)_{\omega\in\Omega}$ as an infinite-dimensional rough path. This leads us to defining in Section \ref{SectionRoughStructure} a rough setting where $\bigl( W_{t}(\omega),W_{t}(\cdot) \bigr)_{0 \leq t \leq T}$ is, for each $\omega$, the first level of a rough path over $\RR^m\times \LL^q\big(\Omega,{\mathcal F},\PP;\RR^m\big)$; \textcolor{black}{seemingly,} the natural choice for $q$, as dictated by the aforementioned lifting procedure of the Wasserstein space, is $q=2$; we shall actually need a larger value. Unlike the seminal works \cite{CassLyons,BailleulRMI} that set the scene in Davie's approach of rough differential equations, such as reshaped by Friz-Victoir and Bailleul respectively, we use here Gubinelli's versatile approach of controlled paths to make sense of equation \eqref{EqRDE}. Our mixed finite/infinite dimensional setting introduces an interesting twist in the notion of controlled path presented in Section \ref{SubsectionControlledTraj}. Defining the rough integral of a controlled path with respect to a rough driver is done classically in Section \ref{SubsectionIntegral} using the sewing lemma. We prove stability of a certain class of controlled paths by nonlinear mappings in Section \ref{SubsectionStability}, which is precisely the place where Lions' differential calculus on $\cP_{2}(\RR^d)$ comes in. One then has all the ingredients needed to formulate in Section \ref{SectionSolving} equation \eqref{EqRDE} as a fixed point problem in some space of controlled paths. Local well-posedness is proved, and sufficient conditions on the law of the driver are given to get well-posedness on any fixed time interval. As expected from any solution theory for rough differential equations, the solution depends continuously on all the parameters in the equation, most notably
its law depends continuously on the law of the driving rough path, as shown in Section \ref{weak}. This latter point is used in the {forthcoming} companion paper \cite{BCDParticleSystem}
 to provide a proof of propagation of chaos for an interacting particle system associated with equation \eqref{EqRDE} and quantify the convergence rate\footnote{We also refer to Section 4 of the Arxiv deposit
\cite{BCDArxiv}; \cite{BCDArxiv} encompasses the original versions of this work and of the {forthcoming} companion one \cite{BCDParticleSystem}.}. Among others, it recovers Sznitman' seminal work \cite{Sznitman} on the case where the noise is a Brownian motion.

While Lyons formulated his theory in a Banach setting from the begining \cite{Lyons98}, the theory has mainly been explored for finite dimensional drivers, with the noticeable exception of the works of Ledoux, Lyons and Qian on Banach space valued rough paths \cite{LedouxLyonsQian,LyonsQianDiffeo}, Dereich follow-up works \cite{Dereich,DereichDimitroff}, Kelly and Melbourne application to homogenization of fast/slow systems of ordinary differential equations \cite{KellyMelbourne}, and Bailleul and Riedel's work on rough flows \cite{BailleulRiedel}. One can see the present work as another illustration of the strength of the theory in its full generality. However, although the underlying rough set-up associated to $(W_{t}(\omega),W_{t}(\cdot))_{0 \leq t \leq T}$ is a mixed finite/infinite dimensional object, a solution to the mean field rough differential equation is more than a solution to a rough differential equation driven by an infinite dimensional rough path. Indeed, the mean field structure imposes an additional fixed point condition, which is to identify the finite dimensional component of the solution as the $\omega$-realization of the infinite dimensional component. This is precisely this constraint that makes the equation difficult to solve and that explains the need for a specific analysis. 
\vspace{4pt}

\noindent \textbf{\textbf{Notations.}} We gather here a number of notations that will be used throughout the text.

\textcolor{gray}{$\bullet$} We set ${\mathcal S}_{2} := \big\{ (s,t) \in [0,\infty)^2 : s \leq t\big\}$, 
and 
$
{\mathcal S}_{2}^T := \big\{(s,t) \in [0,T]^2 : s \leq t \big\}.
$

\textcolor{gray}{$\bullet$} We denote by $(\Omega,{\mathcal F},\PP)$ an atomless Polish probability space, \textcolor{black}{${\mathcal F}$ standing for the completion of the Borel $\sigma$-field under $\PP$}, and denote by $\langle \cdot \rangle$ the expectation operator, by $\langle \cdot \rangle_{r}$, for $r \in [1,+\infty]$, the $\LL^r$-norm on $(\Omega,{\mathcal F},\PP)$ and by $\llangle \cdot \rrangle$ and $\llangle \cdot \rrangle_{r}$ the expectation operator and the $\LL^r$-norm on 
$
\big(\Omega^2,{\mathcal F}^{\otimes 2},\PP^{\otimes 2}\big).$
When $r$ is finite, $\LL^r(\Omega,{\mathcal F},\PP;\RR)$ is separable as $\Omega$ is Polish.

\textcolor{gray}{$\bullet$}
As for processes $X_{\bullet}=(X_{t})_{t \in I}$, defined on a time interval $I$, we often write $X$ for $X_{\bullet}$.

%---------------------------------------------%
\section{Probabilistic Rough Structure}
\label{SectionRoughStructure}
%---------------------------------------------%

We define in this section a notion of rough path appropriate for our purpose. It happens to be a mixed finite/infinite dimensional object. Throughout the section, we work on a finite time horizon $[0,T]$, for a given $T>0$.

\textcolor{gray}{$\bullet$} We define the first level of our rough path structure as an $\omega$-indexed pair of paths
\begin{equation}
\label{EqLevelOne}
\bigl( W_t(\omega),W_t(\cdot) \bigr)_{0\leq t\leq T},
\end{equation}
where $\big(W_t(\cdot)\big)_{0\leq t\leq T}$ 
\textcolor{black}{is a collection of $q$-integrable $\RR^m$-valued random variables on
the space
$(\Omega,{\mathcal F},\PP)$, 
which we regard as a deterministic $\LL^q(\Omega,{\mathcal F},\PP;\RR^m)$-valued path}, for some exponent $q\geq 1$, and
$\bigl(W_{t}(\omega)\bigr)_{0 \le t \le T}$ stands for the realizations of these random variables along the outcome $\omega \in \Omega$; so the pair \eqref{EqLevelOne} takes values in $\RR^m\times \LL^q(\Omega,{\mathcal F},\PP;\RR^m)$. \textcolor{black}{As we already explained, a} natural choice would be to take $q=2$, but for technical reasons that will get clear below we shall require $q \geq 8$. 

\textcolor{gray}{$\bullet$} The second level of the rough path structure includes a two-index path $\big(\WW_{s,t}(\omega)\big)_{0 \leq s \leq t \leq T}$ with values in $\RR^{m \times m}$, obtained as the $\omega$-realizations of a collection of $q$-integrable random variables $\big(\WW_{s,t}(\cdot)\big)_{0\leq s\leq t\leq T}$ defined on $\Omega$; importantly, this second level also comprises the sections $\big(\WW_{s,t}^{\indep}(\omega,\cdot)\big)_{0\leq s\leq t\leq T}$ and $\big(\WW_{s,t}^{\indep}(\cdot,\omega)\big)_{0\leq s\leq t\leq T}$ of a collection of $\RR^{m \times m}$-valued random variables defined on the product space $ \big(\Omega^2,{\mathcal F}^{\otimes 2},\PP^{\otimes 2}\bigr)$ and considered as a deterministic $\LL^q\big(\Omega^2,{\mathcal F}^{\otimes 2},\PP^{\otimes 2};\RR^{m \times m}\big)$-valued path $\big(\WW_{s,t}^{\indep}(\cdot,\cdot)\big)_{0\leq s\leq t\leq T}$. Each ${\mathbb W}_{s,t}^{\indep}(\cdot,\cdot)$, for $(s,t) \in {\mathcal S}_{2}^T$, belonging to the space $\LL^q\bigl(\Omega^2,{\mathcal F}^{\otimes 2},\PP^{\otimes 2};\RR^{m \times m}\bigr)$, we have 
\begin{equation}
\label{eq:finite:q:moment:section}
 \bigl\langle  {\mathbb W}_{s,t}^{\indep}(\omega,\cdot)  \bigr\rangle_{q} < \infty,
\quad 
\bigl\langle {\mathbb W}_{s,t}^{\indep}(\cdot,\omega) \bigr\rangle_{q} < \infty,
\end{equation}
for $\PP$-a.e. $\omega \in \Omega$. Below, we shall assume \eqref{eq:finite:q:moment:section} to be true for every $\omega \in \Omega$. This is not such a hindrance since we can modify in a quite systematic way the definition of the rough path structure on the null event where \eqref{eq:finite:q:moment:section} fails; this is exemplified in Proposition \ref{prop:example:construction} below. Taken this assumption for granted, we can regard $\Omega \ni \omega \mapsto {\mathbb W}^{\indep}_{s,t}(\omega,\cdot)$ and $\Omega \ni \omega \mapsto {\mathbb W}^{\indep}_{s,t}(\cdot,\omega)$ as random variables with values in $\LL^q(\Omega,{\mathcal F},\PP;\RR^{m \times m})$: Since $\LL^q(\Omega,{\mathcal F},\PP;\RR^{m \times m})$ is separable, it suffices to notice from Fubini's theorem that, for any $Z \in \LL^q(\Omega,{\mathcal F},\PP;\RR^{m \times m})$, $\Omega \ni \omega \mapsto \bigl\langle {\mathbb W}_{s,t}^{\indep}(\omega,\cdot) - Z \bigr\rangle_{q}$ is measurable, and similarly for ${\mathbb W}_{s,t}^{\indep}(\cdot,\omega)$.

Hence, the entire second level has the form of an $\omega$-dependent two-index path with values in $\big(\RR^m \times \LL^q(\Omega,{\mathcal F},\PP;\RR^m)\big)^{\otimes 2}$ \textcolor{black}{and is} encoded in matrix form as
\begin{equation}
\label{eq:W:2}
\left(
\begin{array}{cc}
\WW_{s,t}(\omega) &\WW_{s,t}^{\indep}(\omega,\cdot)
\\
\WW_{s,t}^{\indep}(\cdot,\omega) &\WW_{s,t}^{\indep}(\cdot,\cdot)
\end{array}
\right)_{0 \leq s \leq t\leq T}. 
\end{equation}
Here,   

\begin{itemize}
   \item $\WW_{s,t}(\omega)$ is in $(\RR^{m})^{\otimes 2} \simeq \RR^{m \times m}$,  
   
   \item $\WW_{s,t}^{\indep}(\omega,\cdot)$ is in $\RR^m \otimes \LL^q\big(\Omega,{\mathcal F},\PP;\RR^m\big)\simeq \LL^q\big(\Omega,\textcolor{black}{\mathcal F},\PP;\RR^{m \times m}\big)$,   
   
   \item $\WW_{s,t}^{\indep}(\cdot, \omega)$ is in $\LL^q\big(\Omega,{\mathcal F},\PP;\RR^m\big) \otimes \RR^m \simeq \LL^q\big(\Omega,\textcolor{black}{\mathcal F},\PP;\RR^{m \times m}\big)$,
   
 \item $\WW_{s,t}^{\indep}(\cdot,\cdot)$ is in $\LL^q\bigl(\Omega^{\otimes 2},{\mathcal F}^{\otimes 2},\PP^{\otimes 2};
 \RR^{m \times m} \bigr)$, the realizations of which read in the form $\Omega^2 \ni (\omega,\omega') \mapsto \WW_{s,t}^{\indep}(\omega,\omega') \in \RR^{m \times m}$ and the two sections of which are precisely given by $\WW_{s,t}^{\indep}(\omega,\cdot) : \Omega \ni \omega' \mapsto \WW_{s,t}^{\indep}(\omega,\omega')$, and $\WW_{s,t}^{\indep}(\cdot,\omega) \ni \omega' \mapsto \WW_{s,t}^{\indep}(\omega',\omega)$, for $\omega \in \Omega$.  
\end{itemize}
{Below, we formulate several additional assumptions on the rough path structure, the introduction of which is rather lengthy and is, for that reason, split into three distinct subsections.}

\subsection{Algebraic conditions} As usual with rough paths, algebraic consistency requires that Chen's relation\textcolor{black}{s}
\begin{equation}
\label{eq:chen}
\begin{split}
&\WW_{r,t}(\omega) = \WW_{r,s}(\omega) + \WW_{s,t}(\omega) + W_{r,s}(\omega) \otimes W_{s,t}(\omega),   \\
&\WW_{r,t}^{\indep}(\cdot,\omega) = \WW_{r,s}^{\indep}(\cdot,\omega) + \WW_{s,t}^{\indep}(\cdot,\omega) + W_{r,s}(\cdot) \otimes W_{s,t}(\omega),   \\
&\WW_{r,t}^{\indep}(\omega,\cdot) = \WW_{r,s}^{\indep}(\omega,\cdot) + \WW_{s,t}^{\indep}(\omega,\cdot) + W_{r,s}(\omega) \otimes W_{s,t}(\cdot),   \\
&\WW_{r,t}^{\indep}(\cdot,\cdot) = \WW_{r,s}^{\indep}(\cdot,\cdot) + \WW_{s,t}^{\indep}(\cdot,\cdot) + W_{r,s}(\cdot) \otimes W_{s,t}(\cdot),
\end{split}
\end{equation}
hold for any $0\leq r \le s \le t\leq T$. We used here the very convenient notation 
$
f_{r,s} := f_{s} - f_{r}$, 
for a function $f$ from $[0,\infty)$ into a vector space. In \eqref{eq:chen} and throughout, we denote  by $X(\cdot) \otimes Y(\cdot)$, 
for any two $X$ and $Y$ in $\LL^q(\Omega,{\mathcal F},\PP;\RR^m)$, the random variable
$
\bigl(\omega,\omega') \mapsto \big(X_{i}(\omega) Y_{j}(\omega')\bigr)_{1 \leq i,j \leq m}$
defined on $\Omega^2$. It is in $\LL^q\big(\Omega^2,\textcolor{black}{\mathcal F}^{\otimes 2},\PP^{\otimes 2};\RR^{m \times m}\big)$.

\begin{rem}
The last three lines in Chen's relations \eqref{eq:chen} are somewhat redundant. 
Assume indeed that we are given a collection of random variables 
$\bigl( {\mathbb W}^{\indep}_{s,t}(\cdot,\cdot) \bigr)_{0 \leq s \leq t \leq T}$
satisfying the last line of \eqref{eq:chen}. Then, 
for all $0 \leq r \leq s \leq t \leq T$
and for  
$\PP^{\otimes 2}$-a.e. $(\omega,\omega') \in \Omega^2$,  
\begin{equation*}
{\mathbb W}^{\indep}_{r,t}(\omega,\omega') = {\mathbb W}^{\indep}_{r,s}(\omega,\omega') + {\mathbb W}^{\indep}_{s,t}(\omega,\omega') + W_{r,s}(\omega) \otimes W_{s,t}(\omega').
\end{equation*}
Clearly, for $\PP$-almost every $\omega \in \Omega$, the second and third lines in 
\eqref{eq:chen} hold true as well. This is slightly weaker than the formulation \eqref{eq:chen} as, therein, the second and
 third lines are required to hold for all $\omega \in \Omega$. As exemplified in the proof of Proposition \ref{prop:example:construction}, one may modify the definition of ${\mathbb W}^{\indep}$ {on a null event} so that the second and third lines in \eqref{eq:chen} hold true  for all $\omega$ and 
 for all $0 \leq r \leq s \leq t \leq T$.
\end{rem}

\begin{defn}
We shall denote by ${\boldsymbol W}(\omega)$ the \textbf{\textbf{rough set-up}} specified by the $\omega$-dependent collection of maps given by \eqref{EqLevelOne} and \eqref{eq:W:2}. 
\end{defn}

As for the component $\WW^{\indep}$ of ${\boldsymbol W}(\omega)$, the notation $\indep$ is used to indicate, as we shall make it clear below, that $\WW_{s,t}^{\indep}(\cdot,\cdot)$ should be thought of as the random variable
\begin{equation*}
(\omega,\omega') \mapsto \int_{s}^t \Bigl( W_{r}(\omega) - W_{s}(\omega) \Bigr) \otimes dW_{r}(\omega'). 
\end{equation*}
Since $\Omega^2 \ni (\omega,\omega') \mapsto (W_{t}(\omega))_{{0 \le t \le T}}$ and $\Omega^2 \ni (\omega,\omega') \mapsto (W_{t}(\omega'))_{{0 \le t \le T}}$ are independent under $\PP^{\otimes 2}$, we then understand $\WW_{s,t}^{\indep}$ as an iterated integral of two independent copies of the noise. While such a construction is elementary for a random $C^1$ path, the well-defined character of this integral needs to be proved for more general probability measures $\PP$.

\begin{Example}
\label{ex:brownian}
Let $W$ be an $\RR^m$-valued Brownian motion defined on $(\Omega,{\mathcal F},\PP)$. Denote by $W_t(\cdot)$ the equivalence class of $\Omega \ni \omega \mapsto W_{t}(\omega)$ in $\LL^q\big(\Omega,{\mathcal F},\PP;\RR^m\big)$, and extend $W_t$ on the product space $\big(\Omega^2,{\mathcal F}^{\otimes 2},\PP^{\otimes 2}\big)$, setting $W_t(\omega,\omega') := W_t(\omega)$. Define also on the product space the random variable $W_t'(\omega,\omega') := W_t(\omega')$. Then, $W$ and $W'$ are two independent $m$-dimensional Brownian motions under $\PP^{\otimes 2}$, and one can construct the time-indexed Stratonovich stochastic integral 
\begin{equation*}
\Omega^2 \ni (\omega,\omega') \mapsto \biggl(
\left\{\int_{s}^t (W_{r} - W_{s})  \otimes {\circ d}W_{r}' \right\} (\omega,\omega') \biggr)_{0 \leq s \leq t\leq T} \in {\mathcal C}\big({\mathcal S}_{2};\RR^{m \times m}\big).
\end{equation*} 
The stochastic integral  is uniquely defined up to an event of zero measure under $\PP^{\otimes 2}$. Up to an exceptional event (of 
$(\Omega^2,{\mathcal F}^{\otimes 2},\PP^{\otimes 2})$), we then let
\begin{equation*}
\WW^{\indep}_{s,t}(\omega,\omega') := \left(\int_{s}^t \big(W_{r} - W_{s}\big)  \otimes {\circ d}
W_{r}' \right) (\omega,\omega'), \quad 0 \leq s \leq t \leq T.
\end{equation*}
We can specify the definition of $W^{\indep}$ on the remaining exceptional event and then modify the definition of 
$W$ on a null event of $(\Omega,{\mathcal F},\PP)$ in such a way that Chen's relations \eqref{eq:chen} hold everywhere --see the end of the proof of Proposition \ref{prop:example:construction} below for a detailed proof of this fact--. The process $\bigl(\WW_{s,t}(\omega)\bigr)_{0 \leq s \leq t \leq T}$ is defined in a standard way  as a Stratonovich integral outside a set of null measure:
\begin{equation*}
\WW_{s,t}(\omega) := \left(\int_{s}^t (W_{r} - W_{s})  \otimes {\circ d}W_{r} \right) (\omega), \quad 0 \leq s \leq t\leq T.
\end{equation*}
\end{Example}

The principle underpinning the above example may be put in a more general framework which will be useful to prove continuity of the It\^o-Lyons solution map to the equation \eqref{EqRDE}. {We state it in the form of a proposition 
that provides a quite systematic way for constructing rough set-ups in practice.}
We advise the reader to come back to this proposition later on.

\begin{prop}
\label{prop:example:construction}
Let $(\Xi,{\mathcal G},{\mathbb Q})$ be a probability space, and $W^{1} := \big(W^1_{t}\big)_{0 \leq t \leq T}$ and $W^2 := \big(W^2_{t}\big)_{0 \le t \le T}$ be two independent and identically distributed $\RR^m$-valued processes defined on $\Xi$. Assume they have continuous trajectories and
$
\EE_{{\mathbb Q}}\left[\sup_{0 \leq t \leq T} \big\vert W^1_{t} \big\vert^q\right] < \infty$.

Let also $\big((W^{i,j}_{s,t})_{0 \leq s < t \leq T}\big)_{i,j=1,2}$ be four $\RR^m \otimes \RR^m \cong \RR^{m \times m}$-valued continuous paths such that 
$\EE_{{\mathbb Q}}\left[ \sup_{0 \leq s < t \leq T} \big\vert W^{i,j}_{s,t} \big\vert^q \right] < \infty$, 
for $ i,j = 1,2$, and $\big(W^1,W^{1,1}\big)$ is independent of $W^2$. Last, assume that, for a.e. $\xi \in \Xi$, the pair
\begin{equation*}
\biggl( \Bigl( \begin{array}{c}
W^1(\xi)
\\
W^2(\xi)
\end{array}
\Bigr),
\Bigl( \begin{array}{cc}
W^{1,1}(\xi) & W^{1,2}(\xi)
\\
W^{2,1}(\xi) & W^{2,2}(\xi)
\end{array}
\Bigr)
\biggr)
\end{equation*}
satisfies Chen's relation {in the sense that $W^{i,j}_{r,t}(\xi)=W^{i,j}_{r,s}(\xi)+W^{i,j}_{s,t}(\xi) + W^i_{r,s}(\xi) \otimes W^j_{s,t}(\xi)$ for any $i,j \in \{1,2\}$ and $0 \leq r \leq s \leq t \leq T$}. Set
$\Omega := \Xi \times [0,1]$
with $[0,1]$ equipped with its Borel $\sigma$-algebra ${\mathcal B}\big([0,1]\big)$, and denote by $\textrm{\rm Leb}$ the Lebesgue measure on $[0,1]$. Then we can find a triple of random variables $\big(W,{\mathbb W},{\mathbb W}^{\indep}\big)$, 
the first two components being defined on $\big(\Omega, {\mathcal F} \otimes {\mathcal B}([0,1]),{\mathbb Q} \otimes \textrm{\rm Leb}\big)$, the last component being constructed on the product space $\Omega^2$, and the whole family satisfying all the above requirements for a rough set-up, such that
\begin{equation*}
\PP \Bigl( \Bigl\{ (\xi,u) : \bigl(W,{\mathbb W}\bigr)(\xi,u) = \bigl(W^1,W^{1,1}\bigr)(\xi) \Bigr\} \Bigr) = 1, 
\end{equation*}
and, for $\PP$-a.e. $\omega=(\xi,u)$, the law of $W^{\indep}(\cdot,\omega)$ is the same as the conditional law of $W^{2,1}$ given 
$\big(W^1(\xi),W^2(\xi),W^{1,1}(\xi)\big)$. 
\end{prop}

{The reader may worry about the fact that, in the statement, we only appeal to $W^{1,1}$ and $W^{2,1}$, and not to $W^{2,2}$ and $W^{1,2}$. The reason is that, in our construction of the rough set-up, the processes 
${\mathbb W}^{\indep}(\omega,\cdot)$, ${\mathbb W}^{\indep}(\cdot,\omega)$ and ${\mathbb W}^{\indep}(\cdot,\cdot)$
are intrinsically connected. As made clear by the proof below, the relationships that hold true between 
${\mathbb W}^{\indep}(\omega,\cdot)$, ${\mathbb W}^{\indep}(\cdot,\omega)$ and ${\mathbb W}^{\indep}(\cdot,\cdot)$
must transfer to $(W^i)_{i=1,2}$ and $(W^{i,j})_{i,j=1,2}$. In short, everything works as if the pair $(W^2,W^{2,2})$ was
a mere independent copy of $(W^1,W^{1,1})$ and the conditional law of $W^{1,2}$ given $(W^2,W^1,W^{2,2})$
was the same as the conditional law of $W^{2,1}$ given $(W^1,W^2,W^{1,1})$, in which case the only needed ingredients are 
$W^1$, $W^{1,1}$, $W^{2}$ and $W^{2,1}$. The latter is consistent with the statement. }

\begin{proof}
Recall first from \cite{BlackwellDubins} the following form of Skorokhod representation theorem. {\it There exists a function 
$\Psi : [0,1] \times {\mathcal P}\big({\mathcal C}({\mathcal S}_{2}^T;\RR^m \otimes \RR^m)\big) \rightarrow {\mathcal C}\big({\mathcal S}_{2}^T;\RR^m \otimes \RR^m\big)$ 
such that
\begin{enumerate}
   \item[\textcolor{gray}{$\bullet$}] for every probability $\mu$ on ${\mathcal C}({\mathcal S}_{2}^T)$, equipped with its Borel $\sigma$-field, 
$[0,1] \ni u \mapsto \Psi(u,\mu)$ is a random variable with $\mu$ as distribution -- $[0,1]$ being equipped with Lebesgue measure,
   
   \item[\textcolor{gray}{$\bullet$}] the map $\Psi$ is measurable. 
\end{enumerate}}

Let now $\big(q(w^1,w^2,w^{1,1},\cdot)\big)_{w^1,w^2 \in {\mathcal C}([0,T];\RR^m);w^{1,1} \in {\mathcal C}({\mathcal S}_{2}^T;\RR^m \otimes \RR^m)}$ be a regular conditional probability of $W^{2,1}$ given $(W^1,W^2,W^{1,1})$. Define on $\Omega$ the random variables
\begin{equation*}
W(\xi,u) := W^1(\xi), \quad {\mathbb W}(\xi,u) := W^{1,1}(\xi),
\end{equation*}
and, on $\Omega^2$, 
\begin{equation*}
\begin{split}
&W'\big((\xi,u),(\xi',u')\big) := W^1(\xi'), 
\\
&{\mathbb W}^{\indep}\big((\xi,u),(\xi',u')\big) := \Psi \Bigl( u',q\bigl(W^1(\xi'),W^1(\xi),W^{1,1}(\xi'),\cdot\bigr) \Bigr). 
\end{split}
\end{equation*}
Since the law of $\big(W,W',{\mathbb W}\big)$ under ${\mathbb P}^{\otimes 2}$ is the same as the law of $\big(W^1,W^2,W^{1,1}\big)$ under ${\mathbb Q}$, we deduce that the law of $\big(W,W',{\mathbb W},{\mathbb W}^{\indeprev}\big)$ under ${\mathbb P}^{\otimes 2}$, 
with 
${\mathbb W}^{\indeprev}(\omega,\omega') :=
{\mathbb W}^{\indep}(\omega',\omega)$, 
 is the same as the law of $\big(W^1,W^2,W^{1,1},W^{2,1}\big)$ under ${\mathbb Q}$. In particular, with probability 1 under $\PP^{\otimes 2}$, for all $0 \leq r \leq s \leq t \leq T$,
\begin{equation*}
{\mathbb W}_{r,t}^{\indeprev}(\omega,\omega') = {\mathbb W}_{r,s}^{\indeprev}(\omega,\omega') + {\mathbb W}_{s,t}^{\indeprev}(\omega,\omega') + W_{r,s}(\omega') \otimes W_{s,t}(\omega), 
\end{equation*}
that is
\begin{equation*}
{\mathbb W}_{r,t}^{\indep}(\omega,\omega') = {\mathbb W}_{r,s}^{\indep}(\omega,\omega') + {\mathbb W}_{s,t}^{\indep}(\omega,\omega') + W_{r,s}(\omega) \otimes W_{s,t}(\omega'). 
\end{equation*}
Call now $A \in {\mathcal F}$ the set of those $\omega$'s in $\Omega$ for which the above relation fails for $\omega'$ in a set of positive probability measure under $\PP$. Clearly, $\PP(A)=0$. Define in a similar way $A'$ by exchanging the roles of $\omega$ and $\omega'$. For $\omega \in A \cup A'$, set $W(\omega) \equiv 0$; and whenever $\omega \in A$ or $\omega' \in A'$, set ${\mathbb W}^{\indep}(\omega,\omega') \equiv 0$. If 
$\omega \not \in A$, we have, by definition of $A$, the third identity in \eqref{eq:chen} -- pay attention that we use the fact that the identity is understood as an equality between classes of random variables that are $\PP$-a.e. equal.  If $\omega \in A$, it is also true since all the terms are zero. 
The second identity in \eqref{eq:chen} is checked in the same way. As for the first one, it holds on the complementary 
$B^{\complement}$ of a null event $B$. We then replace $A$ by $A \cup B$ and $A'$ by $A' \cup B$ in the previous lines and set 
$W(\cdot) \equiv 0$ and ${\mathbb W}(\cdot) \equiv 0$ on $A \cup A' \cup B$
{and ${\mathbb W}^{\indep}(\omega,\omega') =0$ when $\omega \in A \cup B$ or $\omega' \in A' \cup B$.}
\end{proof}

\subsection{Analytical conditions}

We use in this work the notion of $p$-variation to handle the regularity of the various trajectories in hand. The choice of the $p$-variation, instead of the simplest H\"older (semi-) norm, is dictated by the arguments we use below to prove well-posedness of  \eqref{eq:1:2}. We shall indeed invoke some integrability results from \cite{CassLittererLyons}, which are explicitly based upon the notion of $p$-variation and are not proved in H\"older (semi-) norm. Several types of $p$-variations are needed to handle differently the finite and infinite dimensional components of a rough set-up ${\boldsymbol W}$. 
\textcolor{black}{Throughout, $p$ is taken} in the interval $[2,3)$. For a continuous function ${\mathbb G}$ from the simplex ${\mathcal S}_{2}^T$ into some $\RR^\ell$, we set, for any $p'\geq 1$, 
\begin{equation*}
\begin{split}
&\| {\mathbb G} \|_{[0,T],p'-\textrm{\rm var}}^{p'} := \sup_{0 = t_{0}<t_{1} \cdots < t_{n}=T } \, \sum_{i=1}^n \vert {\mathbb G}_{t_{i-1},t_{i}}\vert^{p'},
\end{split}
\end{equation*}
and define for any function $g$ from $[0,T]$ into $\RR^\ell$, 
$
\| g \|_{[0,T],p-\textrm{\rm var}}^{p} := \| {\mathbb G} \|_{[0,T],p-\textrm{\rm var}}^{p}$
where ${\mathbb G}_{s,t} := g_t - g_s$. Similarly, for a random variable ${\mathbb G}(\cdot)$ on $\Omega$ with values in ${\mathcal C}({\mathcal S}_{2}^T;\RR^{\ell})$, and $p'\geq 1$, we define its $p'$-variation in $\LL^q$ as 
\begin{equation}
\label{eq:q:p-var}
\begin{split}
&\langle {\mathbb G}(\cdot) \rangle_{q; [0,T],p'-\textrm{\rm var}}^{p'} := \sup_{0 = t_{0}<t_{1} \cdots < t_{n}=T } \, \sum_{i=1}^n \big\langle  {\mathbb G}_{t_{i-1},t_{i}}(\cdot) \big\rangle_{q}^{p'},
\end{split}
\end{equation}
and define for a random variable $G(\cdot)$ on $\Omega$, with values in ${\mathcal C}([0,T];\RR^{\ell})$,
$$
\big\langle G(\cdot) \big\rangle_{q ; [0,T],p-\textrm{\rm var}}^{p} := \big\langle {\mathbb G}(\cdot) \big\rangle_{q; [0,T],p-\textrm{\rm var}}^{p},
$$
as the $p$-variation semi-norm in $\LL^q$ of ${\mathcal S}_{2}^T \ni (s,t) \mapsto {\mathbb G}_{s,t}(\cdot) = G_{t}(\cdot) - G_{s}(\cdot)$. Last, for a random variable ${\mathbb G}(\cdot,\cdot)$ from $(\Omega^2,{\mathcal F}^{\otimes 2})$ into ${\mathcal C}({\mathcal S}_{2}^T;\RR^{\ell})$, we set 
\begin{equation}
\label{eq:q:q:p-var}
\begin{split}
&\llangle {\mathbb G}(\cdot,\cdot) \rrangle_{q ; [0,T],p/2-\textrm{\rm var}}^{p/2} := \sup_{0 = t_{0}<t_{1} \cdots < t_{n}=T } \sum_{i=1}^n \left\llangle {\mathbb G}_{t_{i-1},t_{i}}(\cdot,\cdot) \right\rrangle_{q}^{p/2}.
\end{split}
\end{equation}
Given these definitions, we require from the rough set-up ${\boldsymbol W}$ that 

\begin{itemize}
   \item For any $\omega \in \Omega$, the path $W(\omega)$ is in the space ${\mathcal C}([0,T];\RR^m)$, and the map $W : \Omega \ni \omega \mapsto W(\omega) \in {\mathcal C}([0,T];\RR^m)$ is Borel-measurable and $q$-integrable (meaning that
   the supremum of $W$ over $[0,T]$ is $q$-integrable). 
   
   \item For any $\omega \in \Omega$, the two-index path $\WW(\omega)$ is in ${\mathcal C}({\mathcal S}_{2}^T;\RR^{m \times m})$, and the map $\WW : \Omega \ni \omega \mapsto \WW(\omega) \in {\mathcal C}({\mathcal S}_{2}^T;\RR^{m \times m})$ is Borel-measurable and $q$-integrable (i.e., the supremum of ${\mathbb W}$ over ${\mathcal S}_{2}^T$ has a finite $q$-moment).   

   \item For any $(\omega,\omega') \in \Omega^2$, the two-index path $\WW^{\indep}(\omega,\omega')$ is an element of   ${\mathcal C}({\mathcal S}_{2}^T;\RR^{m \times m})$, and the map $\WW^{\indep} : \Omega^2 \ni (\omega,\omega') \mapsto \WW^{\indep}(\omega,\omega') \in {\mathcal C}({\mathcal S}_{2}^T;\RR^{m \times m})$ is Borel-measurable and $q$-integrable. In particular, for a.e. $\omega \in \Omega$, the two-index path $\WW^{\indep}(\omega,\cdot)$ belongs to ${\mathcal C}\big({\mathcal S}_{2}^T; \LL^q(\Omega,{\mathcal F},\PP; \RR^{m \times m})\big)$, and the map $\Omega \ni \omega \mapsto \WW^{\indep}(\omega,\cdot)$ is Borel-measurable and $q$-integrable, and similarly for $\WW^{\indep}(\cdot,\omega)$; as before, we assume the latter to be true for every $\omega \in \Omega$. Also, the two-index deterministic path $\WW^{\indep}(\cdot,\cdot)$ is a continuous mapping from ${\mathcal S}_{2}^T$ into $\LL^q\big(\Omega^2,{\mathcal F}^{\otimes 2},\PP^{\otimes 2};\RR^{m \times m}\big)$.
\end{itemize}

\noindent We then set, for all $0\leq s\leq t \leq T$ and $\omega\in\Omega$,
\begin{equation}
\label{eq:v}
\begin{split}
&v(s,t,\omega) := \big\| W(\omega) \big\|_{[s,t],p-\textrm{\rm var}}^p + \big\langle W(\cdot) \big\rangle_{q ; [s,t],p-\textrm{\rm var}}^p   
 + \big\| \WW(\omega) \big\|_{[s,t],p/2-\textrm{\rm var}}^{p/2} \\
&\hspace{15pt} + \big\langle \WW^{\indep}(\omega,\cdot) \big\rangle_{q ; [s,t],p/2-\textrm{\rm var}}^{p/2}    + \big\langle \WW^{\indep}(\cdot,\omega) \big\rangle_{q ; [s,t],p/2-\textrm{\rm var}}^{p/2} + \big\llangle \WW^{\indep}(\cdot,\cdot) \big\rrangle_{q ; [s,t],p/2-\textrm{\rm var}}^{p/2},
\end{split}
\end{equation}
and we assume that, for any $T>0$ and $\omega \in \Omega$, $v(0,T,\omega)$ is finite. Then, we have the super-additivity property: For any $0 \leq r \leq s \leq t \leq T$, and $\omega\in\Omega$,  
$v(r,t,\omega) \geq v(r,s,\omega) + v(s,t,\omega)$.

Observe also from \cite[Proposition 5.8]{FrizVictoirBook} that $\omega \mapsto (v(s,t,\omega))_{(s,t)\in {\mathcal S}_{2}^T}$ is a random variable with values in ${\mathcal C}({\mathcal S}_{2}^T;\RR_{+})$. Throughout the analysis, we assume 
$\langle v(0,T, \cdot) \rangle_{q} < \infty$,
for any rough set-up considered on the interval $[0,T]$. By Lebesgue's dominated convergence theorem, the function 
$
{\mathcal S}_{2}^T \ni (s,t) \mapsto \langle v(s,t,\cdot) \rangle_q
$ 
is continuous. {\it We shall actually assume that it is of bounded variation on $[0,T]$, i.e.}, 
\begin{equation*}
\langle v(\cdot) \rangle_{q;[s,t],1-\textrm{\rm var}} := \sup_{0 \leq t_{1} < \cdots < t_{n} \leq T} \sum_{i=1}^{n} \langle v(t_{i-1},t_{i},\cdot) \rangle_{q} < \infty. 
\end{equation*}
{Below, we call a control any family
of random variables $(\omega \mapsto w(s,t,\omega))_{(s,t) \in {\mathcal S}_{2}^T}$ 
 that 
is jointly continuous in $(s,t)$ and that satisfies, 
\begin{equation}
\label{eq:w:s:t:omega:ineq}
w (s,t,\omega) \geq v(s,t,\omega) + \langle v(\cdot) \rangle_{q;[s,t],1-\textrm{\rm var}},
\end{equation}
together with
\begin{equation}
\begin{split}
\label{eq:useful:inequality:wT}
&\langle w(s,t,\cdot) \rangle_{q} \leq 2 \, w(s,t,\omega),
\\
&w(r,t,\omega) \geq w(r,s,\omega) + w(s,t,\omega), \quad r \leq s \leq t.
\end{split}
\end{equation}
Of course, 
a typical choice to get 
\eqref{eq:w:s:t:omega:ineq}
and 
\eqref{eq:useful:inequality:wT} is to choose
\begin{equation}
\label{eq:w:s:t:omega}
w (s,t,\omega) := v(s,t,\omega) + \langle v(\cdot) \rangle_{q;[s,t],1-\textrm{\rm var}}.
\end{equation}
}

%Below, we often check that ${\mathcal S}_{2}^T \ni (s,t) \mapsto \langle v(s,t,\cdot) \rangle_{q}$ is 
%of bounded variation 
% by proving that it is Lipschitz continuous. 

\begin{Example}
\label{example:3}
\textbf{\textbf{ Gaussian processes}  -- } Start from an $\RR^m$-valued tuple  $W := (W^1,\cdots,W^m)$ of independent and centred continuous Gaussian processes, defined on some finite time interval $[0,T]$, such that the two-dimensional covariance of $W$ is of finite $\rho$-variation for some $\rho \in [1,3/2)$ and there exists a constant $K$ such that, for any subinterval $[s,t]\subset [0,T]$ and any $k=1,\cdots,m$, one has\footnote{{In fact, \eqref{eq:FV:rho:variation:covariance} implies that the two-dimensional covariance of $W$ is of finite $\rho$-variation.}}
\begin{equation}
\label{eq:FV:rho:variation:covariance}
\begin{split}
\sup \, \sum_{i,j} \, \Bigl\vert {\mathbb E} \Bigl[ \bigl(W^k_{t_{i+1}} - W^{k}_{t_{i}} \bigr) \bigl( W^k_{s_{j+1}} - W^k_{s_{j}} \bigr) \Bigr] \Bigr\vert^\rho \leq K \vert t-s\vert,
\end{split}
\end{equation} 
where the supremum is taken over all dissections $(t_{i})_{i}$ and $(s_{j})_{j}$ of the interval $[s,t]$. See Definition 5.50 in \cite{FrizVictoirBook}. This setting includes the case of fractional Brownian motion, with Hurst index greater than $1/4$. Without any loss of generality, we may assume that the process $W$ is constructed on the canonical space $(\Omega,{\mathcal F},\PP)$, where $\Omega =
{\mathcal W}$, with ${\mathcal W}:=
 {\mathcal C}([0,T];{\mathbb R}^m)$, ${\mathcal F}$ is the Borel $\sigma$-field, and $W$ is the coordinate process. We then denote by $(\Omega={\mathcal W},{\mathcal H},\PP)$ the abstract Wiener space associated with $W$, 
see \cite[Appendix D]{FrizVictoirBook}, 
 where ${\mathcal H}$ is a Hilbert space, which is automatically embedded in the subspace ${\mathcal C}^{\varrho-\textrm{\rm var}}\big([0,T];\RR^m\big)$ of ${\mathcal C}\big([0,T];\RR^m\big)$ consisting of \textcolor{black}{continuous} paths of finite $\varrho$-variation. By Theorem 15.33 in \cite{FrizVictoirBook}, we know that, \textcolor{black}{for $\omega$ outside an exceptional event}, the trajectory $W(\omega)$ may be lifted into a rough path $(W(\omega),{\mathbb W}(\omega))$ with finite $p$-variation for any $p \in (2\rho,3)$, \textcolor{black}{namely $W(\omega)$ has a finite $p$-variation and 
${\mathbb W}(\omega)$ has a finite $p/2$-variation}. We lift arbitrarily \textcolor{black}{(say onto the zero path)} on the null set where the lift is not automatic. The pair $(W,{\mathbb W})$, indexed by $\omega$ is part of our rough set-up.  \textcolor{black}{In this regard}, we recall from Theorem 15.33 in \cite{FrizVictoirBook} that the random variables
\begin{equation}
\label{eq:Holder}
\Omega \ni \omega \mapsto \big\| W(\omega) \big\|_{[0,T],p-\textrm{\rm var}}, 
\quad \Omega \ni \omega \mapsto \big\| {\mathbb W}(\omega) \big\|_{[0,T],
p/2-\textrm{\rm var}},
\end{equation}
have respectively Gaussian and exponential tails, and thus have a finite $\LL^q$-moment.

One can proceed as follows to construct the other elements 
$
\big({\mathbb W}^{\indep}(\omega,\cdot)\big)_{\omega \in \Omega}$, $\big({\mathbb W}^{\indep}(\cdot,\omega)\big)_{\omega \in \Omega}$, ${\mathbb W}^{\indep}(\cdot,\cdot)$ 
of our rough set-up. We extend the space into $(\Omega^2,{\mathcal F}^{\otimes 2},\PP^{\otimes 2})$, with $\Omega$ embedded in the first component say, and denote by $(W,W')$ the canonical coordinate process on $\Omega^2$. They are independent and have independent Gaussian components under $\PP^2$. The associated abstract Wiener space is nothing but $\big(\Omega^2,{\mathcal H} \oplus {\mathcal H},\PP^{\otimes 2}\big)$. The process $(W,W')$ also satisfies Theorem 15.33 in \cite{FrizVictoirBook} for the same exponent $\rho$ as before, so, we can enhance $(W,W')$ into a Gaussian rough path, with some arbitrary extension outside the $\PP^{\otimes 2}$-exceptional event on which we cannot construct the enhancement. To ease the notations, we merely write $W(\omega)$ for $W(\omega,\omega')$ as it  is independent of $\omega$; similarly, we write $W'(\omega')$ for $W'(\omega,\omega')$. Proceeding as before, we call $\bigl({\mathbb W}^{\indep}(\omega,\omega')\bigr)_{\omega,\omega' \in \Omega}$, the upper off-diagonal $m \times m$ block in the decomposition of the second-order tensor of the rough path in the form of a $(2m) \times (2m)$-matrix with four blocks of size $m \times m$. Chen's relationship then yields, 
 for $\PP^{\otimes 2}$-a.e. $(\omega,\omega')$,
\begin{equation*}
{\mathbb  W}^{\indep}_{r,t}(\omega,\omega') = {\mathbb  W}^{\indep}_{r,s}(\omega,\omega') + {\mathbb  W}^{\indep}_{s,t}(\omega,\omega') + W_{r,s}(\omega) \otimes W_{s,t}(\omega'),
 \end{equation*}
for all $r \leq s \leq t$. As before, we know from Theorem 15.33  in \cite{FrizVictoirBook} that the $1/p$-H\"older semi-norm of ${W}(\omega)$, which we denote by $\|\textcolor{black}{W(\omega)}\big\|_{[0,T],(1/p)-\textrm{\rm H\"ol}}$, and the $2/p$-H\"older semi-norm of ${\mathbb W}^{\indep}(\omega,\omega')$, which we denote by $\big\|{\mathbb W}^{\indep}(\omega,\omega')\big\|_{[0,T],(2/p)-\textrm{\rm H\"ol}}$, have respectively Gaussian and exponential tails, when considered as random variables on the spaces $(\Omega,{\mathcal F},{\mathbb P})$ and $\big(\Omega^2,{\mathcal F}^{\otimes 2},{\mathbb P}^{\otimes 2}\big)$. In particular, for a.e. $\omega \in \Omega$, we may consider $\bigl(\WW_{\textcolor{black}{s,t}}^{\indep}(\omega,\cdot)\bigr)_{\textcolor{black}(s,t) \in {\mathcal S}_{2}^T}$ as a continuous process with values in $\LL^q$. Moreover, 
\begin{equation*}
\begin{split}
&\textcolor{black}{\bigl\langle} {\mathbb W}^{\indep}(\omega,\cdot) \textcolor{black}{ \bigr\rangle_{q; [0,T],p/2-\textrm{\rm var}}^{p/2}}
\\
&= \sup_{0=t_{0}<t_{1}<\cdots < t_{n}=T} \sum_{i=1}^n \bigl\langle {\mathbb W}^{\indep}_{t_{i-1},t_{i}}(\omega,\cdot) \bigr\rangle_{q}^{p/2}   \\
&\leq T \, \Bigl\langle \|{\mathbb W}^{\indep}(\omega,\cdot)\|_{[0,T],(2/p)-\textrm{\rm H\"ol}} \Bigr\rangle_{q}^{p/2}   \leq T \, \Bigl\langle\|{\mathbb W}^{\indep}(\omega,\cdot)\|_{[0,T],(2/p)-\textrm{\rm H\"ol}}^{p/2}  \Bigr\rangle_{q}, 
\end{split}
\end{equation*} 
which shows that the left-hand side has finite moments of any order. Arguing in the same way for $\big(\WW^{\indep}(\cdot,\omega)\big)_{\omega \in \Omega}$ and for $\WW^{\indep}$, we deduce that $v$ \textcolor{black}{in \eqref{eq:v}} is almost surely finite and $q$-integrable. Obviously, by replacing $[0,T]$ by $[s,t] \subset [0,T]$, we obtain that the $q$-moment of $v$ is Lipschitz (and thus of finite 1-variation), as required.  

All these properties (that hold true on a full event) may be extended to the full set $\Omega^2$ by arguing as in the proof 
of Proposition 
\ref{prop:example:construction}.
\end{Example}

\subsection{Local accumulation}
To use that rough set-up in our machinery,  we need a version of an integrability result of \cite{CassLittererLyons} whose proof is postponed to Appendix \ref{SectionIntegrability}. Given {a nondecreasing}\footnote{In the sense that 
$\varpi(a,b) \geq \varpi(a',b')$ if $(a',b') \subset (a,b)$.} continuous positive valued function $\varpi$ on ${\mathcal S}_{2}$, a parameter $s \geq 0$ and a threshold $\alpha >0$, we define inductively a sequence of times  %and 
\begin{equation}
\label{eq:stopping:times}
\tau_{0}(s,\alpha) := s, \quad 
\textrm{\rm and} \quad \tau_{n+1}^{\varpi}(s,\alpha) := \inf\Bigl\{ u \geq \tau_{n}^{\varpi}(s,\alpha) \, : \, \varpi\bigl(\tau_{n}^{\varpi}(s,\alpha),u\bigr) \geq \alpha \Bigr\}, 
\end{equation}
with the understanding that $\inf \emptyset= + \infty$. For $t \geq s$, set
\begin{equation}
\label{eq:N:s:t:omega}
N_{\varpi}\bigl([s,t],\alpha\bigr) := \sup \Bigl\{ n \in {\mathbb N} \ : \ \tau_{n}^{\varpi}(s,\alpha) \leq t \Bigr\}. 
\end{equation}

Below, we call $N_{\varpi}$ the \textbf{\textbf{local accumulation of}} $\varpi$ (of size $\alpha$ if we specify the value of the threshold): {$N_{\varpi}([s,t],\alpha)$ is the largest number of disjoint open sub-intervals $(a,b)$ of  
$[s,t]$ on which $\varpi(a,b)$ is greater than or equal to $\alpha$}. When $\varpi(s,t) = w(s,t,\omega)^{1/p}$ with $w$ a {control satisfying 
\eqref{eq:w:s:t:omega:ineq}
and 
\eqref{eq:useful:inequality:wT}}
and when the framework makes it clear, we just write $N([s,t],\omega,\alpha)$ for $N_{\varpi}([s,t],\alpha)$. 
Similarly, we also write $\tau_{n}(s,\omega,\alpha)$ for $\tau_{n}^{\varpi}(s,\alpha)$ when $\varpi(s,t) = w(s,t,\omega)^{1/p}$. We will also use the notation
$\tau_{n}^{\varpi}(s,t,\alpha) := \tau_{n}^{\varpi}(s,\alpha) \wedge t$.

The proof of the following statement is given in Appendix \ref{SubsectionAppendixCLL}. Recall that a positive random variable $A$ has a Weibull tail with shape parameter $2/\varrho$ if $A^{1/\rho}$ has a Gaussian tail.

\begin{thm}
\label{thm:cass:litterer:lyons}
Let $W$ be a continuous centred Gaussian process, defined over some finite interval $[0,T]$. Assume it has independent components, and denote by 
$({\mathcal W},{\mathcal H},\PP)$ its associated Wiener space. Suppose that the covariance function is of finite two dimensional $\varrho$-variation for some $\varrho \in [1,3/2)$ and satisfies the Lipschitz estimate \eqref{eq:FV:rho:variation:covariance}. Then, for $p \in (2 \varrho,3)$ and $\alpha >0$, the process $N(\cdot, \alpha):=(N([0,T], \omega\textcolor{black}{,}\alpha))_{\omega \in \Omega}$ associated to the rough-set up built from $W$, {with $w$ being defined as in \eqref{eq:w:s:t:omega},} has a Weibull tail with shape parameter $2/\varrho$. 
\end{thm}

{As a corollary, we deduce that the} estimate on $N$ required in Theorem \ref{ThmMain} is satisfied in the above setting. For the same value of $p$, the quantity $w(0,T)$ in \eqref{eq:w:s:t:omega} also satisfies the integrability statement of Theorem \ref{ThmMain}; the latter then applies in the above Gaussian setting. Building on the work \cite{CassOgrodnik} on Markovian rough paths one can prove a similar result as Theorem \ref{thm:cass:litterer:lyons} for Markovian rough paths.

%-------------------------------------------------------------%
\section{Controlled Trajectories and Rough Integral}
\label{SectionControlledTrajectories}
%-------------------------------------------------------------%

Following \cite{Gubinelli}, we now define a controlled path and the corresponding rough integral.
{Throughout the section, we are given a control $w$ satisfying 
\eqref{eq:w:s:t:omega:ineq}
and 
\eqref{eq:useful:inequality:wT}. 
%We refer once again to 
%\eqref{eq:w:s:t:omega} for a typical choice. 
}

%%-------------------------------------%%
\subsection{Controlled Trajectories}
\label{SubsectionControlledTraj}
%%-------------------------------------%%

We first define the notion of controlled trajectory for a given outcome $\omega\in \Omega$.

\begin{defn}
\label{definition:omega:controlled:trajectory}
An $\omega$-dependent continuous $\RR^d$-valued path $(X_{t}(\omega))_{0 \le t \le T}$ is called an \textbf{\textbf{$\omega$-controlled path}} on $[0,T]$ if its increments can be decomposed as 
\begin{equation}
\label{eq:omega:controlled:eq}
X_{s,t}(\omega) = \delta_{x} X_{s}(\omega) W_{s,t}(\omega) + \EE \bigl[ \delta_{\mu} X_{s}(\omega,\cdot) W_{s,t}(\cdot) \bigr] + R^X_{s,t}(\omega),
\end{equation}
where 
$\big(\delta_{x} X_{t}(\omega)\big)_{0 \leq t \leq T}$ belongs to the space $\cC\big([0,T];\RR^{d \times m}\big)$, 
$\big(\delta_{\mu} X_{t}(\omega,\cdot)\big)_{0 \leq t \leq T}$ to the space ${\mathcal C}\big([0,T];\LL^{4/3}(\Omega,{\mathcal F},\PP;\RR^{d \times m})\big)$,
$\big(R_{s,t}^X(\omega)\big)_{s,t \in {\mathcal S}_{2}^T}$ is in the space $\cC({\mathcal S}_{2}^T;\RR^d)$,
and 
\begin{equation*}
\begin{split}
&\vvvert X(\omega) \vvvert_{\star,[0,T],w,p} := \vert X_{0}(\omega) \vert + \big\vert \delta_{x} X_{0}(\omega) \big\vert + \big\langle  \delta_{\mu} X_{0}(\omega,\cdot) \big\rangle_{4/3} + \vvvert X(\omega) \vvvert_{[0,T],w,p} < \infty,
\end{split}
\end{equation*}
where $\vvvert X(\omega) \vvvert_{[0,T],w,p} := \| X(\omega) \|_{[0,T],w,p} + \| \delta_{x} X(\omega) \|_{[0,T],w,p} + \big\langle  \delta_{\mu} X(\omega,\cdot)  \big\rangle_{[0,T],w,p,4/3}   + \| R^X(\omega) \|_{[0,T],w,p/2}$, with
\begin{equation*}
\begin{split}
&\| X(\omega) \|_{[0,T],w,p} := \sup_{[s,t] \subset [0,T]} \frac{\big\vert   X_{{s,t}}(\omega) \big\vert}{w(s,t,\omega)^{1/p}},   \
\| \delta_{x} X(\omega) \|_{[0,T],w,p} := \sup_{[s,t] \subset [0,T]} \frac{\big\vert \delta_{x} X_{{s,t}}(\omega) \big\vert}{w(s,t,\omega)^{1/p}},   
\\
&\big\langle \delta_{\mu} X(\omega,\cdot) \big\rangle_{[0,T],w,p,4/3} := \sup_{[s,t] \subset [0,T]} \frac{\big\langle  \delta_{\mu}X_{{s,t}}(\omega,\cdot)  \big\rangle_{4/3}}{w(s,t,\omega)^{1/p}},
\\
&\| R^X(\omega) \|_{[0,T],w,p/2} := \sup_{[s,t] \subset [0,T]} \frac{\big\vert R_{{s,t}}^X(\omega) \big\vert}{w(s,t,\omega)^{2/p}}.
\end{split}
\end{equation*}
We call \textcolor{black}{$\delta_x X(\omega)$} and \textcolor{black}{$\delta_\mu X(\omega,\cdot)$} in  \eqref{eq:omega:controlled:eq} the \textbf{\textbf{derivatives of the controlled path}} \textcolor{black}{$X(\omega)$}.
\end{defn}

The value $4/3$ is somewhat arbitrary here. Our analysis could be managed  with another exponent strictly greater than 1, but this would require higher values for the exponent $q$ than that one we use in the definition of the rough set-up -- recall $q \geq 8$. It seems that the value $4/3$ is pretty convenient, as $4/3$ is the conjugate exponent of $4$. It follows \textcolor{black}{from} the fact that $\vvvert X(\omega) \vvvert_{\star,[0,T],p}$ is finite that an $\omega$-controlled path is controlled in the usual sense by the first level $\big(W_t(\omega),W_t(\cdot)\big)_{0\leq t\leq T}$ of our rough set-up, \textcolor{black}{provided the latter is considered as taking values in an infinite dimensional space}, see Section \ref{SubsectionIntegral} below.

We now define the notion of random controlled trajectory, which consists of a collection of $\omega$-controlled trajectories indexed by the elements of $\Omega$.

\begin{defn}
\label{definition:random:controlled:trajectory}
A family of $\omega$-controlled paths $(X(\omega))_{\omega \in \Omega}$ such that the maps
\begin{equation*}
\begin{split}
&\Omega \ni \omega \mapsto \big(X_{t}(\omega)\big)_{0 \leq t \leq T}
\in \cC\big([0,T];\RR^d\big), \ \Omega \ni \omega \mapsto \big(\delta_{x} X_{t}(\omega)\big)_{0 \leq t \leq T}
\in \cC\big([0,T];\RR^{d \times m}\big)   \\
&\Omega \ni \omega \mapsto \big(\delta_\mu X_{t}(\omega)\big)_{0 \leq t \leq T}
\in \cC\big([0,T];\LL^{4/3}(\Omega,{\mathcal F},\PP;\RR^{d \times m})\big), 
\\
&\Omega \ni \omega \mapsto \big(R_{s,t}^X(\omega)\big)_{\textcolor{black}{(s,t)} \in {\mathcal S}_{2}^T} \in 
\cC\big({\mathcal S}_{2}^T;\RR^d\big),
\end{split}
\end{equation*}
are measurable and \textcolor{black}{satisfy} 
\begin{equation}
\label{eq:Lq:vvvert}
\big\langle X_{0}(\cdot) \big\rangle_{2} + \bigl\langle \vvvert X(\cdot) \vvvert_{[0,T],w,p}\bigr\rangle_{8} < \infty
\end{equation}
is called a \textbf{\textbf{random controlled path}} on $[0,T]$.
\end{defn}

Note from \eqref{eq:useful:inequality:wT} the following elementary fact, whose proof is left to the reader.

\begin{lem}
\label{lem:1:1}
Let $\bigl((X_{t}(\omega))\bigr)_{0 \leq t \leq T})_{\omega \in \Omega}$ be a random controlled path on a time interval $[0,T]$.  Then, for any $0 \leq s<t\leq T$, we have 
\begin{equation*}
\begin{split}
\big\langle X_{s,t}(\cdot) \big\rangle_{2} &\leq \Big\langle \vvvert X(\cdot) \vvvert_{[0,T],w,p}^2 \, w(s,t,\cdot)^{2/p} \Big\rangle^{1/2}   \\
&\leq \big\langle \vvvert X(\cdot) \vvvert_{[0,T],w,p} \big\rangle_{4} \, \big\langle w(s,t,\cdot) \big\rangle_{4}^{1/p} \leq 2 \, \big\langle \vvvert X(\cdot)\vvvert_{[0,T],w,p} \big\rangle_{4} \, w(s,t,\omega)^{1/p}. 
\end{split}
\end{equation*}
Similarly,
\begin{equation*}
\begin{split}
\big\langle X_{s,t}(\cdot)\big\rangle_{4} &\leq \big\langle \vvvert X(\cdot) \vvvert_{[0,T],w,p} \big\rangle_{8} \, \big\langle w(s,t,\cdot) \big\rangle_{8}^{1/p} \leq 2 \,\big\langle \vvvert X(\cdot) \vvvert_{[0,T],w,p} \big\rangle_{8} \, w(s,t,\omega)^{1/p}. 
\end{split}
\end{equation*}
\end{lem}

A straightforward consequence of Lemma \ref{lem:1:1} is that a random controlled trajectory induces a continuous path from $[0,T]$ to $L^2(\Omega,{\mathcal F},\PP;\RR^d)$.

%%--------------------------%%
\subsection{Rough Integral}
\label{SubsectionIntegral}
%%--------------------------%%

Set $U:=\RR^m \times \LL^q(\Omega,{\mathcal F},\PP;\RR^m)$ and note that $U\otimes U$ can be canonically identified with
\begin{equation*}
\begin{split}
\big(\RR^m \otimes \RR^m\big) \oplus \Big(\RR^m \otimes \LL^q(\Omega,{\mathcal F},\PP;\RR^m)\Big)   &\oplus \Big(\LL^q(\Omega,{\mathcal F},\PP;\RR^m) \otimes \RR^m\Big) 
\\
&\oplus \Big(\LL^q(\Omega,{\mathcal F},\PP;\RR^m)^{\otimes 2}\Big).
\end{split}
\end{equation*}
We take as a starting point of our analysis the fact that ${\boldsymbol W}(\omega)$ may be considered as a rough path with values in $U \oplus U^{ \otimes 2}$, for any given $\omega$. Indeed the first level ${\boldsymbol W}^{(1)}(\omega) := \big(W_{t}(\omega),W_{t}(\cdot)\big)_{t \geq 0}$ of ${\boldsymbol W}(\omega)$ is a continuous path with values in $U$ and its second level
\begin{equation*}
{\boldsymbol W}^{(2)}(\omega)
:= \left( \begin{array}{cc} \WW_{0,t}(\omega) &\WW_{0,t}^{\indep}(\omega,\cdot)   \\
\WW_{0,t}^{\indep}(\cdot,\omega) &\WW_{0,t}^{\indep}(\cdot,\cdot) \end{array} \right)_{t \geq 0} 
\end{equation*}
is a continuous path with values in $U \otimes U$, {with} ${\mathbb W}_{0,t}(\omega)$ seen as an element of ${\mathbb R}^m \otimes {\mathbb R}^m$, ${\mathbb W}_{0,t}^{\indep}(\omega,\cdot)$ as an element of ${\mathbb R}^m \otimes \LL^q(\Omega,{\mathcal F},\PP;\RR^m)$, ${\mathbb W}_{0,t}^{\indep}(\cdot,\omega)$ as an element of $\LL^q(\Omega,\mathcal F,\PP;\RR^m) \otimes {\mathbb R}^m $, and ${\mathbb W}_{0,t}^{\indep}(\cdot,\cdot)$ as an element of $\LL^q(\Omega,\mathcal F,\PP;\RR^m) \otimes \LL^q(\Omega,\mathcal F,\PP;\RR^m)$. Condition \eqref{eq:chen} then reads as Chen's relation for ${\boldsymbol W}(\omega)$.

We can then use sewing lemma \cite{FeyeldelaPradelle}, in the form given in \cite{CoutinLejay,CoutinLejay2}, to construct the rough integral of an $\omega$-controlled path and a Banach-valued rough set-up.

\begin{thm}
\label{thm:integral}
\textcolor{black}{There exists a universal constant $c_0$ and}, for any $\omega \in \Omega$, there exists a continuous linear map 
\begin{equation*}
\bigl(X_{t}(\omega)\bigr)_{0 \leq t \leq T} \mapsto \biggl( \int_{s}^t X_{s,u}(\omega) \otimes d {\boldsymbol W}_{u}(\omega) \biggr)_{\textcolor{black}{(s,t)} \in {\mathcal S}_{2}^T}
\end{equation*}
from the space of $\omega$-controlled trajectories equipped with the norm $\vvvert \cdot \vvvert_{\star,[0,T],p}$, onto the space of continuous functions from ${\mathcal S}_{2}^T$ into $\RR^d \otimes \RR^m$ with finite norm $\| \cdot \|_{[0,T],w,p/2}$, with $w$ {in the latter norm} being evaluated along the realization $\omega$, that satisfies for any $0\leq r\leq s\leq t\leq T$ the identity
{\color{black}
\begin{equation*}
\begin{split}
&\int_{r}^t X_{r,u}(\omega) \otimes d {\boldsymbol W}_{u}(\omega) 
\\
&\hspace{15pt}= \int_{r}^s X_{r,u}(\omega) \otimes d {\boldsymbol W}_{u}(\omega) + \int_{s}^t X_{s,u}(\omega) \otimes d {\boldsymbol W}_{u}(\omega) + X_{r,s}(\omega)
\otimes W_{s,t}(\omega),
\end{split}
\end{equation*}
together with the estimate}
\begin{equation}
\label{eq:remainder:integral}
\begin{split}
&\biggl\vert \int_{s}^t X_{s,u}(\omega) \otimes d {\boldsymbol W}_{u}(\omega) - \Big\{\delta_{x} X_{s}(\omega) {\mathbb W}_{s,t}(\omega) + {\mathbb E} \bigl[ \delta_{\mu} X_{s}(\omega,\cdot) {\mathbb W}_{s,t}^{\indep}(\cdot,\omega) \bigr] \Big\}\biggr\vert   \\
&\leq c_0\, \vvvert X(\omega) \vvvert_{\textcolor{black}{[0,T],w,p}}\, w(s,t,\omega)^{3/p}. 
\end{split}
\end{equation}
\end{thm}

Here, $\delta_{x} X_{s}(\omega) \, {\mathbb W}_{s,t}(\omega)$ is the product of two $d \times m$ and $m \times m$ matrices, so it gives back a $d \times m$ matrix, with components 
$
\bigl( \delta_{x} X_{s}(\omega) {\mathbb W}_{s,t}(\omega) \bigr)_{i,j} = \sum_{k=1}^m \bigl( \delta_{x} X_{s}^{i}(\omega) \bigr)_{k} \bigl( {\mathbb W}_{s,t}(\omega) \bigr)_{k,j}$, 
for $i \in \{1,\cdots,d\}$ and $j \in \{1,\cdots,m\}$. We stress that the notation 
$
{\mathbb E}\big[\delta_{\mu} X_{s}(\omega,\cdot) {\mathbb W}_{s,t}^{\indep}(\cdot,\omega)\big],
$
which reads as the expectation of a matrix of size $d \times m$, can be also interpreted as a contraction product between an element of $\RR^d \otimes \LL^{\textcolor{black}{4/3}}(\Omega,{\mathcal F},\PP;\RR^m)$ and an element of $\LL^q(\Omega,{\mathcal F},\PP;\RR^m) \otimes \RR^m$. This remark  is  important for the proof below.

\begin{proof}
The proof is a consequence of Proposition 2 in Coutin and Lejay's work \cite{CoutinLejay}, except for one main fact. In order to use Coutin and Lejay's result, we consider ${\boldsymbol W}(\omega)$ as a rough path with values in $U \oplus U^{\otimes 2}$ and $\big(X(\omega),\delta_{x} X(\omega),\delta_{\mu} X(\omega),R^X(\omega)\big)$ as a controlled path; this was explained above. When doing so, the resulting integral is constructed as a process with values in $\RR^d \otimes U$, whilst the integral given by the statement of Theorem \ref{thm:integral} takes values in $\RR^d$. We denote the $\RR^d \otimes U$-valued integral by $(I_{s}^t X_{s,u}(\omega) \otimes d{\boldsymbol W}_{u}(\omega))_{(s,t) \in {\mathcal S}_{2}^T}$. We use a simple projection to pass from the infinite dimensional-valued quantity $I_{s}^t X_{s,u}(\omega) \otimes d{\boldsymbol W}_{u}(\omega)$ to the finite dimensional-valued quantity $\int_{s}^t X_{s,u}(\omega) \otimes d{\boldsymbol W}_{u}(\omega)$. Indeed, we may use the canonical projection from  $\RR^d \otimes U \cong \big(\RR^d \otimes \RR^m\big) \oplus \big(\RR^d \otimes \LL^q(\Omega, \mathcal F,\PP;\RR^m)\big)$ onto $\RR^d \otimes \textcolor{black}{\RR^m}$ to project $I_{s}^t X_{s,u}(\omega) \otimes d{\boldsymbol W}_{u}(\omega)$ onto $\int_{s}^t X_{s,u}(\omega) \otimes d{\boldsymbol W}_{u}(\omega)$.  
\end{proof}

As usual, we define an additive process setting
\begin{equation*}
\int_{s}^t X_{u}(\omega) \otimes d {\boldsymbol W}_{u}(\omega) := \int_{s}^t X_{s,u}(\omega) \otimes d {\boldsymbol W}_{u}(\omega) + X_{s}(\omega) \otimes \textcolor{black}{W_{s,t}}(\omega), 
\end{equation*}
for $0\leq t\leq T$. We can thus consider the integral process $\big(\int_{0}^t X_{s}(\omega) \otimes d{\boldsymbol W}_{s}(\omega)\big)_{0 \leq t \leq T}$ as an $\omega$-controlled trajectory with values in $\RR^{d \times m}$, with $x$-derivative a linear map from $\RR^{m}$ into $\RR^{d \times m}$, and entries
\begin{equation*}
\biggl( \delta_{x} \biggl[ \int_{0}^{\cdot} X_{s}(\omega) \otimes d{\boldsymbol W}_{s}(\omega) \biggr]_{t} \biggr)_{(i,j),k}= \textcolor{black}{\bigl( X_{t}(\omega)
\bigr)_{i} \delta_{j,k}},
\end{equation*}
for $i \in \{1,\cdots,d\}$ and $j,k \in \{1,\cdots,m\}$, where $\delta_{j,k}$ stands for the usual Kronecker symbol, and
with null $\mu$-derivative, namely
\begin{equation}
\label{EqZeroDerivativeMu}
\delta_{\mu} \biggl[ \int_{0}^{\cdot}  X_{s}(\omega) \otimes d{\boldsymbol W}_{s}(\omega) \biggr]_{t} = 0.
\end{equation}
This property is fundamental. The remainder $R^{\int X \otimes d {\boldsymbol W}}$
can
be estimated by combining
{Definition 
\ref{definition:omega:controlled:trajectory}
and}
 \eqref{eq:remainder:integral} together with the inequality
\begin{equation*}
\begin{split}
\Bigl\vert  \delta_{x} X_{s}(\omega) {\mathbb W}_{s,t}(\omega) &+ {\mathbb E} \bigl[ \delta_{\mu} X_{s}(\omega,\cdot) {\mathbb W}_{s,t}^{\indep}(\cdot,\omega) \bigr] \Bigr\vert   \\
&\leq\left\{ \sup_{r \in [0,T]} \vert \delta_{x} X_{r}(\omega,\cdot) |   + \sup_{r \in [0,T]} \langle \delta_{\mu} X_{r}(\omega) \rangle_{4/3} \right\} 
\, w(s,t,\omega)^{2/p}   \\
&\leq \vvvert X(\omega) \vvvert_{\star,[0,T],w,p} \, \Bigl( 1 + w(0,T,\omega)^{1/p} \Bigr) \, w(s,t,\omega)^{2/p}, 
\end{split}
\end{equation*}
so that, with the notation of Definition \ref{definition:omega:controlled:trajectory},
\begin{equation}
\label{eq:vvvert:integral}
\biggl\vvvert \int_{0}^{\cdot} X_{s}(\omega) \otimes d{\boldsymbol W}_{s}(\omega) \biggr\vvvert_{[0,T],w,p} < \infty.
\end{equation}
When $X(\omega)$ is given as the $\omega$-realization of a random \textcolor{black}{controlled path} $(X(\omega'))_{\omega' \in \Omega}$, the integral may be defined for any $\omega' \in \Omega$. For the integral $\int_0^{\cdot} X_s(\omega)\otimes d{\boldsymbol W}_s(\omega)$ to define a random controlled path, its $\vvvert \cdot\vvvert_{[0,T],w,p}$-semi-norm needs to have finite $8$-th moment, {see \eqref{eq:Lq:vvvert} (we give later on more precise estimates to guarantee that this may be indeed the case)}. 
{In this respect, it is worth noticing that the measurability properties of the integral with respect to 
$\omega$ can be checked by approximating the integral with compensated Riemann sums, see once again 
\eqref{eq:remainder:integral}. This gives measurability of $\Omega \ni \omega \mapsto \int_0^{t} X_s(\omega)\otimes d{\boldsymbol W}_s(\omega)$ for any given time $t \in [0,T]$. 
Measurability of the functional $\Omega \ni \omega \mapsto \int_{0}^{\cdot} X_{s}(\omega) \otimes d{\boldsymbol W}_{s}(\omega) 
\in {\mathcal C}([0,T];\RR^d \otimes \RR^m)$ then follows from the continuity of the paths.}
When the trajectory $X(\omega)$ takes values in $\RR^{d} \otimes \RR^m$ rather than $\RR^d$, the integral 
$
\int_{0}^t X_{s}(\omega) \otimes d {\boldsymbol W}_{s}(\omega) \in \RR^{d} \otimes \RR^m \otimes \RR^m$ may be identified with a tuple
\begin{equation*}
\left( \biggl(
\int_{0}^t X_{s}(\omega)\otimes d {\boldsymbol W}_{s}(\omega) 
\biggr)_{i,j,k}
\right)_{(i,j,k) \in \{1,\cdots,d\} \times
\{1,\cdots,m\} \times
\{1,\cdots,m\}  }. 
\end{equation*}
We then set for $i \in \{1,\cdots,d\}$
\begin{equation*}
\biggl( \int_{0}^t X_{s}(\omega) d {\boldsymbol W}_{s}(\omega) \biggr)_{i} := \sum_{j=1}^m \biggl( \int_{0}^t X_{s}(\omega)\otimes d {\boldsymbol W}_{s}(\omega) \biggr)_{i,j,j},
\end{equation*}
and consider $\int_{0}^t X_{s}(\omega) d {\boldsymbol W}_{s}(\omega)$ as an element of $\RR^d$.

%%-----------------------------------------------------------------------%%
\subsection{Stability of Controlled Paths under Nonlinear Maps}
\label{SubsectionStability}
%%-----------------------------------------------------------------------%%

We show in this section that controlled paths are stable under some nonlinear, sufficiently regular, maps and start by recalling the reader about the regularity notion used when working with functions defined on Wasserstein space. \textcolor{black}{We refer the reader to 
Lions' lectures \cite{Lions}, to the lecture notes \cite{LionsCardialiaguet} of Cardaliaguet 
or to Carmona and Delarue's monograph \cite[Chapter 5]{CarmonaDelarue_book_I}
for basics on the subject.}

\textcolor{gray}{$\bullet$} Recall that $(\Omega,\mathcal{F},{\mathbb P})$ stands for an atomless probability space, with $\Omega$ a Polish space and $\mathcal{F}$ its Borel $\sigma$-algebra. Fix a finite dimensional space $E=\RR^k$ and denote by $L^2:\textcolor{black}{=\LL^2(\Omega,{\mathcal F},\PP;E)}$ the space of $E$-valued random variables on $\Omega$ with finite second moment. We equip the space $\mathcal{P}_2(E) := \big\{\mathcal{L}(Z)\,;\,Z\in L^2\big\}$ with the 2-Wasserstein distance 
$$
d_2(\mu_1,\mu_2) := \inf \Big\{\|Z_1-Z_2\|_2\,;\, \mathcal{L}(Z_1) = \mu_1,\, \mathcal{L}(Z_2)=\mu_2\Big\}.
$$
An $\RR^k$-valued function $u$ defined on $\mathcal{P}_2(E)$ is canonically extended into $L^2$ by setting, for any $Z\in L^2$,
$$
U(Z) := u\big(\mathcal{L}(Z)\big).
$$

\textcolor{gray}{$\bullet$} The function $u$ is then said to be differentiable at $\mu\in\mathcal{P}_2(E)$ if its canonical lift is Fr\'echet differentiable at some point $Z$ such that $\mathcal{L}(Z)=\mu$; \textcolor{black}{we} denote by $\nabla_ZU\in (L^2)^k$ the gradient of $U$ at $Z$. The function $U$ is then differentiable at any other point $Z'\in L^2$ such that $\mathcal{L}(Z')=\mu$, and the laws of $\nabla_ZU$ and $\nabla_{Z'}U$ are equal, for any such $Z'$. 
\smallskip  
   
 \textcolor{gray}{$\bullet$} The function $u$ is said to be of class $C^1$ on some open set $O$ of $\mathcal{P}_2(E)$ if its canonical lift is of class $C^1$ in some open set of $L^2$ projecting onto $O$. It is then of class $C^1$ in the whole fiber in $L^2$ above $O$. If $u$ is of class $C^1$ on ${\mathcal P}_{2}(E)$, then $\nabla_ZU$ is $\sigma(Z)$-measurable and given by an $\mathcal{L}(Z)$-dependent function $Du$ from $E$ to 	\textcolor{black}{$E^k$} such that 
   \begin{equation}
   \label{EqDefnDu}
   \nabla_ZU = (Du)(Z);
   \end{equation}
   we have in particular $Du
   \in L^2_\mu(E;\textcolor{black}{E^k})
   \textcolor{black}{:= 
\LL^2(E,{\mathcal B}(E),\mu;E^k)}$  
   \textcolor{black}{, where ${\mathcal B}(E)$ is the Borel $\sigma$-field on $E$. In order to emphasize the fact that $Du$ depends upon 
   ${\mathcal L}(Z)$, we shall write $Du({\mathcal L}(Z))(\cdot)$ instead of $Du(\cdot)$.
 Sometimes, we shall put an index $\mu$ and write $D_{\mu} u ({\mathcal L}(Z))(\cdot)$ in order to emphasize the fact that the derivative is taken with respect to the measure argument; this will be especially useful for functionals $u$ depending on additional variables.    
   Importantly, this representation is independent of the choice of the probability space $(\Omega,{\mathcal F},{\mathbb P})$; in fact, it can be easily transported from one probability space to another. ({Simpler proofs of the structural equation \eqref{EqDefnDu} can be found in \cite{AlfonsiJourdain,WuZhang}}.)  }
\vskip 1pt

\textcolor{gray}{$\bullet$} As an example, {take $u$ 
of the form $u(\mu) = \int_{\RR^d} f(y) d\mu(y)$ for a continuously differentiable function $f : \RR^d \rightarrow \RR$ such that 
$\nabla f$ is at most of linear growth. The lift $Z \mapsto U(Z) = \EE[f(Z)]$ has differential $(d_{Z} U)(H) = 
{\mathbb E}[ \nabla f(Z) H]$ and gradient $\nabla f(Z)$. Hence, $DU(\mu)(z) = f'(z)$. 
Another example (to which we come back below) is $u(\mu) = f\big(\int_{\RR^d} \vert x \vert^2 \mu(dx)\big)$, for a continuously differentiable function 
$f : \RR \rightarrow \RR$. {The lift} ${Z \mapsto} U({Z}) = f\big(\EE[\vert {Z} \vert^2]\big)$ has differential $(d_{{Z}}U)(H) = 2 f'\big(\EE[\vert {Z}\vert^2\big)\,\EE[{Z}H]$ and gradient $2 f'\big(\EE[\vert{Z}\vert^2]\big)\,{Z}$, so ${Du(\mu)}(z) = 2 f'\big(\int_{\RR^d} \vert x\vert ^2 \mu(dx)\big)z$ here. We refer to \cite{LionsCardialiaguet} and  \cite[Chapter 5]{CarmonaDelarue_book_I} for further examples.}
\smallskip

\textcolor{gray}{$\bullet$} Back to controlled paths. Let F stand here for a map from $\RR^d\times \LL^2(\Omega,{\mathcal F},\PP;\RR^d)$ into the space ${\mathscr L}(\RR^m,\RR^d) \cong \RR^d \otimes \RR^m$ of linear mappings from $\RR^m$ to $\RR^d$. Intuitively, F should be thought of as the lift of the coefficient driving equation \eqref{EqRDE}, or, with the same notation as in \eqref{eq:lifting}, as $\widehat {\textrm{\rm F}}$ itself, with the slight abuse of notation that it requires to identify F and $\widehat{ \textrm{\rm F}}$. Our goal now is to expand the image of a controlled trajectory by F.    %{[F. : Je dois dire que je n'aime pas trop la phrase en bleu, parce que nos hypoth\`eses sont vraiment ad hoc avec la structure mean field. On ne proofande pas une d\'eriv\'ee 2e au sens de Fr\'echet : il n'y aurait presque pas d'exemples compatibles avec la structure mean field ; du coup, on proofande une autre forme de d\'eriv\'ee, propre au mean field.]} \textcolor{blue}{However, it is worth mentioning that,  for the specific purpose of this paragraph, F may be in fact a general map on $\RR^d\times \LL^2(\Omega,{\mathcal F},\PP;\RR^d)$; in words, F may not be a lift or, equivalently, it may happen that $\textrm{\rm F}(x,Z) \neq \textrm{\rm F}(x,Z')$ for two random variables $Z$ and $Z'$ with the same distribution}. The rationale for switching to such a level of generality is well understood: Lions' approach to the differential calculus on the Wasserstein space permits to lift the computations to an $\LL^2$ space; but, as long as the representation formula \eqref{EqDefnDu} of the Fr\'echet gradient does not come in the computations, the fact that the lift derives from a mapping defined on the Wasserstein space plays no role; this is precisely what happens in most of the computations below.
\vspace{4pt}

\noindent \textbf{\textbf{Regularity assumptions  1 -- }} \textit{Assume that \emph{F} is continuously differentiable in the joint variable $(x,Z)$, that 
$\partial_{x} F$ is also continuously differentiable in $(x,Z)$ and that there is some positive finite constant $\Lambda$ such that
\begin{equation}
\label{EqRegularityF}
\begin{split}
\sup_{x\in\RR^d, \,\mu\in\mathcal{P}_2(\RR^d)}\,
\bigl|\textrm{\emph{F}}(x,\mu)
\bigr\vert
\vee
\bigl|\partial_{x} \textrm{\emph{F}}(x,\mu)
\bigr\vert
\vee
\bigl|\partial_{x}^2 \textrm{\emph{F}}(x,\mu)
\bigr\vert
 \leq \textcolor{black}{\Lambda},   \\
\sup_{x\in\RR^d, \,\mathcal{L}(Z)\in\mathcal{P}_2(\RR^d)}\,\big\|\nabla_{Z}\textrm{\emph{F}}(x,Z)\big\|_2 \vee \big\|\partial_x \nabla_{Z}\textrm{\emph{F}}(x,Z)\big\|_2 \leq 
\textcolor{black}{\Lambda},
\end{split}
\end{equation}
and 
\begin{equation*}
\begin{split}
\nabla_{\textcolor{black}{Z}}\textrm{\emph{F}}(x,\cdot) : 
\LL^2(\Omega,{\mathcal F},\PP;\RR^d) &\rightarrow  \LL^2(\Omega,{\mathcal F},\PP; {\mathscr L}(\RR^d,\RR^d \otimes \RR^m))   \\
Z &\mapsto 
\nabla_{\textcolor{black}{Z}}\textrm{\emph{F}}(x,Z) = D_{\mu} F(x,{\mathcal L}(Z))(Z)
\end{split}
\end{equation*}
is \textcolor{black}{a $\Lambda$}-Lipschitz function of $\textcolor{black}{Z} \in \LL^2(\Omega,{\mathcal F},\PP;\RR^d)$, uniformly in $x\in\RR^d$.} 
\vspace{4pt}

Importantly, the $\LL^{2}$-Lipschitz bound required in the second line of \eqref{EqRegularityF} may be formulated as a Lipschitz bound on $\cP_{2}(\RR^d)$ equipped with $d_{2}$. Moreover, notice that the space $ \LL^2\big(\Omega,{\mathcal F},\PP; {\mathscr L}(\RR^d,\RR^d \otimes \RR^m)\big)$ can be identified with $ \LL^2(\Omega,{\mathcal F},\PP;\RR^d)^{d \times m}$; also, $\partial_{x} \textrm{\rm F}(x,Z)$ and $\nabla_{Z} \textrm{\rm F}(x,Z)$ will be considered as random variables with values in ${\mathscr L}(\RR^d,\RR^d \otimes \RR^m) \cong \RR^d \otimes \RR^{m} \otimes \RR^d$. As an example, the functions
$
\textrm{F}(x,\mu) = \int_{\RR^d} f(x,y)\mu(dy)
$
for some function $f$ of class $C^2_b$, and 
$
\textrm{F}(x,\mu) = g\left(x, \int_{\RR^d} y\mu(dy)\right)
$
for some function $g$ of class $C^2_b$, both satisfy \textbf{\textbf{Regularity assumptions  1}}. {A counter-example is the function $\textrm{F}(x,\mu) = \int_{\RR^d} \vert z \vert^2 d\mu(z)$}.

We expand below the path $\bigl( \textrm{F}(X_t(\omega),Y_t(\cdot))\big)_{0\leq t\leq T}$, which we write $\textrm{F}(X(\omega),Y(\cdot))$, where $X(\omega)$ is an $\omega$-controlled path and $Y(\cdot)$ is an $\RR^d$-valued random controlled path, both of them being defined on some finite interval $[0,T]$. Identity \eqref{EqZeroDerivativeMu} tells us that a fixed point formulation of  \eqref{EqRDE} will only involve pairs $(X(\omega),Y(\cdot))$ such that 
\begin{equation}
\label{EqZeroConditionMu}
\delta_\mu X(\omega) \equiv 0,\qquad \delta_\mu Y(\cdot) \equiv 0,
\end{equation}\textcolor{black}{which prompts us to restrict ourselves to the case when $X(\omega)$ and $Y$ have null $\mu$-derivatives in the expansion
\eqref{eq:omega:controlled:eq}.}

\begin{prop}
\label{prop:chaining}
Let $X(\omega)$ be an $\omega$-controlled path and $Y(\cdot)$ be an $\RR^d$-valued random controlled path. Assume that condition \eqref{EqZeroConditionMu} hold together with the \emph{$\omega$-independent} bound
\begin{equation*}
M := \sup_{0 \le t \le T} \Big( \big\vert \delta_{x} X_{t}(\omega) \big\vert \vee \big\langle \delta_{x} Y_{t}(\cdot)\big\rangle_{\infty}\Big) <\infty.
\end{equation*}
Then, $\textrm{\emph{F}}\big(X(\omega),Y(\cdot)\big)$ is an $\omega$-controlled path with 
\begin{equation*}
\delta_{x} \Bigl( \textrm{\emph{F}}\bigl(X(\omega),Y(\cdot)\bigr) \Bigr)_{t} = \partial_x\textrm{\emph{F}}\bigl(X_{t}(\omega),Y_{t}(\cdot)\bigr) \, \delta_{x} X_{t}(\omega),
\end{equation*}
which is understood as $\bigl( {\sum_{\ell=1}^d}\partial_{x_{\ell}}
\textrm{\emph{F}}^{i,j}\bigl(X_{t}(\omega),Y_{t}(\cdot)\bigr)
\bigl(\delta_{x} X_{t}^{\ell}(\omega) \bigr)_{k} \bigr)_{i,j,k}$, with $i \in \{1,\cdots,d\}$
and $j,k \in \{1,\cdots,m\}$, and (with a similar interpretation for the product)
\begin{equation*}
\begin{split}
\delta_{\mu} \Bigl(\textrm{\emph{F}}\bigl(X(\omega),Y(\cdot)\bigr) \Bigr)_{t} &= \nabla_{Z}\textrm{\emph{F}}\bigl(X_{t}(\omega),Y_{t}(\cdot)\bigr) \delta_{x} Y_{t}(\cdot)   
=  D_{\mu} \textrm{\rm F}\bigl(X_{t}(\omega),\cL({Y_{t}})\bigr)\bigl({Y_{t}}(\cdot)\bigr) \delta_{x} Y_{t}(\cdot),
\end{split}
\end{equation*}
and one can find a constant $C_{\textcolor{black}{\Lambda},M}$, depending only on \textcolor{black}{$\Lambda$} and $M$, such that 
\begin{equation*}
\big\vvvert \textrm{\emph{F}}\bigl(X(\omega),Y(\cdot)\bigr) \big\vvvert_{\star,[0,T],w,p} \leq C_{\textcolor{black}{\Lambda},M} \, \Bigl( 1+ \vvvert X(\omega) \vvvert_{[0,T],w,p}^2 + \big\langle \vvvert Y(\cdot) \vvvert_{[0,T],w,p} \big\rangle_{8}^2 \Bigr). 
\end{equation*}
\end{prop}

\begin{proof}
For $0 \leq s < t$, expand $\textrm{F}(X(\omega),Y(\cdot))_{s,t}$ into 
\begin{equation}
\label{eq:decomposition:000}
\begin{split}
&\textrm{F}(X(\omega),Y(\cdot))_{s,t} = \textrm{F}\bigl(X_{t}(\omega),Y_{t}(\cdot)\bigr) - \textrm{F}\bigl(X_{s}(\omega),Y_{s}(\cdot)\bigr)   
\\
&= 
\Bigl\{ \textrm{F}
\bigl(X_{t}(\omega),Y_{t}(\cdot) \bigr) 
- 
\textrm{F}
\bigl(X_{s}(\omega),Y_{t}(\cdot) \bigr)
\Bigr\}   \textcolor{black}{+ 
\Bigl\{\textrm{F}
\bigl(X_{s}(\omega),Y_{t}(\cdot) \bigr)
- 
\textrm{F}
\bigl(X_{s}(\omega),Y_{s}(\cdot) \bigr)
\Bigr\}}
\\
&=: \Bigl\{ \textbf{(1)} +  \textbf{(2)} +  \textbf{(3)} \Bigr\} + \Bigl\{  \textbf{(4)} +  \textbf{(5)} \Bigr\},
\end{split}
\end{equation}
where
\begin{equation*}
\begin{split}
& \textbf{(1)} := \partial_x \textrm{F}\bigl(X_{s}(\omega),Y_{s}(\cdot)\bigr) \Bigl\{ \delta_{x} X_{s}(\omega) W_{s,t}(\omega) + R_{s,t}^X(\omega) \Bigr\},   
\\
& \textbf{(2)} := \int_{0}^1 \textcolor{black}{\Bigl[} \partial_x \textrm{F}\Bigl( 
\textcolor{black}{
X_{s;(s,t)}^{(\lambda)}(\omega)},Y_{t}(\cdot) \Bigr) - \partial_x \textrm{F}\Bigl( \textcolor{black}{
X_{s;(s,t)}^{(\lambda)}(\omega)},Y_{s}(\cdot) \Bigr) \textcolor{black}{\Bigr]} X_{s,t}(\omega)\,d\lambda,   \\
& \textbf{(3)} := \int_{0}^1 
\textcolor{black}{\Bigl[} \partial_x \textrm{F}\Bigl( 
\textcolor{black}{
X_{s;(s,t)}^{(\lambda)}(\omega)},Y_{s}(\cdot) \Bigr) - \partial_x \textrm{F}\bigl( X_{s}(\omega),Y_{s}(\cdot) \bigr) \textcolor{black}{\Bigr]} X_{s,t}(\omega)\,d\lambda,   \\
& \textbf{(4)} := \textcolor{black}{\Bigl\langle} \nabla_{Z}\textrm{F}\bigl(X_{s}(\omega), \textcolor{black}{Y_{s}(\cdot)} \bigr) Y_{s,t}(\cdot) \textcolor{black}{\Bigr\rangle }
= {\Bigl\langle \nabla_{Z} \textrm{F}\bigl(X_{s}(\omega),Y_{s}(\cdot)\bigr) \Bigl\{ \delta_{x}  Y_{s}(\cdot) W_{s,t}(\cdot)
+ R_{s,t}^Y(\cdot) \Bigr\} \Bigr\rangle},   \\
& \textbf{(5)} := \int_{0}^1 \textcolor{black}{\Bigl\langle}
\textcolor{black}{\Bigl(}
 \nabla_{Z}\textrm{F}\bigl( X_{s}(\omega),Y_{s;(s,t)}^{(\lambda)}(\cdot)\bigr) - \nabla_{Z}\textrm{F}\bigl( X_{s}(\omega), \textcolor{black}{Y_{s}(\cdot)} \bigr)
\textcolor{black}{\Bigr)} Y_{s,t}(\cdot) \textcolor{black}{ \Bigr\rangle} \,d\lambda;
\end{split}
\end{equation*} 
we used here the fact that $X(\omega)$ and $Y(\cdot)$ have null $\mu$-derivative \textcolor{black}{and where we let
\begin{equation}
\label{eq:compactified:notation:interpolation}
X_{s;(s,t)}^{(\lambda)}(\omega) 
= X_{s}(\omega) + \lambda X_{s,t}(\omega), 
\quad 
Y_{s;(s,t)}^{(\lambda)}(\cdot) 
= Y_{s}(\cdot) + \lambda Y_{s,t}(\cdot).
\end{equation}}We read on %the decomposition 
\textcolor{black}{\eqref{eq:decomposition:000}}
the formulas for the $x$ and $\mu$-derivatives of $\textrm{F}(X(\omega),Y(\cdot))$. The remainder $R^{\textrm{F}(X,Y)}_{s,t}$ in the controlled decomposition of the path $\textrm{F}(X(\omega),Y(\cdot))$ is 
\begin{equation}
\label{eq:chaining:remainder}
\begin{split}
\textcolor{black}{\partial_x\textrm{F}\bigl(X_{s}(\omega),Y_{s}(\cdot)\bigr)  R_{s,t}^X(\omega) }+ \textcolor{black}{\Bigl\langle} \nabla_{Z}\textrm{F}\bigl(X_{s}(\omega),
\textcolor{black}{Y_{s}(\cdot)}
\bigr) R_{s,t}^Y(\cdot) \textcolor{black}{\Bigr\rangle} + \textbf{(2)} + \textbf{(3)}+ \textbf{(5)}. 
\end{split}
\end{equation}

We now compute $\bigl\vvvert \textrm{F}\bigl(X(\omega),Y(\cdot)\bigr) \bigr\vvvert_{\star,[0,T],w,p}$.

\begin{itemize}
   \item We have first from the assumptions on F that the initial conditions for the quantities 
   $
   \textrm{F}\big(X(\omega),\textcolor{black}{Y(\cdot)}\big)$,
   $\delta_x 
   \textcolor{black}{\bigl(}
   \textrm{F}\bigl(X(\omega),Y(\cdot)\bigr)
   \textcolor{black}{\bigr)}$,
   $\delta_\mu 
   \textcolor{black}{\bigl(}
   \textrm{F}\bigl(X(\omega),Y(\cdot)\bigr)\textcolor{black}{\bigr)},
   $
   are all bounded above by {$\Lambda(1+M)$, the bound for 
   $\delta_\mu 
   \textcolor{black}{\bigl(}
   \textrm{F}\bigl(X(\omega),Y(\cdot)\bigr)\textcolor{black}{\bigr)}
   $ being understood in $\LL^{4/3}(\Omega,{\mathcal F},\PP;\RR^d \otimes \RR^m \otimes \RR^m)$.}
     
   \item \textbf{\textbf{Variation of}} $\textrm{F}(X(\omega),Y(\cdot))$\textbf{\textbf{.}} Using the Lipschitz property of F and Lemma \ref{lem:1:1}, we have
\begin{equation*}
\begin{split}
\Big\vert \bigl[ \textrm{F}\bigl(X(\omega),Y(\cdot)\bigr) \bigr]_{s,t} \Big\vert &= \Big\vert \bigl[ \textrm{F}\bigl(X(\omega),Y(\cdot)\bigr) \bigr]_{t} - \bigl[ \textrm{F}\bigl(X(\omega),Y(\cdot)\bigr) \bigr]_{s} \Big\vert   
\\
&\leq \textcolor{black}{\Lambda}\, \Bigl( \big\vert X_{s,t}(\omega) \big\vert + \big\langle Y_{s,t}(\cdot)\big\rangle_{2}\Bigr)  
\\ 
&\leq 2 \textcolor{black}{\Lambda} \, \Bigl( \vvvert X(\omega) \vvvert_{[0,T],w,p} + \bigl\langle \vvvert Y(\cdot) \vvvert_{[0,T],w,p} \bigr\rangle_{4} \Bigr) \, w(s,t,\omega)^{1/p}.
\end{split}
\end{equation*}

   \item \textbf{\textbf{Variation of}} $\delta_{x}\big(\textrm{F}(X(\omega),Y(\cdot))\big)$ \textbf{\textbf{and}} $\delta_{\mu} \bigl( \textrm{F}(X(\omega),Y(\cdot))\bigr)$\textbf{\textbf{.}} The Lipschitz properties of $\partial_x\textrm{F}$ and $\nabla_Z\textrm{F}(x,\cdot)$ also give
\begin{equation*}
\begin{split}
&\Bigl\vert \delta_{x} \bigl[ \textrm{F}\bigl(X(\omega),Y(\cdot)\bigr) \bigr]_{s,t} \Bigr\vert 
\\
&\leq 2 \textcolor{black}{\Lambda} M\,\Bigl( \vvvert X(\omega) \vvvert_{[0,T],w,p} + \bigl\langle \vvvert Y(\cdot) \vvvert_{[0,T],w,p} \bigr\rangle_{4} \Bigr) \, w(s,t,\omega)^{1/p}   \\
&\hspace{15pt} + \textcolor{black}{\Lambda} \, \vvvert X(\omega) \, \vvvert_{[0,T],w,p} \, w(s,t,\omega)^{1/p},
\end{split}
\end{equation*}
and, applying H\"older inequality with exponents $3/2$ and $3$,
\begin{equation*}
\begin{split}
\Big\langle &\delta_{\mu} \bigl[ \textrm{F}\bigl(X(\omega),Y(\cdot)\bigr) \bigr]_{s,t} \Big\rangle_{4/3}   
\\
&\leq {\bigl\langle \delta_{x} Y_{t}(\cdot) \bigr\rangle_{\infty}
\Bigl\langle \bigl[ D_{\mu} \textrm{F}\bigl(X(\omega),Y(\cdot)\bigr) \bigr]_{s,t}
\Bigr\rangle_{2} + 
\Bigl\langle 
 D_{\mu} \textrm{F}\bigl(X_{s}(\omega),Y_{s}(\cdot)\bigr) 
\Bigr\rangle_{2}
\bigl\langle \delta_{x} Y_{s,t}(\cdot) \bigr\rangle_{4}}
\\
&\leq 2 \textcolor{black}{\Lambda} \, \big\langle \delta_{x} Y_{t}(\cdot) \big\rangle_{\infty} \, \Bigl( \vvvert X(\omega) \vvvert_{[0,T],w,p} + \langle \vvvert Y(\cdot) \vvvert_{[0,T],w,p} \rangle_{4} \Bigr) \, w(s,t,\omega)^{1/p}   
\\
&\hspace{15pt} + {\Lambda} \, \langle \delta_{x} Y_{s,t}(\cdot)\rangle_{4}   \\
&\leq 2 \textcolor{black}{\Lambda M} \Bigl( \vvvert X(\omega) \vvvert_{[0,T],w,p} + \langle \vvvert Y(\cdot) \vvvert_{[0,T],w,p} \rangle_{4} \Bigr) \, w(s,t,\omega)^{1/p}   \\
&\hspace{15pt} + 2 \textcolor{black}{\Lambda} \, \big\langle \vvvert Y(\cdot) \vvvert_{[0,T],w,p} \big\rangle_{8} \, w(s,t,\omega)^{1/p}.
\end{split}
\end{equation*}

   \item \textbf{Remainder \eqref{eq:chaining:remainder}.} The first two terms in \eqref{eq:chaining:remainder} are less than 
\begin{equation*}
\begin{split}
&\textcolor{black}{\Lambda} \, \vvvert X \vvvert_{[0,T],w,p} \, w(s,t,\omega)^{2/p} + \textcolor{black}{\Lambda} \, \big\langle R_{s,t}^Y(\cdot) \big\rangle_{2}   \\
&\leq \textcolor{black}{\Lambda} \vvvert X \vvvert_{[0,T],w,p} \,w(s,t,\omega)^{2/p} + \textcolor{black}{\Lambda} \, \big\langle \vvvert Y(\cdot) \vvvert_{[0,T],w,p} \,w(s,t,\cdot)^{2/p} \big\rangle_{2}   \\
&\leq \textcolor{black}{\Lambda} \vvvert X \vvvert_{[0,T],w,p} \,w(s,t,\omega)^{2/p} + \textcolor{black}{\Lambda} \, \big\langle \vvvert Y(\cdot) \vvvert_{[0,T],w,p} \big\rangle_{4} \,\big\langle \, w(s,t,\cdot) \big\rangle_{4}^{2/p}   \\
&\leq  \textcolor{black}{\Lambda}\, \vvvert X \vvvert_{[0,T],w,p} \,w(s,t,\omega)^{2/p} + 2\textcolor{black}{\Lambda}\, \big\langle \vvvert Y(\cdot) \vvvert_{[0,T],w,p} \big\rangle_{4} \, w(s,t,\omega)^{2/p},
\end{split}
\end{equation*}
from Lemma \ref{lem:1:1} {and the fact that $p \in [2,3)$}. We also have
\begin{equation*}
\begin{split}
&\big\vert \textbf{(2)} \big\vert \leq \textcolor{black}{\Lambda} \, \big\vert X_{s,t}(\omega) \big\vert\, \big\langle Y_{s,t}(\cdot)\big\rangle_{2}   
\leq 2\textcolor{black}{\Lambda} \, \big\vvvert X(\omega) \big\vvvert_{[0,T],w,p} \, \big\langle \vvvert Y(\cdot) \vvvert_{[0,T],w,p} \big\rangle_{4} \, w(s,t,\omega)^{2/p},
\\
&\big\vert  \textbf{(3)} \big\vert \leq  \textcolor{black}{\Lambda}\, \big\vert X_{s,t}(\omega) \big\vert^2 \leq \textcolor{black}{\Lambda}\, \big\vvvert X(\omega) \big\vvvert_{[0,T],w,p}^2 \,  w(s,t,\omega)^{2/p}.
\end{split}
\end{equation*}
Last, since $\textcolor{black}{\nabla_{Z}} \textrm{F}$ is a Lipchitz function of its second argument,
\begin{equation*}
\textbf{(5)} \leq \textcolor{black}{\Lambda} \, \big\langle Y_{s,t}(\cdot)\big\rangle_{2}^2 \leq 4 \textcolor{black}{\Lambda} \, \big\langle \vvvert Y(\cdot) \vvvert_{[0,T],w,p}\big\rangle_{4}^{\textcolor{black}{2}} \, w(s,t,\omega)^{2/p}.
\end{equation*}
\end{itemize}
{Collecting the various terms, we complete the proof.}
\end{proof}

%-----------------------------------%
\section{Solving the Equation}
\label{SectionSolving}
%-----------------------------------%

We now have all the tools to formulate the  equation \eqref{eq:1:2} (or 
\eqref{EqRDE}) as a fixed point problem and solve it by Picard iteration. 
Our definition of the fixed point is given in the form of a two-step procedure: The first step is to write a \textit{frozen} version of the equation, in which the mean field component is seen as an exogenous collection of $\omega$-controlled trajectories; the second step is to regard the family of exogenous controlled trajectories as an input and to map it to the collection of controlled trajectories solving the frozen version of the equation. In this way, we define a solution as a collection of $\omega$-controlled trajectories. In order to proceed, recall the generic notation $\big(X(\omega) ; \delta_xX(\omega) ; \partial_\mu X(\omega,\cdot)\big)$ for an $\omega$-controlled path and its derivatives; we sometimes abuse notations and talk of $X(\omega)$ as an $\omega$-controlled path.

\begin{defn}
\label{DefnGamma}
Let $W$ together with its enhancement ${\boldsymbol W}$ satisfy the assumption of Section \emph{\ref{SectionRoughStructure}} on a finite interval $[0,T]$, and let $Y(\cdot)$ stand for some $\RR^d$-valued random controlled path on $[0,T]$,  with the property that $\delta_{\mu} Y(\cdot) \equiv 0$ and  $\sup_{0 \leq t \leq T} \, \langle    \delta_{x} Y_{t}(\cdot) \rangle_{\infty} <\infty$.
For a given $\omega \in \Omega$, let $X(\omega)$ be an $\RR^d$-valued $\omega$-controlled path on $[0,T]$, with the properties that $\delta_{\mu} X(\omega) \equiv 0$ and $\sup_{0 \leq t \leq T} \vert \delta_{x} X_{t}(\omega)| < \infty$. We associate to $\omega$ and $X(\omega)$ an $\omega$-controlled path by setting
\begin{equation*}
\begin{split}
&\Gamma\big(\omega, X(\omega) ,  Y(\cdot)\big) 
\\ 
&\hspace{15pt} := \biggl( X_{0}(\omega) + \int_{0}^{t} \textrm{\emph{F}}\bigl(X_{s}(\omega),Y_{s}(\cdot)\bigr) d{\boldsymbol W}_{s}(\omega) \,;\, \textrm{\emph{F}}\bigl(X_{t}(\omega),Y_{t}({\cdot})\bigr)\,;\, 0 \biggr)_{0 \le t \le T}.
\end{split}
\end{equation*}   
A \textbf{\textbf{solution to the mean field rough differential equation}}
$
dX_t = \textrm{\emph{F}}\big(X_t, \mathcal{L}(X_t)\big)\, d{\boldsymbol W}_t,
$
on the time interval $[0,T]$, with given initial condition $X_{0}(\cdot) \in L^2(\Omega,{\mathcal F},{\mathbb P};\RR^d)$ is a random controlled path $X(\cdot)$ starting from $X_0(\cdot)$ and satisfying the same prescription as $Y(\cdot)$, such that for ${\mathbb P}$-a.e. $\omega$ the path $X(\omega)$ and $\Gamma\big(\omega , X(\omega) , X(\cdot)\big)$ coincide.
\end{defn}

We should more properly replace $X(\omega)$ in $\Gamma\big(\omega , X(\omega) , Y(\cdot)\big)$ by $\big(X(\omega) \, ; \delta_xX(\omega) \, ; 0\big)$ and $Y(\cdot)$ by $\big(Y(\cdot) \, ; \delta_xY(\cdot) \, ; 0\big)$, but we stick to the above lighter notation. Observe also that our formulation bypasses any requirement on the properties of the map $\Gamma$ itself. To make it clear, we should be indeed tempted to check that, for a random controlled path $X(\cdot)$, the collection $\big(\Gamma(\omega , X(\omega) , Y(\cdot))\big)_{\omega \in \Omega}$, for $Y(\cdot)$ as in the statement, is also a random controlled path. Somehow, our definition of a solution avoids this question; however, we need to check this fact in the end; below, we refer to it as the \textit{stability properties of $\Gamma$}, see Section \ref{SubsectionStability}.

What remains of the above definition when ${\boldsymbol W}$ is the It\^o or Stratonovich enhancement of a Brownian motion? The key point to connect the above notion of solution with the standard notion of solution to mean field stochastic differential equation is to observe that the rough integral therein should be, if a solution exists, the limit of the compensated Riemann sums 
\begin{equation*}
\begin{split}
&\sum_{j=0}^{n-1} \biggl( \textrm{\rm F} \bigl(X_{t_{j}}(\omega),X_{t_{j}}(\cdot) \bigr) W_{t_{j},t_{j+1}}(\omega)   
+ \partial_x \textrm{\rm F}\bigl(X_{t_{j}}(\omega),X_{t_{j}}(\cdot)\bigr) \textrm{\rm F} \bigl(X_{t_{j}}(\omega),X_{t_{j}}(\cdot) \bigr)
{\mathbb W}_{t_{j},t_{j+1}}(\omega)    \\
&\hspace{15pt}
+ \Bigl\langle D_{\mu} \textrm{\rm F}\bigl(X_{t_{j}}(\omega),X_{t_{j}}(\cdot)\bigr)\bigl(X_{t_{j}}(\cdot)\bigr) \textrm{\rm F} \bigl(X_{t_{j}}(\omega),X_{t_{j}}(\cdot) \bigr) {\mathbb W}_{t_{j},t_{j+1}}^{\indep}(\cdot,\omega) \Bigr\rangle \biggr),
\end{split}
\end{equation*}
as the step of the dissection $0=t_{0}<\dots<t_{n}=t$ tends to $0$. When the solution is constructed by a contraction argument, such as done below, the process $(X_{t}(\cdot))_{0 \leq t \leq T}$ is adapted with respect to the completion of the filtration $({\mathcal F}_{t})_{0 \le t \le T}$ generated by the initial condition $X_{0}(\cdot)$ and the Brownian motion $W(\cdot)$. Returning if necessary to Example \ref{example:3}, we then 
check that 
${\mathbb E}
\bigl[ {\mathbb W}^{\indep}_{t_{j},t_{j+1}}(\cdot,\omega) \, \vert\, {\mathcal F}_{t_{j}}
\bigr] = 0$, 
whatever the interpretation of the rough integral, It\^o or Stratonovich. Pay attention that the conditional expectation is taken with respect to  ``$\cdot$'', while  $\omega$ is kept frozen. This implies that, for any $j \in \{0,\cdots,n-1\}$, we have
\begin{equation*}
\Bigl\langle D_{\mu} \textrm{\rm F}\bigl(X_{t_{j}}(\omega),X_{t_{j}}(\cdot)\bigr)\bigl(X_{t_{j}}(\cdot)\bigr) \textrm{\rm F} \bigl(X_{t_{j}}(\omega),X_{t_{j}}(\cdot) \bigr) {\mathbb W}_{t_{j},t_{j+1}}^{\indep}(\cdot,\omega) \Bigr\rangle = 0.
\end{equation*}
This proves that the solution to the rough mean field equation coincides with the solution that is obtained when \eqref{EqRDE} is interpreted in the standard McKean-Vlasov sense ({the stochastic integral in the McKean-Vlasov equation being usually understood in the It\^o sense and the iterated integral ${\mathbb W}$ being defined accordingly}).

We formulate here the regularity assumptions on $\textrm{F}(x,\mu)$ needed \textcolor{black}{to show that $\Gamma$ satisfies the required stability properties and to run Picard's iteration for proving} the well-posed character of   \eqref{eq:1:2} (or \eqref{EqRDE}) in small time, or in some given time interval. Recall from \eqref{EqDefnDu} the definition of $D_{\mu}\textrm{F}(x,\cdot)(\cdot)$ as a function from $\textcolor{black}{\cP_{2}(\RR^d) \times} \RR^d$ to ${\mathscr L}(\RR^{d},\RR^d \otimes \RR^m) \cong \RR^{d} \otimes \RR^m \otimes \RR^d$ such that $D_{\mu}\textrm{F}(x,\textcolor{black}{{\mathcal L}(Z)})(Z) = \nabla_Z\textrm{F}(x,Z)$, \textcolor{black}{where we} emphasize the dependence of $D_{\mu}\textrm{F}(x,\cdot)$ on $\mu=\mathcal{L}(Z)$ by writing $D_{\mu}\textrm{F}(x,\mu)(\cdot)$. On top of \textbf{\textbf{Regularity assumptions 1}}, we assume
\vspace{4pt}

\noindent \textbf{\textbf{Regularity assumptions  2 -- }} 

\textcolor{gray}{$\bullet$} \textit{The function $\partial_{x} \textrm{\rm F}$ is differentiable in $(x,\mu)$ in the same sense as {\textrm{\rm F} itself}.} 

\textcolor{gray}{$\bullet$} \textit{For each $(x,\mu) \in \RR^d \times \cP_{2}(\RR^d)$, there exists a version of $D_{\mu}\textrm{\emph{F}}(x,\mu)(\cdot) \in L^2_{\mu}(\RR^d;\RR^d \otimes \RR^m)$ such that the map
$
(x,\mu,z) \mapsto D_{\mu}\textrm{\emph{F}}(x,\mu)(z)
$
from $\RR^d\times\mathcal{P}_2(\RR^d)\times\RR^d$ to $\textcolor{black}{\RR^{d} \otimes \RR^m \otimes \RR^d}$ is of class $C^1$, the derivative in the direction $\mu$ being understood as before.}

\textcolor{gray}{$\bullet$}
\textit{The function 
$\big(x,Z\big) \mapsto \partial_{x}^2 \textrm{\emph{F}}\bigl(x,\cL(Z)\bigr)$ from $\RR^d \times \LL^2(\Omega,{\mathcal F},\PP;\RR^d)$ to $\RR^{d} \otimes \RR^{m} \otimes \RR^d \otimes \RR^d
\cong {\mathscr L}(\RR^d \otimes \RR^d,\RR^d \otimes \RR^m)$ is bounded by $\Lambda$ and $\Lambda$-Lipschitz continuous. }

\textcolor{gray}{$\bullet$} \textit{The three functions
$(x,Z) \mapsto \partial_x D_{\mu}\textrm{\emph{F}}\bigl(x,\cL(Z)\bigr)(Z(\cdot))$, 
$(x,Z) \mapsto D_{\mu} \partial_x \textrm{\emph{F}}\bigl(x,\cL(Z)\bigr)(Z(\cdot))$,  
and $(x,Z) \mapsto \partial_z D_{\mu}\textrm{\emph{F}}\bigl(x,\cL(Z)\bigr)(Z(\cdot))$ from $\RR^d \times \LL^2(\Omega,{\mathcal F},\PP;\RR^d)$ to $\LL^2\bigl(\Omega,{\mathcal F},\PP;\RR^{d} \otimes \RR^{m} \otimes \RR^d \otimes \RR^d\bigr)$, are bounded by $\Lambda$ and $\Lambda$-Lipschitz continuous. (By Schwarz' theorem, the transpose of $\partial_{x} D_{\mu} \textrm{\rm F}^{i,j}$ is in fact equal to $D_{\mu} \partial_{x} \textrm{\rm F}^{i,j}$, for any 
$i \in \{1,\cdots,d\}$ and $j \in \{1,\cdots,m\}$.) 
}

\textcolor{gray}{$\bullet$} \textit{For each $\mu \in {\mathcal P}_{2}(\RR^d)$, we denote by 
$
D^2_{\mu}\textrm{\emph{F}}(x,\mu)(z,\cdot)
$
the derivative of $D_{\mu}\textrm{\emph{F}}(x,\mu)(z)$ with respect to $\mu$ -- which is indeed given by a function. For $z' \in \RR^d$, 
$D_{\mu}^2\textrm{\emph{F}}(x,\mu)(z,z')$
is an element of $\RR^d \otimes \RR^m \otimes \RR^d \otimes \RR^d$.}  

 \textit{Denote by $\big(\widetilde{\Omega},\widetilde{\mathcal F},\widetilde{\PP}\big)$ a copy of $(\Omega,{\mathcal F},\PP)$, and given a random variable $Z$ on $(\Omega,{\mathcal F},\PP)$, write $\widetilde{Z}$ for its copy on $(\widetilde\Omega,\widetilde{\mathcal F},\widetilde\PP)$. We assume that 
$
(x,Z) \mapsto D^2_{\mu} \textrm{\emph{F}}\bigl(x,\cL(Z)\bigr)\big(Z(\cdot),\widetilde{Z}(\cdot)\big),
$ from $\RR^d \times \LL^2(\Omega,{\mathcal F},\PP;\RR^d)$ to $\LL^2\bigl(\Omega \times \widetilde{\Omega},{\mathcal F} \otimes \widetilde{\mathcal F},\PP \otimes \widetilde{\mathbb P};\RR^{d} \otimes \RR^{m} \otimes \RR^d \otimes \RR^d\bigr)$, is bounded by $\Lambda$ and $\Lambda$-Lipschitz continuous.}  
\vspace{4pt}

The two functions
$
\textrm{F}(x,\mu) = \int f(x,y)\mu(dy)
$
for some fuction $f$ of class $C^3_b$, and
$
\textrm{F}(x,\mu) = g\left(x, \int y\mu(dy)\right)
$
for some function $g$ of class $C^3_b$, both satisfy \textbf{\textbf{Regularity assumptions 2}}. We refer to \cite[Chapter 5]{CarmonaDelarue_book_I} and \cite[Chapter 5]{CarmonaDelarue_book_II} for other examples of functions that satisfy the above assumptions and for sufficient conditions under which these assumptions are satisfied. We feel free to abuse notations and write $Z(\cdot)$ for $\cL(Z)$ in the argument of the functions $\partial_x D_{\mu}\textrm{F}$, $\partial_z D_{\mu}\textrm{F}$ and $D^2_{\mu}\textrm{F}$. We prove in Section \ref{SubsectionStability} that the map $\Gamma$ sends some large ball of its state space into itself for a small enough $T$. The contractive character of $\Gamma$ is proved in Section \ref{SubsectionContractivity}, and Section \ref{SubsectionWellPosedness} is dedicated to proving the well-posed character of \eqref{eq:1:2}.

%%-----------------------------------------------%%
\subsection{Stability of Balls by $\Gamma$}
\label{SubsectionStability}
%%-----------------------------------------------%%

Recall $\Lambda$ was introduced in \textbf{\textbf{Regularity assumptions  1}} and
\textbf{\textbf{2}} as a bound on F and some of its derivatives. 
Recall {also} from \eqref{eq:N:s:t:omega} the definition of 
$
N\big([0,T],\omega;\alpha\big).
$
We also use  below the notations $\vvvert\cdot\vvvert_{[a,b], w, p}$ and $\vvvert\cdot\vvvert_{\star, [a,b], w, p}$, for some interval $[a,b]$, to denote the same quantity as in Definition \ref{definition:random:controlled:trajectory} but for paths defined on $[a,b]$ rather than on  $[0,T]$ (the initial condition is then taken at time $a$).

\begin{prop}
\label{thm:main:1}
Let \emph{F} satisfy \textcolor{black}{\textbf{\textbf{Regularity assumptions 1}} {and $w$ be a control 
satisfying \eqref{eq:w:s:t:omega:ineq}
and 
\eqref{eq:useful:inequality:wT}}}. Consider an $\omega$-controlled path $X(\omega)$ together with a random controlled path $Y(\cdot)$, \textcolor{red}{both of them} satisfying {\eqref{EqZeroConditionMu}} together with 
\begin{equation}
\label{eq:invariant:0000}
\sup_{0 \leq t \leq T}
\Bigl(
\bigl|\delta_{x} X_{t}(\omega)\bigr|  \vee \big\langle \delta_{x} Y_{t}(\cdot) \big\rangle_{\infty} \Bigr) \leq {\Lambda}.
\end{equation}

\textcolor{gray}{$\bullet$} Assume that there exists a positive constant $L$ such that we have
\begin{equation}
\label{eq:invariant:1}
%\bigl\langle \| Y(\cdot) \|_{[0,T],w,p} \bigr\rangle_{8}^2 \leq \sqrt{L}, \qquad 
 \bigl\langle \vvvert Y(\cdot) \vvvert_{[0,T],w,p}\bigr\rangle_{8}^2 \leq L, 
\end{equation}
and 
\begin{equation}
\label{eq:invariant:2}
\big\vvvert X(\omega) \big\vvvert_{[t_{i},t_{i+1}],w,p}^2 \leq {L}, 
\end{equation}
for all $0\leq i \leq N$, with $N := N(\textcolor{black}{[0,T]},\omega,1/(4L))$, and for the sequence of times  $\big(t_{i}:=\tau_{i}(0,T,\omega,1/(4L))\big)_{i=0,\cdots,N+1}$ given by \eqref{eq:stopping:times}
{with 
$\varpi(s,t) = w(s,t,\omega)^{1/p}$}. 

Then:

\textcolor{gray}{$\bullet$} There exists a constant 
$c>1$, {only depending on $\Lambda$,} such that 
 \eqref{eq:invariant:1}
 and
\eqref{eq:invariant:2} remain true if we replace $L$ by $L'$, provided that $L' \geq c L$ 
and the partition 
$(t_{i})_{i=0,\cdots,N+1}$ is recomputed accordingly ({since $L$ enters the definition of the partition}).
Also, 
{we can find a constant $L_{0}'$, only depending on $L$, such that 
for the same constant $c$ and for $L' \geq L_{0}'$, the path $\Gamma\big(\omega , X(\omega), Y(\cdot)\big)$ satisfies for each $\omega$ the size 
   estimate \eqref{eq:invariant:2}, $L$ being replaced by $c$ in the right-hand side and 
   the partition  
   $(t_{i})_{i=0,\cdots,N+1}$ in the left-hand side being defined with respect to $L'$ instead of $L$.}
 
\textcolor{gray}{$\bullet$} Moreover, there exist {two constants $L_{0}$
and $C$}, 
 {only depending on $\Lambda$}, %and a constant $C_{L}$, only depending on $L$ and {$\Lambda$}, 
 such that, if $L$ in 
 \eqref{eq:invariant:1}
 and
\eqref{eq:invariant:2}
is greater than $L_{0}$, the following estimates hold for each $\omega$:
\begin{equation}
\label{eq:invariant:3}
\begin{split}
\big\vvvert \Gamma\big(\omega, X(\omega), Y(\cdot)\big)\big\vvvert_{[0,T],w,p}^2 &\leq {C} \, \Big\{ 1+ N\Bigl([0,T], \omega,1/(4L)\Bigr)^{2(1-1/p)} \Bigr\},   \\
\big\vvvert \Gamma\bigl(\omega , X(\omega), Y(\cdot)\bigr) \big\vvvert_{\star,[0,T],w,p}^2 &\leq {C \big\vert X_{0}(\omega) \big\vert^2}
%\\
%&\hspace{15pt} 
+ {C} \,\left\{ 1+ N\Bigl([0,T], \omega,1/(4L)\Bigr)^{2(1-1/p)} \right\};
\end{split}
\end{equation}   
   
\textcolor{gray}{$\bullet$} Lastly, if $X(\omega)$ is the \textcolor{black}{$\omega$}-realization of a random controlled path $X(\cdot)=\bigl(X(\omega')\bigr)_{\omega' \in \Omega'}$ such that the estimate $\big\vvvert X(\omega') \big\vvvert_{[t_{i},t_{i+1}],w,p}^2 \leq {L}$ holds for all $\omega'$, for the $\omega'$-dependent partition $\bigl(t_{i}:=
  \tau_{i}(0,T,\omega',1/(4L))\bigr)_{i=0,\cdots,N+1}$ 
  of $[0,T]$, with
{$L$ in 
 \eqref{eq:invariant:1} satisfying $L \geq L_{0}$}
 and with
   $N:=N([0,T],\omega',1/(4L))$, 
  and if $T$ is small enough to have
$$
\textcolor{red}{\Big\langle N\bigl(\textcolor{black}{[0,T]},\cdot,1/(4L)\bigr) \Big\rangle_{8}} \leq {1};
$$   
then 
\begin{equation*}
\begin{split}
%&\bigl\langle \| \Gamma(\cdot , X(\cdot), Y) \|_{[0,T],w,p} \bigr\rangle_{8}^2 \leq \sqrt{L}, \qquad  
&\bigl\langle \vvvert \Gamma(\cdot , X(\cdot), Y) \vvvert_{[0,T],w,p}\bigr\rangle_{8}^2 {\leq 2C \leq L}, 
\\ 
\text{and} \qquad &\Bigl\langle \big\vvvert \Gamma(\cdot,  X(\cdot),  Y) \big\vvvert_{\star,[0,T],w,p} \Bigr\rangle_{\textcolor{black}{2}}^2 \leq C \Bigl(2+ \bigl\langle X_{0}(\cdot) \bigr\rangle_{\textcolor{black}{2}}^2 \Bigr).
\end{split}
\end{equation*}
\end{prop}

{Following 
the discussion after 
\eqref{eq:vvvert:integral},} the measurability properties of the map $\omega\mapsto \Gamma\big(\omega,X(\omega),Y(\cdot)\big)$ implicitly required above can be checked by approximating the integral in the definition of $\Gamma\big(\omega, X(\omega), Y(\cdot)\big)$, using \eqref{eq:remainder:integral}.
{We also notice that the constraint $L \geq L_{0}$ required in the second and third bullet points may be easily circumvented. Indeed,
the first claim in the statement guarantees that, for $L$ satisfying  
 \eqref{eq:invariant:1}
 and
\eqref{eq:invariant:2}, 
$L' \geq cL$ also satisfy 
 \eqref{eq:invariant:1}
 and
\eqref{eq:invariant:2}, see footnote\footnote{{While the reader may find it obvious, she/he must be aware of the fact that, in \eqref{eq:invariant:2}, $t_{i}$ and $t_{i+1}$ themselves depend on $L$, which forces to recompute the subdivision when 
$L$ is changed.}}. In particular, we can always apply the second and third bullet points with $L' \geq c L_{0}$ instead of $L$ itself, which is a good point since $L'$ is here a free parameter while the value of $L$ is prescribed by the statement. }

\begin{proof}
We first explain the reason 
why 
\eqref{eq:invariant:2}
remains true for possibly larger values of $L$ provided that the right-hand side is multiplied by a universal multiplicative constant. 
Take $L'>L$ and call $(t_{j}')_{j=0,\cdots,N'+1}$ the corresponding dissection. 
It is clear that any interval 
$[t_{j}',t_{j+1}']$ must be included in an interval of the form $[t_{i},t_{i+2} \wedge T]$. 
If $[t_{j}',t_{j+1}'] \subset [t_{i},t_{i+1}]$, the proof is done. If 
$t_{i+1} \in (t_{j}',t_{j+1}')$, 
it is an easy exercise\footnote{\label{footnote:6} The proof is as follows. 
By the super-addivitiy of $w$, 
see 
\eqref{eq:useful:inequality:wT}, and the inequality 
$a^{1/p} + b^{1/p} \leq 2^{1-1/p} (a+b)^{1/p}$,
the terms $\| X(\omega) \|_{[t_{j}',t_{j+1}'],w,p}$, $\| \delta_{x} X(\omega) \|_{[t_{j}',t_{j+1}'],w,p}$ 
and $\langle \delta_{\mu} X(\omega,\cdot) \rangle_{[t_{j}',t_{j+1}'],w,p,4/3}$ 
are easily handled. So, the only difficulty is to handle 
$\| R^X \|_{[t_{j}',t_{j+1}'],w,p}$. By 
\eqref{eq:omega:controlled:eq}, 
we have, for any $0 \leq r \leq s \leq t \leq T$,
$R^X_{r,t}(\omega) =
% X_{r,t}(\omega) - 
%\delta_{x} X_{r}(\omega) W_{r,t}(\omega) - 
%\EE \bigl[ \delta_{\mu} X_{r}(\omega,\cdot) W_{r,t}(\cdot) \bigr]
%\\
%&= 
R^X_{r,s}(\omega) + R^X_{s,t}(\omega) 
+ \delta _{x} X_{r,s}(\omega)
W_{s,t}(\omega) + \EE \bigl[ \delta_{\mu} X_{r,s}(\omega,\cdot) W_{s,t}(\cdot) \bigr]$, 
which suffices for our purpose.} to check that 
$\vvvert \, \cdot \, \vvvert_{[t_{j}',t_{j+1}'],w,p}
\leq \gamma 
\vvvert \, \cdot \, \vvvert_{[t_{j}',t_{i+1}],w,p}
+ 
\gamma 
\vvvert \, \cdot \, \vvvert_{[t_{i+1},t_{i+2} \wedge T],w,p}$,
for some universal constant $\gamma$. 
%(If needed, we make it clear in the second bullet point below, but in a more general framework.)
{This yields 
$\vvvert \, \cdot \, \vvvert_{[t_{j}',t_{j+1}'],w,p}
\leq 2 \gamma L^{1/2}$, which is indeed less than 
$(L')^{1/2}$ if $L' \geq 2^2 \gamma^2 L$.}

{Given this preliminary remark, the proof proceeds in three steps.}
\smallskip

\textcolor{gray}{$\bullet$} For $\omega \in \Omega$, consider a subdivision $(t_{i})_{0 \leq i \leq N+1}$ of $[0,T]$ such that 
$w(t_{i},t_{i+1},\omega) \leq 1$
for all $i \in \{0,\cdots,N\}$, for some integer $N \geq 0$. Then, following {\cite[Proposition 4]{CoutinLejay2}} (rearranging the terms therein), we know that\footnote{\label{footnote:7} 
In fact, the inequality may be checked directly. 
%Precisely, using the same notation as therein, 
%\cite[(18)]{CoutinLejay2} 
%with $y_{s}=F(X_{s}(\omega),Y_{s}(\cdot))$, 
%$z_{s}=W_{s}(\omega)$, ${\mathbf x}={\boldsymbol W}(\omega)$, 
%with $(p,q,r)=(p,p,p/2)$ 
Identity \eqref{eq:remainder:integral}
together with Proposition 
\ref{prop:chaining}
and 
\textbf{Regularity assumptions  1} 
say that the remainder $R^{\int \textrm{F}}$ in the $\omega$-controlled expansion of 
$\int_{t_{i}}^{\cdot} \textrm{F}\bigl(X_{r}(\omega),Y_{r}(\cdot)\bigr) d {\boldsymbol W}_{r}(\omega)$
satisfies
\begin{equation*}
\begin{split}
\bigl\| R^{\int \textrm{F}} \bigr\|_{[t_{i},t_{i+1}],w,p/2} 
&\leq 2 \sup_{s \in [t_{i},t_{i+1}]} 
\Bigl( 
\bigl\vert \delta_{x} \bigl[ \textrm{\rm F}\bigl(X_{s}(\omega),Y_{s}(\cdot)\bigr) \bigr] \bigr\vert  
+ \bigl\langle 
 \delta_{\mu} \bigl[ \textrm{\rm F}\bigl(X_{s}(\omega),Y_{s}(\cdot)\bigr) \bigr]  
\bigr\rangle_{4/3} 
\Bigr)
\\
&\hspace{15pt} +
\gamma  
\vvvert \textrm{F}(X(\omega),Y(\cdot)) \vvvert_{[t_{i},t_{i+1}],w,p} w(t_{i},t_{i+1},\omega)^{1/p}
\\
&\leq \gamma + \gamma  
\vvvert \textrm{F}(X(\omega),Y(\cdot)) \vvvert_{[t_{i},t_{i+1}],w,p} w(t_{i},t_{i+1},\omega)^{1/p},
\end{split}
\end{equation*}
for a constant $\gamma$ that may depend on $\Lambda$. 
 This permits to handle $R^{\int \textrm{F}}$.
As the Gubinelli derivative of $\int_{t_{i}}^{\cdot} \textrm{F}\bigl(X_{r}(\omega),Y_{r}(\cdot)\bigr) d {\boldsymbol W}_{r}(\omega)$ is exactly given by $\textrm{F}(X_{\cdot}(\omega),Y_{\cdot}(\cdot))$ itself, we 
%can invoke 
%\cite[(12)]{CoutinLejay2} (together with the fact that $\textrm{F}$ is $1$-Lipschitz and $\delta_{x} X_{t}$ is bounded by $1$) to 
get from \eqref{eq:omega:controlled:eq} with $X=\textrm{F}$  that 
\begin{equation*}
\begin{split}
\bigl\| \textrm{F}(X(\omega),Y(\cdot)) \bigr\|_{[t_{i},t_{i+1}],w,p}  &\leq 
2 \sup_{s \in [t_{i},t_{i+1}]} 
\Bigl( 
\bigl\vert \delta_{x} \bigl[ \textrm{\rm F}\bigl(X_{s}(\omega),Y_{s}(\cdot)\bigr) \bigr] \bigr\vert  
+ \bigl\langle 
 \delta_{\mu} \bigl[ \textrm{\rm F}\bigl(X_{s}(\omega),Y_{s}(\cdot)\bigr) \bigr]   
\bigr\rangle_{4/3} 
\Bigr)
\\
&\hspace{15pt}+ \| R^{\textrm{F}} \|_{[t_{i},t_{i+1}],w,p/2} w(t_{i},t_{i+1},\omega)^{1/p},
\end{split}
\end{equation*}
where $R^{\textrm{F}}$ is the remainder in the expansion of $\textrm{F}$. We conclude as for $R^{\int \textrm{F}}$. 
%Clearly, $\| R^{\textrm{F}} \|_{[t_{i},t_{i+1}],w,p/2}$ 
%is also easily handled in 
%terms of $\vvvert \textrm{F}(X(\omega),Y(\cdot)) \vvvert_{{\star},[t_{i},t_{i+1}],w,p}$.
In order to control the variation of $\int_{t_{i}}^{\cdot} \textrm{F}\bigl(X_{r}(\omega),Y_{r}(\cdot)\bigr) d {\boldsymbol W}_{r}(\omega)$ itself, it suffices to invoke \eqref{eq:omega:controlled:eq} again,  
%\cite[(12)]{CoutinLejay2} 
but with $X=\int \textrm{\rm F}$, which yields 
\begin{equation*}
\begin{split}
&\biggl\| \int_{t_{i}}^{\cdot} \textrm{F}\bigl(X_{r}(\omega),Y_{r}(\cdot)\bigr) d {\boldsymbol W}_{r}(\omega) \biggr\|_{[t_{i},t_{i+1}],w,p}
\\
&\hspace{15pt} \leq 
\sup_{s \in [t_{i},t_{i+1}]} 
\bigl\vert   \textrm{\rm F}\bigl(X_{s}(\omega),Y_{s}(\cdot)\bigr)   \bigr\vert  
+ \| R^{\int \textrm{F}} \|_{[t_{i},t_{i+1}],w,p/2} w(t_{i},t_{i+1},\omega)^{1/p}.
\end{split}
\end{equation*}
The conclusion is the same.
%$y = \int_{t_{i}}^{\cdot} 
% \textrm{F}\bigl(X_{r}(\omega),Y_{r}(\cdot)\bigr) d {\boldsymbol W}_{r}(\omega)$, which precisely permits to handle the $p$-variation of the integral in terms of the sup norm of its Gubinelli derivative (but $\textrm{F}$ is bounded) and of the $p/2$-variation of the remainder $R^{\int \textrm{F}}$.
}
\begin{equation*}
\begin{split}
\biggl\vvvert \int_{t_{i}}^{\cdot} \textrm{F}\bigl(X_{r}(\omega),Y_{r}(\cdot)\bigr) &d {\boldsymbol W}_{r}(\omega) \biggr\vvvert_{[t_{i},t_{i+1}],w,p}   \\
&\leq \gamma + \gamma w(t_{i},t_{i+1},\omega)^{1/p} \Bigl\vvvert \textrm{F}\bigl(X(\omega),Y(\cdot)\bigr) \Big\vvvert_{[t_{i},t_{i+1}],w,p}, 
\end{split}
\end{equation*}
for a universal constant $\gamma$ {that may depend on $\Lambda$}. By \textcolor{black}{Proposition \ref{prop:chaining}}
\textcolor{black}{and 
\eqref{eq:invariant:0000}}, we deduce that (for a new value of {$C_{\Lambda,\Lambda}$})
\begin{equation}
\label{eq:starstarstar}
\begin{split}
&\biggl\vvvert \int_{t_{i}}^{\cdot}  \textrm{F} \bigl(X_{r}(\omega),Y_{r}(\cdot)\bigr) d {\boldsymbol W}_{r}(\omega)\biggr\vvvert_{[t_{i},t_{i+1}],w,p}   \\
&\hspace{15pt}\leq \gamma + {C_{\Lambda,\Lambda}} \, \gamma \, w(t_{i},t_{i+1},\omega)^{1/p} \Bigl(1+ \vvvert X{(\omega)} \vvvert_{[t_{i},t_{i+1}],w,p}^2 + \bigl\langle 
\vvvert Y(\cdot) \vvvert_{[0,T],w,p} \bigr\rangle_{8}^2 \Bigr).
\end{split}
\end{equation}
{By the first conclusion in the statement (see also the discussion after the statement itself), we can assume that 
$L$ differs from the value prescribed in the statement and is as large as needed. So, for the time being, we take $L \geq 1$ and we assume that} $\,w(t_{i},t_{i+1},\omega)^{1/p} \leq 1/(4L) \leq 1$ and
\begin{equation}
\label{eq:proof:stab}
%\bigl\langle \| Y(\cdot) \|_{[0,T],w,p} \bigr\rangle_{8}^2 \leq \sqrt{L}, \qquad 
\bigl\langle \vvvert Y(\cdot) \vvvert_{[0,T],w,p} \bigr\rangle_{8}^2 \leq L, 
\end{equation}
and
\begin{equation}
\label{eq:proof:stab:2}
\big\vvvert X(\omega) \big\vvvert_{[t_{i},t_{i+1}],w,p}^2 \leq {L}, 
\end{equation}
{but we are free to increase the value of $L$ if needed.}
Then, {by 
\eqref{eq:starstarstar},}
\begin{equation*}
%\label{eq:proof:stab:3}
\biggl\vvvert\int_{t_{i}}^{\cdot}  \textrm{F} \bigl(X_{r}(\omega),Y_{r}(\cdot)\bigr) d {\boldsymbol W}_{r}(\omega) \biggr\vvvert_{[t_{i},t_{i+1}],w,p} \leq (1+ {C_{\Lambda,\Lambda}}) \gamma. 
\end{equation*}
Hence, changing $\gamma$ into $(1+ {C_{\Lambda,\Lambda}}) \gamma$, 
%{(in which case $\gamma$ is allowed to depend on $\Lambda$)}, 
\begin{equation}
\label{eq:proof:stab:3}
\begin{split}
&\biggl\vvvert \int_{t_{i}}^{\cdot} \textrm{F} \bigl(X_{r}(\omega),Y_{r}(\cdot)\bigr) d {\boldsymbol W}_{r}(\omega) \biggr\vvvert_{[t_{i},t_{i+1}],w,p}^2 \leq  \gamma^2 < 
{{L}}, 
\end{split}
\end{equation}
if $L> { \gamma^2}$, in which case $\Gamma\big(\omega, X(\omega), Y(\cdot)\big)$ satisfies \eqref{eq:invariant:2}.  
{This completes the proof of the first bullet point in the conclusion of the statement.} 

\textcolor{gray}{$\bullet$} We now use a concatenation argument to get an estimate on the whole interval $[0,T]$. For all $s<t$ in $[0,T]$, we have
\begin{align}
&\Bigl| \bigl[\Gamma\bigl(\omega ,  X(\omega) , Y(\cdot)\bigr)\bigr]_{s,t} \Bigr|  
\\
&\leq \sum_{j=0}^N \Bigl| \bigl[\Gamma\bigl(\omega ,  X(\omega) ,  Y(\cdot)\bigr)\bigr]_{t_{j}',t_{j+1}'} \Bigr|  
\nonumber
 \\
&\leq \gamma \,\sum_{j =0}^N w\bigl(t_{j}',t_{j+1}',\omega\bigr)^{1/p}  \nonumber
\\
&  \leq \gamma \,\left(\sum_{j =0}^N w(t_{j}',t_{j+1}',\omega) \right)^{1/p} \, \bigl(N + 1\bigr)^{(p-1)/p}   
\leq \gamma \, w(s,t,\omega)^{1/p}\bigl(N+ 1\bigr)^{(p-1)/p},
\nonumber
\end{align}
where 
\textcolor{black}{we let}
$t_{i}' = \max(s,\min(t,t_{i}))$
\textcolor{black}{and where used the super-additivity of $w$ in the last line}. In the same way,
\begin{equation}
\label{eq:deltax:N:p-1}
\begin{split}
\Big| \delta_{x}\bigl[ \Gamma\bigl(\omega ,  X(\omega) ,  Y(\cdot)\bigr)\bigr]_{s,t} \Big| & \leq  \gamma\, w(s,t,\omega)^{1/p}\,\big(N+ 1\big)^{(p-1)/p}. 
\end{split}
\end{equation}
Setting, abusively, {$
\textrm{F}(\omega,\cdot) := \big(\textrm{F}_{r}(\omega,\cdot)\big)_{0\leq r \leq T} := \big(\textrm{F}(X_{r}(\omega),Y_{r}(\cdot))\big)_{0 \leq r \leq T},$ 
we have
\begin{flalign}
R^{\Gamma}_{s,t}(\omega)  \nonumber 
&= \int_{s}^t \textrm{F}_{r}(\omega,\cdot)d {\boldsymbol W}_{r}(\omega) - \textrm{F}_{s}(\omega,\cdot) W_{s,t}(\omega)  \nonumber 
\\
&= \sum_{j=0}^N \biggl( \int_{t_{j}'}^{t_{j+1}'} \textrm{F}_{r}(\omega,\cdot) d {\boldsymbol W}_{r}(\omega)
 - \textrm{F}_{s}(\omega,\cdot) W_{t_{j}',t_{j+1}'} \biggr)   
\label{eq:R:Gamma}
\\
&= \sum_{j=0}^N \Bigl\{R_{t_{j}',t_{j+1}'}^{\Gamma}(\omega) + \bigl(\textrm{F}_{t_{j}'}(\omega,\cdot)-\textrm{F}_{s}(\omega,\cdot) \bigr) W_{t_{j}',t_{j+1}'}(\omega) \Bigr\}.\nonumber
\end{flalign}
{The most difficult term in \eqref{eq:R:Gamma} is $\sum_{j=0}^N \big(\textrm{F}_{t_{j}'}(\omega,\cdot)-\textrm{F}_{s}(\omega,\cdot)\big) \, W_{t_{j}',t_{j+1}'}(\omega)$. We notice that  $\textrm{F}_{t_{j}'}(\omega,\cdot)-\textrm{F}_{s}(\omega,\cdot)=
\delta_{x}[\Gamma(\omega,X(\omega),Y(\cdot))]_{s,t_{j'}}$ , for $j=0,\cdots,N,$ can be bounded by 
$\gamma (N+1)^{(p-1)/p} w(s,t_{j}',\omega)^{1/p}$, see 
\eqref{eq:deltax:N:p-1}.}
%
%${\Lambda} \sum_{i=0}^{j-1}  \bigl(\vert X_{t_{i}',t_{i+1}'}(\omega) \vert + \big\langle Y_{t_{i}',t_{i+1}'}(\cdot) \big\rangle_{2}\bigr)$, since $\textrm{F}$ is ${\Lambda}$-Lipschitz. Then, $ \vert X_{t_{i}',t_{i+1}'}(\omega)  \vert $ is less than $\| X(\omega)\|_{[t_{i}',t_{i+1}'],w,p} \, w(t_{i}',t_{i+1}',\omega)^{1/p}$ and, following Lemma \ref{lem:1:1}, $\langle Y_{t_{i}',t_{i+1}'}(\cdot)  \rangle_{2} \leq 2 \, \big\langle \| Y(\cdot)\|_{[t_{i}',t_{i+1}'],w,p}\big\rangle_{8} \, w(t_{i}',t_{i+1}',\omega)^{1/p}$. Invoking the first bound in \eqref{eq:proof:stab} --this is the rationale for it-- together with \eqref{eq:proof:stab:2}, 
We deduce that the sum $\sum_{j=0}^N \big(\textrm{F}_{t_{j}'}(\omega,\cdot)-\textrm{F}_{s}(\omega,\cdot)\big) \, W_{t_{j}',t_{j+1}'}(\omega)$ is bounded by
\begin{equation*}
\begin{split}
&{\gamma (N+1)^{(p-1)/p} \,  w(s,t,\omega)^{1/p}} \,   \sum_{j=0}^N  w(t_{j}',t_{j+1}',\omega)^{1/p}  \leq  {\gamma} (N+1)^{2(p-1)/p} \, w(s,t,\omega)^{2/p}. 
\end{split}
\end{equation*}
To proceed with the other term in \eqref{eq:R:Gamma}, we note that the remainder term  $R_{t_{j}',t_{j+1}'}^{\Gamma}(\omega)$ can be also estimated by means of \eqref{eq:proof:stab:3}. We have 
$\vert R_{t_{j}',t_{j+1}'}^{\Gamma}(\omega) \vert \leq {\gamma} w(t_{j}',t_{j+1}',\omega)^{2/p}$. Since $2(p-1)/p=2-1/p \geq 1$, we deduce that there exists a constant $C_{\gamma}$ depending only on $\gamma$ such that }
\begin{equation*}
\big\vert R^\Gamma_{s,t}(\omega) \big\vert \leq C_{\gamma} \,(N+1)^{2(p-1)/p} \, w(s,t,\omega)^{2/p}.
%\bigl( 1+ L^{1/4} \bigr)\,(N+1)^{2(p-1)/p} \, w(s,t,\omega)^{2/p}.
\end{equation*}
Changing the value of $C_{\gamma}$ from line to line, we end up with 
\begin{equation*}
\begin{split}
%&\Bigl\| \Gamma\bigl(\omega ,  X(\omega) ,  Y(\cdot)\bigr) \Bigr\|_{[0,T],w,p}^2 \leq C_{\gamma}\,(N+1)^{2(p-1)/p},   \\
\Bigl\vvvert \Gamma\bigl(\omega , X(\omega),  Y(\cdot)\bigr) \Bigr\vvvert_{[0,T],w,p}^2 &\leq C_{\gamma}\,
%\bigl( 1+ \sqrt{L} \bigr)\,
(N+1)^{2(p-1)/p} 
\\
&\leq C_\gamma \, \bigl( 1 + N^{2(p-1)/p} \bigr). 
\end{split}
\end{equation*}
which proves the bound \eqref{eq:invariant:3} by choosing  
$
(t_{i})_{i=0,\cdots,N+1} = \big(\tau_{i}(0,T,\omega,1/(4L))\big)_{i=0,\cdots,N+1}
$, 
as defined in \eqref{eq:stopping:times}, and $N = N\big(\textcolor{black}{[0,T]},\omega,1/(4L)\big)$.  
{Recall that the above is true for $L > \gamma^2$.}  

\textcolor{gray}{$\bullet$} Assume now that $X(\omega)$ is the $\omega$-realization of a random controlled path $X(\cdot)=(X(\omega'))_{\omega' \in \Omega'}$ satisfying  \eqref{eq:invariant:2} for any $\omega'$, for \textcolor{black}{the} $\omega'$-dependent partition $(t_{i})_{\textcolor{black}{i=0,\cdots,N+1}}$. Then, {taking the fourth moment} with respect to $\omega$ the conclusion of the second point we get
\begin{equation*}
\begin{split}
%&\Bigl\langle  \Bigl\| \Gamma\bigl(\cdot ,  X(\cdot) ,  Y\bigr) \Bigr\|_{[0,T],w,p} \Bigr\rangle_{8}^2 \leq C_{\gamma}\,\Big\langle N\bigl(\textcolor{black}{[0,T]},\cdot,1/(4L)\bigr) +1\Big\rangle_{8}^{2(p-1)/p},   \\
&\Bigl\langle  \Bigl\vvvert \Gamma\bigl(\cdot ,  X(\cdot),  Y\bigr) \Bigr\vvvert_{[0,T],w,p} \Bigr\rangle_{8}^2 \leq  C_{\gamma}\,%\bigl( 1+ \sqrt{L} \bigr)\,
\Bigl( 1 +\Big\langle N\bigl(\textcolor{black}{[0,T]},\cdot,1/(4L)\bigr) \Big\rangle_{8}^{2(p-1)/p} \Bigr).
\end{split}
\end{equation*}
We get the conclusion of the statement if one assumes that 
$
\textcolor{red}{\big\langle N\big(\textcolor{black}{[0,T]},\cdot,1/(4L)\big) \big\rangle_{8}} \leq 1,
$ 
by choosing $L$ such that {$2\,C_{\gamma} \leq L$}.
%\sqrt{L}$ and $2\,C_{\gamma}\,(1+\sqrt{L}) \leq L$. 
\end{proof}

Remark that if $\big\langle N\bigl([0,1],\cdot,1/(4L)\bigr)\big\rangle_{8}$ is finite, then we can choose $T \leq 1$ small enough such that $\textcolor{red}{\big\langle N\big([0,T],\cdot,1/(4L)\big)\big\rangle_{8}}  \leq {1}$. (Since $N\bigl([0,t],\omega,1/(4L)\bigr)$ converges to $0$ as $t \searrow 0$, for any $\omega \in \Omega$, the result follows from dominated convergence.)

%%-------------------------------------------------------------%%
\subsection{Contractive Property of ${\bf \Gamma}$}
\label{SubsectionContractivity}
%%-------------------------------------------------------------%%

\begin{prop}
\label{thm:fixed:1}
Let \emph{F} satisfy \textbf{\textbf{Regularity assumptions 1}}  and \textbf{\textbf{Regularity assumptions 2}} and $w$ be a control 
satisfying \eqref{eq:w:s:t:omega:ineq}
and 
\eqref{eq:useful:inequality:wT}. Consider two $\omega$-controlled paths $X(\omega)$ and $X'(\omega)$, defined on a time interval $[0,T]$, together with two random controlled paths $Y(\cdot)$ and $Y'(\cdot)$, 
\textcolor{red}{all of them} satisfying
{\eqref{EqZeroConditionMu} together with}
\begin{equation}
\label{eq:fixed:00}
\bigl|\delta_{x} X(\omega) \bigr| \vee \bigl| \delta_{x} X'(\omega)\bigr| \vee \big\langle \delta_{x} Y(\cdot)\big\rangle_{\infty} \vee  \big\langle \delta_{x} Y'(\cdot) \big\rangle_{\infty} \leq {\Lambda},
\end{equation}
together with the size estimates
\begin{equation}
\label{eq:fixed:1}
\begin{split}
&% \bigl\langle \| Y(\cdot) \|_{[0,T],w,p} \bigr\rangle_{8}^2 \leq \sqrt{L_{0}}, \qquad 
\bigl\langle \vvvert Y(\cdot) \vvvert_{[0,T],w,p} \bigr\rangle_{8}^2 \leq L_{0},   \\
&% \bigl\langle  \| Y'(\cdot) \|_{[0,T],w,p} \bigr\rangle_{8}^2 \leq \sqrt{L_{0}},  \qquad   
\bigl\langle \vvvert Y'(\cdot) \vvvert_{[0,T],w,p} \bigr\rangle_{8}^2 \leq L_{0}, 
\end{split}
\end{equation}
and 
\begin{equation}
\label{eq:fixed:2}
\big\vvvert X(\omega) \big\vvvert_{[t_{i}^0,t_{i+1}^0],w,p}^2 \leq {L_{0}}, \qquad \big\vvvert X'(\omega) \big\vvvert_{[t_{i}^0,t_{i+1}^0],w,p}^2 \leq {L_{0}},  
\end{equation}
for $i \in \{0,\cdots,N^{0}\}$, for {some $L_0 \geq 1$}, with $N^0=N\bigl(\textcolor{black}{[0,T]},\omega,1/(4L_{0})\bigr)$ given by \eqref{eq:N:s:t:omega}, and for the sequence $\bigl(t_{i}^0=\tau_{i}(0,T,\omega,1/(4L_{0}))\bigr)_{i=0,\cdots,N^0+1}$ given by \eqref{eq:stopping:times}. 

Then, we can find a constant $\gamma$, only depending on $L_{0}$ {and $\Lambda$}, such that, for any partition $(t_{i})_{\textcolor{black}{i=0,\cdots,N}}$ refining\footnote{{This means that $(t_{i})_{i=0,\cdots,N}$ is included in $(t_{i}^0)_{i=0,\cdots,N^0}$}} $(t_i^0)_{\textcolor{black}{i=0,\cdots,N^0}}$ and satisfying $w(t_{i},t_{i+1},\omega)^{1/p} \leq 1/(4L)$ for some \textcolor{black}{$L\geq L_0$}, we have
\begin{equation*}
\begin{split}
&\biggl\vvvert \int_{t_{i}}^{\cdot} \textrm{F}\bigl(X_{r}(\omega),Y_{r}(\cdot)\bigr) d {\boldsymbol W}_{r}(\omega) - \int_{t_{i}}^{\cdot} \textrm{F}\bigl(X_{r}'(\omega),Y_{r}'(\cdot)\bigr) d {\boldsymbol W}_{r}(\omega) \biggr\vvvert_{[t_{i},t_{i+1}],w,p}   \\
&\leq \gamma\, w(0,t_{i},\omega)^{1/p}\,\big(1+\textcolor{black}{\frac1{4L}}\big)\,\Big(\big\vvvert \Delta X(\omega) \big\vvvert_{[0,t_{i}],w,p} + \big\langle \vvvert \Delta Y(\cdot) \vvvert_{[0,T],w,p} \big\rangle_{8} \Big)   \\
&\hspace{30pt} + \frac{\gamma}{4L}  \Big(\big\vvvert \Delta X(\omega)\big\vvvert_{[t_{i},t_{i+1}],w,p} + \big\langle \vvvert \Delta Y(\cdot) \vvvert_{[0,T],w,p}\big\rangle_{8} \Big),
\end{split}
\end{equation*}\textcolor{black}{where
$
\Delta X_{t}(\omega) := X_{t}(\omega) - X_{t}'(\omega)$, $\Delta Y_{t}(\cdot) := Y_{t}(\cdot) - Y_{t}'(\cdot), \quad t \in [0,T].$
}
\end{prop}

\begin{proof}
We get the conclusion after four  steps. Following the statement, we are given a subdivision $(t_{i})_{i=0,\cdots,N+1}$ of $[0,T]$ such that $w(t_{i},t_{i+1},\omega)^{1/p} \leq 1/(4L)$, for a frozen $\omega \in \Omega$
{and for 
 $L\geq L_0$}. We  assume that $(t_{i})_{i=0,\cdots,N+1}$ refines the subdivision $\bigl(t_{i}^0=\tau_{i}(0,T,\omega,1/(4L_{0}))\bigr)_{i=0,\cdots,N^0+1}$, where $N^0(\omega)=N\big([0,T],\omega,1/(4L_{0})\big)$. Like in the first step of the proof of \textcolor{black}{Proposition} \ref{thm:main:1} ({see in particular footnote${}^{\ref{footnote:7}}$}), we start from the estimate
\begin{equation*}
%\label{eq:contraction:proof:0}
\begin{split}
&\biggl\vvvert \int_{t_{i}}^{\cdot} \textrm{F}\bigl(X_{r}(\omega),Y_{r}(\cdot)\bigr) d {\boldsymbol W}_{r}(\omega) - \int_{t_{i}}^{\cdot} \textrm{F}\bigl(X_{r}'(\omega),\textcolor{black}{Y_{r}'}(\cdot)\bigr) d {\boldsymbol W}_{r}(\omega) \biggr\vvvert_{[t_{i},t_{i+1}],w,p}   
\\
&\leq \gamma \, { \Big(
\sup_{s \in [t_{i},t_{i+1}]} \bigl\vert  F(s,X_{s}(\omega),Y_{s}(\cdot)\bigr) 
- F(s,X_{s}'(\omega),Y_{s}'(\cdot)\bigr) \bigr\vert} 
\\
&\hspace{30pt}+ {  \sup_{s \in [t_{i},t_{i+1}]} \bigl\vert \delta_{x} \bigl[F(s,X_{s}(\omega),Y_{s}(\cdot) \bigr) 
- F(s,X_{s}'(\omega),Y_{s}'(\cdot) \bigr) \bigr] \bigr\vert }
\\
&\hspace{30pt} 
{ + \sup_{s \in [t_{i},t_{i+1}]} \Bigl\langle \delta_{\mu} \bigl[F(s,X_{s}(\omega),Y_{s}(\cdot) \bigr) 
- F(s,X_{s}'(\omega),Y_{s}'(\cdot) \bigr)  \bigr]
\Bigr\rangle_{4/3}
\Bigr)}
\\ 
&\hspace{15pt} + \gamma  \, w(t_{i},t_{i+1},\omega)^{1/p} \, \bigl\vvvert \textrm{F}\textcolor{black}{\bigl(}X(\omega),Y(\cdot)\textcolor{black}{\bigr)} - \textrm{F}\textcolor{black}{(}X'(\omega),Y'(\cdot) \textcolor{black}{\bigr)}\bigr\vvvert_{[t_{i},t_{i+1}],w,p}, 
\end{split}
\end{equation*}
for a universal constant $\gamma \geq 1$. {Modifying the constant $\gamma$ if necessary, we may easily change 
$s$ into $t_i$ in the first three lines of the right-hand side. We obtain
\begin{equation}
\label{eq:contraction:proof:0}
\begin{split}
&\biggl\vvvert \int_{t_{i}}^{\cdot} \textrm{F}\bigl(X_{r}(\omega),Y_{r}(\cdot)\bigr) d {\boldsymbol W}_{r}(\omega) - \int_{t_{i}}^{\cdot} \textrm{F}\bigl(X_{r}'(\omega),\textcolor{black}{Y_{r}'}(\cdot)\bigr) d {\boldsymbol W}_{r}(\omega) \biggr\vvvert_{[t_{i},t_{i+1}],w,p}   
\\
&\leq \gamma \, {  \Big(
 \bigl\vert  F\bigl(X_{t_{i}}(\omega),Y_{t_{i}}(\cdot)\bigr) 
- F\bigl(X_{t_{i}}'(\omega),Y_{t_{i}}'(\cdot)\bigr)  \bigr\vert} 
\\
&\hspace{30pt}+     \bigl\vert \delta_{x} \bigl[F(X_{t_{i}}(\omega),Y_{t_{i}}(\cdot) \bigr) 
- F\bigl(X_{t_{i}}'(\omega),Y_{t_{i}}'(\cdot) \bigr)\bigr] \bigr\vert 
\\
&\hspace{30pt} 
 +  \Bigl\langle \delta_{\mu} \bigl[F\bigl(X_{t_{i}}(\omega),Y_{t_{i}}(\cdot) \bigr) 
- F\bigl(X_{t_{i}}'(\omega),Y_{t_{i}}'(\cdot) \bigr)   \bigr]
\Bigr\rangle_{4/3}
\Bigr)
\\ 
&\hspace{15pt} + \gamma  \, w(t_{i},t_{i+1},\omega)^{1/p} \, \bigl\vvvert \textrm{F}\textcolor{black}{\bigl(}X(\omega),Y(\cdot)\textcolor{black}{\bigr)} - \textrm{F}\textcolor{black}{\bigl(}X'(\omega),Y'(\cdot) \textcolor{black}{\bigr)}\bigr\vvvert_{,[t_{i},t_{i+1}],w,p}, 
\end{split}
\end{equation}
}

The first point is to bound the quantity $\textcolor{black}{\bigl\vvvert} \textrm{F} \textcolor{black}{\bigl(}X(\omega),Y(\cdot)\textcolor{black}{\bigr)} - \textrm{F}\textcolor{black}{\bigl(}X'(\omega),Y'(\cdot) \textcolor{black}{\bigr)} \textcolor{black}{\bigr \vvvert}_{\textcolor{black}{\star},[t_{i},t_{i+1}],w,p}$, which contains all the terms that appear in the above inequality.   \vskip 4pt

\textbf{\textbf{Step 1.}} We first analyse the term 
\begin{equation*}
\begin{split}
\Delta \textrm{F} (\omega,\cdot ) &:= \textrm{F}\textcolor{black}{\bigl(}X(\omega),Y(\cdot)\textcolor{black}{\bigr)}- \textrm{F}\textcolor{black}{\bigl(}X'(\omega),Y'(\cdot)\textcolor{black}{\bigr)}  
\\
&:= \Bigl(\textrm{F}\textcolor{black}{\bigl(}X_{t}(\omega),Y_{t}(\cdot)\textcolor{black}{\bigr)}- \textrm{F}\textcolor{black}{\bigl(}X_{t}'(\omega),Y_{t}'(\cdot)\textcolor{black}{\bigr)} \Bigr)_{0 \le t \le T}.
\end{split}
\end{equation*} 

\textcolor{gray}{$\bullet$} \textbf{Initial condition of $\Delta \textrm{F}(\omega,\cdot)$ -- } As $\big\vert [\Delta \textrm{F}(\omega,\cdot)]_{t_{i}} \big\vert \leq  {\Lambda} \big( \vert \Delta X_{t_{i}}(\omega) \vert + \langle \vert \Delta Y_{t_{i}}(\cdot) \vert \rangle_{2}\big)$, we have, from Lemma \ref{lem:1:1} \textcolor{black}{and from the two identities $\Delta X_{0}(\omega) =0$ and  $\Delta Y_{0}(\cdot) = 0$},
\begin{equation*}
\big\vert [\Delta \textrm{F}(\omega,\cdot)]_{t_{i}} \big\vert \leq 2 {\Lambda} \, w(0,t_{i},\omega)^{1/p} \, \Bigl( \vvvert \Delta X(\omega) \vvvert_{[0,t_{i}],w,p} + \bigl\langle \vvvert \Delta Y(\cdot) \vvvert_{[0,t_{i}],w,p} \bigr\rangle_{4} \Bigr).
\end{equation*}

\vskip 4pt

\textcolor{gray}{$\bullet$} \textbf{Variation of $\Delta \textrm{F}(\omega,\cdot)$.} Using the notations  \eqref{eq:compactified:notation:interpolation} together with similar ones for the processes tagged with a \textit{prime}, we have
\begin{equation*}
\begin{split}
&\big [\Delta \textrm{F}(\omega,\cdot)\big]_{s,t} 
\\
&= \int_{0}^1 \Big\{ \partial_{x} \textrm{F}\Bigl( \textcolor{black}{X_{s;(s,t)}^{(\lambda)}}(\omega),\textcolor{black}{Y_{s;(s,t)}^{(\lambda)}}(\cdot) \Bigr) X_{s,t}(\omega)   
-  \partial_{x} \textrm{F}\Bigl( \textcolor{black}{X_{s;(s,t)}^{(\lambda)\prime}}(\omega),
\textcolor{black}{Y_{s;(s,t)}^{(\lambda)\prime}(\cdot) } \Bigr) X_{s,t}'(\omega) \Big\} d\lambda   \\
&\hspace{3pt} + \int_{0}^1 {\mathbb E}\Big\{  \nabla_{Z}\textrm{F}\Bigl(
X_{s;(s,t)}^{(\lambda)}(\omega) ,
Y_{s;(s,t)}^{(\lambda)}(\cdot) \Bigr) \, Y_{s,t}(\cdot)   -   
\nabla_{Z} \textrm{F}\Bigl(
\textcolor{black}{X_{s;(s,t)}^{(\lambda)\prime}}(\omega) ,
\textcolor{black}{Y_{s;(s,t)}^{(\lambda)\prime}}(\cdot)
\Bigr) \, Y_{s,t}'(\cdot) \Big\} d\lambda.
\end{split}
\end{equation*} 
We now use the following three facts. First,   $X_{0}(\omega)=X_{0}'(\omega)$ and $Y_{0}(\cdot) = Y_{0}'(\cdot)$; second, from \textbf{\textbf{Regularity assumptions 1}}, for any $x \in \RR^d$ and $Z \in \LL^2(\Omega,{\mathcal F},\PP;\RR^d)$,  $\vert \partial_{x} \textrm{F}(x,Z) \vert$ and $\bigl\langle \nabla_{Z} \textrm{F}(x,Z) \rangle_{2}$ are bounded by {$\Lambda$}; last, $(x,Z) \mapsto \partial_{x} \textrm{F}(x,Z)$ and $(x,Z) \mapsto \nabla_{Z} \textrm{F}(x,Z)$ are {$\Lambda$}-Lipschitz continuous. Hence, {allowing $\gamma$ to depend on $\Lambda$ and to increase from line to line}, we get, for $s,t$ in the interval $[t_{i},t_{i+1}]$, 
\begin{equation*}
\begin{split}
\bigl\vert  [\Delta \textrm{F}(\omega,\cdot)]_{s,t} \bigr\vert &\leq  {\Lambda} \, \Bigl( \big\vert \Delta X_{s,t}(\omega) \big\vert + \big\langle 
 \Delta Y_{s,t}(\cdot)  \big\rangle_{2} \Bigr)   
 \\
&\hspace{15pt} + {\Lambda} \Bigl( \vert X_{s,t}(\omega) \vert + \big\langle Y_{s,t}(\cdot) \big\rangle_{2} \Bigr)    
\\
&\hspace{30pt} \times \Big\{ \vert \Delta X_{s}(\omega) \vert + \langle  \Delta Y_{s}(\cdot)  \rangle_{2} + \vert \Delta X_{s,t}(\omega) \vert
+ \big\langle \Delta Y_{s,t}(\cdot) \big\rangle_{2} \Big\}   \\
&\leq \textcolor{black}{\textbf{(}{\boldsymbol a}\textbf{)} +
\textbf{(}{\boldsymbol b}\textbf{)}},
\end{split}
\end{equation*}
{\color{black} where
$\textbf{(}{\boldsymbol a}\textbf{)} := \gamma \,w(s,t,\omega)^{1/p} \, \Bigl( \vvvert \Delta X(\omega) \vvvert_{[t_{i},t_{i+1}],w,p} + \big\langle \vvvert \Delta \textcolor{black}{Y}(\cdot) \vvvert_{[t_{i},t_{i+1}],w,p} \big\rangle_{4} \Bigr)$,
and $\textbf{(}{\boldsymbol b}\textbf{)} = \textbf{(}{\boldsymbol b}_{\textbf{1}}\textbf{)} \times 
\textbf{(}{\boldsymbol b}_{\textbf{2}}\textbf{)}$ with
\begin{equation*}
\begin{split}
\textbf{(}{\boldsymbol b}_{\textbf{1}}\textbf{)} &:= 
\gamma  \, w(s,t,\omega)^{1/p} \, \Bigl( \vvvert X(\omega) \vvvert_{[t_{i},t_{i+1}],w,p} + \bigl\langle \vvvert Y(\cdot) \vvvert_{[t_{i},t_{i+1}],w,p} \bigr\rangle_{4} \Bigr)
\\
\textbf{(}{\boldsymbol b}_{\textbf{2}}\textbf{)} &:= w(0,t_{i},\omega)^{1/p} \, \Bigl( \vvvert \Delta X(\omega) \vvvert_{[0,t_{i}],w,p} + \big\langle \vvvert \Delta Y(\cdot) \vvvert_{[0,t_{i}],w,p} \big\rangle_{4} \Bigr)   \\
&\quad + w(t_{i},t_{i+1},\omega)^{1/p} \, \Bigl( \vvvert \Delta X(\omega) \vvvert_{[t_{i},t_{i+1}],w,p} + \big\langle  \vvvert \Delta Y(\cdot) \vvvert_{[t_{i},t_{i+1}],w,p} \big\rangle_{4} \Bigr).
\end{split}
\end{equation*}
}It follows that we have
\begin{equation*}
\begin{split}
\bigl\| \Delta \textrm{F}(\omega,\cdot)\|_{[t_{i},t_{i+1}],w,p} &\leq \gamma \Bigl( \vvvert \Delta X(\omega) \vvvert_{[t_{i},t_{i+1}],w,p} + \bigl\langle
\vvvert \Delta Y(\cdot) \vvvert_{[t_{i},t_{i+1}],w,p}  \bigr\rangle_{4} \Bigr)   \\
&\hspace{15pt} + \gamma \Bigl( \vvvert X(\omega) \vvvert_{[t_{i},t_{i+1}],w,p} + 
\bigl\langle \vvvert Y(\cdot) \vvvert_{[t_{i},t_{i+1}],w,p} \bigr\rangle_{4} \Bigr)
\times 
\textcolor{black}{\textbf{(}{\boldsymbol b}_{\textbf{2}}\textbf{)}}.
\end{split}
\end{equation*}
Allowing the constant $\gamma$ to depend on {$L_{0}$ and $\Lambda$}, and using \eqref{eq:fixed:1} and \eqref{eq:fixed:2} {together with the bound $w(t_{i},t_{i+1},\omega)^{1/p} \leq 1/(4L)$}, we get
\begin{equation*}
\begin{split}
\bigl\|\Delta \textrm{F}(\omega,\cdot)\|_{[t_{i},t_{i+1}],w,p} &\leq \gamma \Bigl( \vvvert \Delta X(\omega) \vvvert_{[t_{i},t_{i+1}],w,p} + \bigl\langle
\vvvert \Delta \textcolor{black}{Y}(\cdot) \vvvert_{[t_{i},t_{i+1}],w,p} \bigr\rangle_{4} \Bigr)   
\\
&\ \ + \gamma\, w(0,t_{i},\omega)^{1/p}\, \Bigl( \vvvert \Delta X(\omega) \vvvert_{[0,t_{i}],w,p} + \big\langle  \vvvert \Delta Y(\cdot) \vvvert_{[0,t_{i}],w,p} \big\rangle_{4} \Bigr).
\end{split}
\end{equation*}
\vskip 4pt

\textbf{\textbf{Step 2 -- }} \textit{We now handle the \textcolor{black}{\textit{Gubinelli} derivative} $\delta_{x} \textcolor{black}{[} \Delta \textrm{\rm F}(\omega,\cdot) \textcolor{black}{]}$.} We start from 
\begin{equation}
\label{eq:step:2:ci}
\begin{split}
\textcolor{black}{\delta_{x}[\Delta \textrm{F}(\omega,\cdot)]_{t}} &= \bigl[ 
\textcolor{black}{
\partial_{x}\textrm{F} \bigl(X_{t}(\omega),Y_{t}(\cdot) \bigr) 
- 
\partial_{x}\textrm{F} \bigl(X_{t}'(\omega),Y_{t}'(\cdot) \bigr) }
\bigr] \,\delta_{x} X_{t}(\omega)
\\
&\hspace{15pt} +  \textcolor{black}{\partial_{x}\textrm{F}\bigl(X_{t}'(\omega),Y_{t}'(\cdot)\bigr)} \,\Delta \delta_{x} X_{t}(\omega).
\end{split}
\end{equation}

\textcolor{gray}{$\bullet$} \textbf{Initial condition of $\delta_{x} \big[\Delta \textrm{F}(\omega,\cdot) \big]$.} 
By \textbf{\textbf{Regularity assumptions 1}}, \eqref{eq:fixed:00} {and the fact that 
$\Delta \delta_{x} X_{t}=\delta_{x} \Delta X_{t}$},  
\begin{equation*}
\begin{split}
\bigl\vert \delta_{x}\bigl[\Delta \textrm{F}(\omega,\cdot)\bigr]_{t_{i}} \bigr\vert &\leq  
{\Lambda}
\Bigl(
\big\vert \delta_{x} \Delta X_{t_{i}}(\omega) \big\vert + \big\vert \Delta X_{t_{i}}(\omega) \big\vert + \big\langle \Delta Y_{t_{i}}(\cdot)\big\rangle_{2}
\Bigr)
   \\
&\leq \gamma \, w(0,t_{i},\omega)^{1/p}\, \Bigl( \vvvert \Delta X(\omega) \vvvert_{[0,t_{i}],w,p} + \bigl\langle \vvvert \Delta Y(\textcolor{black}{\cdot})\vvvert_{[0,t_{i}],w,p} \bigr\rangle_{4} \Bigr).
\end{split}
\end{equation*}
\vskip 4pt

\textcolor{gray}{$\bullet$} \textbf{Variation of $\partial_{x} \big[ \Delta \textrm{F}(\omega,\cdot) \big]$.} Similarly, 
{using formula 
\eqref{eq:step:2:ci},} we get 
\begin{equation}
\label{eq:contraction:proof:1}
\begin{split}
\Bigl\vert 
 \delta_{x} \bigl[\Delta \textrm{F}(\omega,\cdot) \bigr]_{s,t} \Bigr\vert &\leq 
 {\Lambda}
 \bigl\vert [\delta_{x} X(\omega)]_{s,t} \bigr\vert \,\Bigl( \vert \Delta X_{s}(\omega) \vert + \bigl\langle \Delta Y_s(\cdot)  \bigr\rangle_{2} \Bigr)   \\
&\hspace{5pt}+ {\Lambda} \Bigl\vert \bigl[\textcolor{black}{ \partial_{x} \textrm{F} \bigl(X(\omega),Y(\cdot) \bigr) - \partial_{x} \textrm{F} \bigl(X'(\omega),Y'(\cdot) \bigr) } \bigr]_{s,t} \Bigr\vert   
\\
&\hspace{5pt} + {\Lambda} \Bigl\vert \bigl[\Delta \delta_{x} X(\omega)\bigr]_{s,t} \Bigr\vert + {\Lambda} \bigl\vert \Delta \delta_{x} X_{s}(\omega) \bigr\vert \, \Bigl\vert  \bigl[\partial_{x} \textrm{F}\textcolor{black}{\bigl(X'(\omega),Y'(\cdot)\bigr)} \bigr]_{s,t} \Bigr\vert. 
\end{split}
\end{equation}
The second term in the right-hand side is handled as $[\Delta \textrm{F}(\omega,\cdot)]_{s,t}$ in the first step, with $s,t$ in $[t_{i},t_{i+1}]$. {By the aforementioned identity} $\Delta \delta_{x} X(\omega) = \delta_{x} \Delta X(\omega)$, the third term is less than ${\Lambda} w(s,t,\omega)^{1/p}\, \vvvert \Delta X(\omega) \vvvert_{[t_{i},t_{i+1}],w,p}$. The term $\bigl\vert \Delta \delta_{\textcolor{black}{x}} X_{s}(\omega) \bigr\vert \bigl\vert [\partial_{x} \textrm{F}(\textcolor{black}{X'(\omega),Y'(\cdot)})]_{s,t} \bigr\vert$ is less than 
\begin{align}
&\gamma \,w(s,t,\omega)^{1/p} \, \Bigl( w(0,t_{i},\omega)^{1/p} \vvvert  \Delta X(\omega) \vvvert_{[0,t_{i}],w,p} + w(t_{i},t_{i+1},\omega)^{1/p} \vvvert \Delta X(\omega) \vvvert_{[t_{i},t_{i+1}],w,p} \Bigr)   
\nonumber
\\
&\hspace{15pt} \times \Bigl( \vvvert \textcolor{black}{X'}(\omega) \vvvert_{[t_{i},t_{i+1}],w,p} + \big\langle \vvvert 
\textcolor{black}{Y'}(\cdot) \vvvert_{[t_{i},t_{i+1}],w,p} \big\rangle_{4} \Bigr)   
\label{eq:step:2:variation}
\\
&\leq \gamma \, w(s,t,\omega)^{1/p} \, \Bigl( w(0,t_{i},\omega)^{1/p} \vvvert \Delta X(\omega) \vvvert_{[0,t_{i}],w,p} + \vvvert \Delta X(\omega) \vvvert_{[t_{i},t_{i+1}],w,p} \Bigr), \nonumber
\end{align}
where we used again \eqref{eq:fixed:1} and \eqref{eq:fixed:2}. \textcolor{black}{Now, the first term in \eqref{eq:contraction:proof:1} 
is less than}
{\color{black}
 \begin{equation*}
\begin{split} 
&\gamma\, w(s,t,\omega)^{1/p}\, \vvvert X \vvvert_{[t_{i},t_{i+1}],w,p} 
\Big\{ w(0,t_{i},\omega)^{1/p}\Bigl( \vvvert \Delta X(\omega) \vvvert_{[0,t_{i}],w,p} + \bigl\langle \vvvert \Delta Y(\cdot) \vvvert_{[0,t_{i}],w,p} \bigr\rangle_{4}\Bigr)   \\
&\hspace{30pt}+ w(t_{i},t_{i+1},\omega)^{1/p}\, \Bigl(\vvvert \Delta X(\omega) \vvvert_{[t_{i},t_{i+1}],w,p} + \bigl\langle \vvvert \Delta Y(\cdot) \vvvert_{[t_{i},t_{i+1}],w,p} \bigr\rangle_{4} \Bigr) \Big\}.  
\end{split}
\end{equation*}   
}Hence, by 
\eqref{eq:fixed:2} and the fact that {$w(t_{i},t_{i+1},\omega)^{1/p} \leq 1/(4L)$},
\begin{equation*}
\begin{split} 
&\bigl\vert [\delta_{x} X(\omega)]_{s,t} \bigr\vert \Bigl( \vert \Delta X_{s}(\omega) \vert + \langle \vert \Delta Y_s(\cdot) \vert \rangle_{2} \Bigr)   \\
&\hspace{0pt} \leq \gamma \,w(s,t,\omega)^{1/p} \, \Big\{ w(0,t_{i},\omega)^{1/p} \,\Bigl( \vvvert \Delta X(\omega) \vvvert_{[0,t_{i}],w,p} +\big\langle \vvvert \Delta Y(\cdot) \vvvert_{[0,t_{i}],w,p} \big\rangle_{4} \Bigr)   \\
&\hspace{15pt}+ \Bigl( \vvvert \Delta X(\omega) \vvvert_{[t_{i},t_{i+1}],w,p} + \big\langle \vvvert \Delta Y(\cdot) \vvvert_{[t_{i},t_{i+1}],w,p} \big\rangle_{4} \Bigr) \Big\}.  
\end{split}
\end{equation*}
So, the final bound for $\bigl\| \delta_{x} \textcolor{black}{\bigl[} \Delta \textrm{F}(\omega,\cdot) \textcolor{black}{\bigr]} \bigr\|_{[t_{i},t_{i+1}],w,p}$ is 
\begin{equation*}
\begin{split}
&\gamma\, \Bigl( \vvvert \Delta X(\omega) \vvvert_{[t_{i},t_{i+1}],w,p} + \big\langle \vvvert \Delta Y(\cdot) \vvvert_{[t_{i},t_{i+1}],w,p}\big\rangle_{4} \Bigr)   \\
&\hspace{15pt} + \gamma  \, w(0,t_{i},\omega)^{1/p} \, \Bigl( \vvvert \Delta X(\omega) \vvvert_{[0,t_{i}],w,p} + \big\langle \vvvert \Delta Y(\cdot) \vvvert_{[0,t_{i}],w,p} \big\rangle_{4}\Bigr),
\end{split}
\end{equation*}
which yields the same bound as in the first step.

\textbf{\textbf{Step 3 -- }} We now handle the other Gubinelli derivative $\delta_{\mu} \big[\Delta \textrm{F}(\omega,\cdot)
\big]$, for which we have
\begin{equation*}
\begin{split}
\delta_{\mu}\big[\Delta \textrm{F}(\omega,\cdot)\big]_{t} &= \Bigl[ \nabla_{Z}\textrm{F} \bigl(X_{t}(\omega),
Y_{t}(\cdot)\bigr) - \nabla_{Z}\textrm{F} \bigl(X_{t}'(\omega),Y_{t}'(\cdot)\bigr) \Bigr] \,\delta_{x} Y_{t}(\cdot)   \\
&\hspace{15pt} +  \textcolor{black}{\nabla_{Z}\textrm{F}\bigl(X_{t}'(\omega),Y_{t}'(\cdot)\bigr)} \,\Delta \delta_{x} Y_{t}(\cdot).
\end{split}
\end{equation*}

\textcolor{gray}{$\bullet$} \textbf{Initial condition of $\delta_{\mu} \big[\Delta \textrm{F}(\omega,\cdot)\big]$.} Proceeding as before,
\begin{equation*}
\begin{split}
\textcolor{black}{\Big\langle}  \delta_{\mu} [\Delta \textrm{F}(\omega,\cdot)]_{t_{i}} 
\textcolor{black}{
\Big\rangle_{4/3}} &\leq {\Lambda} \Bigl( \big\vert \Delta X_{t_{i}}(\omega) \big\vert + 
\big\langle \Delta Y_{t_{i}}(\cdot) \big\rangle_{4} + \big\langle \delta_{x} \Delta Y_{t_{i}}(\cdot) \big\rangle_{4} \Bigr)
   \\
&\leq \textcolor{black}{\gamma} \, w(0,t_{i},\omega)^{1/p} \, \Bigl( \vvvert \Delta X\textcolor{black}{(\omega)} \vvvert_{[0,t_{i}],w,p} + \big\langle \vvvert \Delta 
\textcolor{black}{Y(\cdot)} \vvvert_{[0,t_{i}],w,p} \big\rangle_{8}\Bigr), 
\end{split}
\end{equation*}
where we used the H\"older inequality
with exponents $3$ and $3/2$:
\begin{equation*}
\begin{split}
&{\mathbb E} \Bigl[ \bigl\vert \Delta \delta_{x} Y_{t}(\cdot)\bigr\vert^{4/3} \bigl\vert   
\textcolor{black}{\nabla_{Z}}\textrm{F}
\bigl(\textcolor{black}{X_{t}'(\omega)},Y_{t}'(\cdot)\bigr)  \bigr\vert^{4/3}
\Bigr]^{3/4} 
\\
&\leq {\mathbb E} \Bigl[ \bigl\vert \Delta \delta_{x} Y_{t}(\cdot)\bigr\vert^{4} \Bigr]^{1/4} {\mathbb E} \Bigl[ \bigl\vert
\textcolor{black}{\nabla_{Z}}\textrm{F}
\bigl(\textcolor{black}{X_{t}'(\omega)},Y_{t}'(\cdot)\bigr)
\bigr\vert^{2} \Bigr]^{1/2}.
\end{split}
\end{equation*}
\vskip 4pt
\textcolor{gray}{$\bullet$} \textbf{Variation of $\partial_{\mu} \textcolor{black}{[} \Delta \textrm{F}(\omega,\cdot) \textcolor{black}{]}$.} Following 
\eqref{eq:contraction:proof:1}
and using again H\"older inequality with exponents $3$ and $3/2$, 
\begin{align}
\textcolor{black}{\Bigl\langle \bigl[ \delta_{\mu} [\Delta \textrm{F}(\omega,\cdot)] \bigr]_{s,t} \Bigr\rangle_{4/3}} 
&\leq {\Lambda} \bigl\langle [\delta_{x} Y(\cdot)]_{s,t} \bigr\rangle_{4} \Bigl( \vert \Delta X_{s}(\omega) \vert + \big\langle \Delta Y_s(\cdot)\big\rangle_{2} \Bigr) 
\nonumber
\\
&\hspace{5pt}+ {\Lambda}
\textcolor{black}{\Bigl\langle 
\bigl[  \nabla_{Z} \textrm{F}\bigl(X(\omega),Y(\cdot)
\bigr) 
- 
\nabla_{Z} \textrm{F}\bigl(X'(\omega),Y'(\cdot)
\bigr) 
\bigr]_{s,t}\Bigr\rangle_{4/3}}   \label{EqEstimate}
\\
&\hspace{5pt} + {\Lambda} \bigl\langle [\Delta \delta_{x} Y(\cdot)]_{s,t}\bigr\rangle_{4}  
+ {\Lambda} \bigl\langle\Delta \delta_{x} Y_{s}(\cdot)\bigr\rangle_{4} \textcolor{black}{\Bigl\langle \bigl[\nabla_{Z} \textrm{F}\bigl(X'(\omega),Y'(\cdot)\bigr) \bigr]_{s,t}\Bigr\rangle_{2}}. 
\nonumber
\end{align}
As for the fourth term,
we get, 
{following \eqref{eq:step:2:variation}}, 
\begin{equation*}
\begin{split}
&\bigl\langle \Delta \delta_{x} Y_{s}(\cdot)\bigr\rangle_{4}   \textcolor{black}{\Bigl\langle \bigl[\nabla_{Z} \textrm{F}\bigl(X(\omega),Y(\cdot)\bigr) \bigr]_{s,t}\Bigr\rangle_{2}}   
\\
%&\leq \gamma w(s,t,\omega)^{1/p} \, \Bigl( \vvvert \Delta X(\omega) \vvvert_{[t_{i},t_{i+1}],w,p} + \big\langle \vvvert\Delta Y(\cdot)\vvvert_{[t_{i},t_{i+1}],w,p} \big\rangle_{4} \Bigr)   \\
%&\hspace{15pt} \times \Big\{ w(0,t_{i},\omega)^{1/p} \,\big\langle \vvvert \Delta Y(\cdot) \vvvert_{[0,t_{i}],w,p} \big\rangle_{8} + \big\langle \vvvert \Delta Y(\cdot) \vvvert_{[t_{i},t_{i+1}],w,p} \big\rangle_{8} \Big\}   \\
&\leq \gamma \, w(s,t,\omega)^{1/p} \, \Bigl( w(0,t_{i},\omega)^{1/p}\, \big\langle \vvvert \Delta Y(\cdot) \vvvert_{[0,t_{i}],w,p}
\big\rangle_{8} + \big\langle \vvvert \Delta Y(\cdot) \vvvert_{[t_{i},t_{i+1}],w,p} \big\rangle_{8} \Bigr).
\end{split}
\end{equation*}
Recalling that $\Delta \delta_{x} Y(\cdot) = \delta_{x} \Delta Y(\cdot)$, the third term in \eqref{EqEstimate} is less than $2 {\Lambda}w(s,t,\omega)^{1/p} \times \,\big\langle \vvvert \Delta Y(\cdot) \vvvert_{[t_{i},t_{i+1}],w,p} \big\rangle_{8}$. To handle the first term in \eqref{EqEstimate}, we 
proceed as in the second step:
\begin{equation*}
\begin{split}
&\bigl\langle [\delta_{x} Y(\cdot)]_{s,t} \bigr\rangle_{4} \, \Bigl( \vert \Delta X_{s}(\omega) \vert + 
\bigl\langle \textcolor{black}{ \Delta Y_s(\cdot) } \bigr\rangle_{2} \Bigr)   \\
&\leq \gamma \, w(s,t,\omega)^{1/p} \, \Big\{ w(0,t_{i},\omega)^{1/p} \textcolor{black}{\Bigl(} \vvvert \Delta X(\omega) \vvvert_{[0,t_{i}],w,p} + \big\langle \vvvert \Delta Y(\cdot) \vvvert_{[0,t_{i}],w,p}\big\rangle_{4} \textcolor{black}{\Bigr)}   \\
&\hspace{110pt} + \textcolor{black}{\Bigl(} \vvvert \Delta X(\omega) \vvvert_{[t_{i},t_{i+1}],w,p} + \big\langle \vvvert \Delta Y(\cdot) \vvvert_{[t_{i},t_{i+1}],w,p} \big\rangle_{8}\textcolor{black}{\Bigr)} \Big\}. 
\end{split}
\end{equation*}
As for the second term in \eqref{EqEstimate},  we write 
\textcolor{black}{$\bigl[  \nabla_{Z} \textrm{F}\bigl(X(\omega),Y(\cdot)
\bigr) 
- 
\nabla_{Z} \textrm{F}\bigl(X'(\omega),Y'(\cdot)\bigr) 
\bigr]_{s,t}$ in} the form
\textcolor{black}{$\bigl[ 
D_{\mu}\textrm{F}\bigl(X(\omega),Y(\cdot)\bigr)\bigl(Y(\cdot)\bigr) 
- 
D_{\mu} \textrm{F}\bigl(X'(\omega),Y'(\cdot)
\bigr)\bigl(Y'(\cdot)\bigr) 
\bigr]_{s,t}$}
and then expand it as
\begin{equation}
\label{eq:variation:step:3}
\begin{split}
&\int_{0}^1 \Big\{ \partial_{x} D_{\mu}\textrm{F}\textcolor{black}{\Bigl(}X_{s\textcolor{black}{;(s,t)}}^{\textcolor{black}{(\lambda)}}(\omega) ,
Y_{s\textcolor{black}{;(s,t)}}^{\textcolor{black}{(\lambda)}}(\cdot)
 \textcolor{black}{\Bigr)}\textcolor{black}{\Bigl(} Y_{s\textcolor{black}{;(s,t)}}^{\textcolor{black}{(\lambda)}}(\cdot)
 \textcolor{black}{\Bigr)} X_{s,t}(\omega)  
 \\
 &\qquad-  \partial_{x} D_{\mu}\textrm{F}\textcolor{black}{\Bigl(}X_{s\textcolor{black}{;(s,t)}}^{\textcolor{black}{(\lambda) \prime}}(\omega) ,
Y_{s\textcolor{black}{;(s,t)}}^{\textcolor{black}{(\lambda) \prime}}(\cdot)
 \textcolor{black}{\Bigr)}\textcolor{black}{\Bigl(} Y_{s\textcolor{black}{;(s,t)}}^{\textcolor{black}{(\lambda) \prime}}(\cdot)
 \textcolor{black}{\Bigr)} X_{s,t}^{\prime}(\omega)
 \Big\} d\lambda   
\\
&+ \int_{0}^1 \Big\{ \partial_{\textcolor{black}{z}} D_{\mu}\textrm{F}
\textcolor{black}{\Bigl(}X_{s\textcolor{black}{;(s,t)}}^{\textcolor{black}{(\lambda)}}(\omega) ,
Y_{s\textcolor{black}{;(s,t)}}^{\textcolor{black}{(\lambda)}}(\cdot)
 \textcolor{black}{\Bigr)}\textcolor{black}{\Bigl(} Y_{s\textcolor{black}{;(s,t)}}^{\textcolor{black}{(\lambda)}}(\cdot)
 \textcolor{black}{\Bigr)}
Y_{s,t}(\cdot)   \\
&\qquad-  \partial_{\textcolor{black}{z}} D_{\mu}\textrm{F}\textcolor{black}{\Bigl(}X_{s\textcolor{black}{;(s,t)}}^{\textcolor{black}{(\lambda) \prime}}(\omega) ,
Y_{s\textcolor{black}{;(s,t)}}^{\textcolor{black}{(\lambda) \prime}}(\cdot)
 \textcolor{black}{\Bigr)}\textcolor{black}{\Bigl(} Y_{s\textcolor{black}{;(s,t)}}^{\textcolor{black}{(\lambda) \prime}}(\cdot)
 \textcolor{black}{\Bigr)} Y_{s,t}'(\cdot) \Big\} d\lambda
 \\
&+ \int_{0}^1 \tilde{\EE}\Big\{D^2_{\mu}\textrm{F}
\textcolor{black}{\Bigl(}X_{s\textcolor{black}{;(s,t)}}^{\textcolor{black}{(\lambda)}}(\omega) ,
Y_{s\textcolor{black}{;(s,t)}}^{\textcolor{black}{(\lambda)}}(\cdot)
 \textcolor{black}{\Bigr)}\textcolor{black}{\Bigl(} Y_{s\textcolor{black}{;(s,t)}}^{\textcolor{black}{(\lambda)}}(\cdot), 
\tilde{Y}_{s\textcolor{black}{;(s,t)}}^{\textcolor{black}{(\lambda)} }
 \textcolor{black}{\Bigr)}
 \tilde Y_{s,t}(\cdot)   \\
&\qquad- 
\tilde{\EE}\Big\{D^2_{\mu}\textrm{F}
\textcolor{black}{\Bigl(}X_{s\textcolor{black}{;(s,t)}}^{\textcolor{black}{(\lambda)\prime}}(\omega) ,
Y_{s\textcolor{black}{;(s,t)}}^{\textcolor{black}{(\lambda)\prime}}(\cdot)
 \textcolor{black}{\Bigr)}\textcolor{black}{\Bigl(} Y_{s\textcolor{black}{;(s,t)}}^{\textcolor{black}{(\lambda)\prime}}(\cdot), 
\tilde{Y}_{s\textcolor{black}{;(s,t)}}^{\textcolor{black}{(\lambda)\prime} }
 \textcolor{black}{\Bigr)}
 \tilde Y_{s,t}'(\cdot) \Big\} d\lambda,
\end{split}
\end{equation}   
where the symbol $\sim$ is used to denote independent copies of the various random variables \textcolor{black}{and where, as before, we used the notation 
\eqref{eq:compactified:notation:interpolation}, with an obvious analogue
%\begin{equation*}
%X_{s;(s,t)}^{(\lambda)}(\omega) 
%= X_{s}(\omega) + \lambda X_{s,t}(\omega), 
%\quad 
%Y_{s;(s,t)}^{(\lambda)}(\cdot) 
%= Y_{s}(\cdot) + \lambda Y_{s,t}(\cdot),
%\end{equation*}
%and similarly 
for the processes tagged with a \textit{prime}
or
a \textit{tilde}.}
By using H\"older inequality with exponents 3 and 3/2, we get
\begin{equation*}
\begin{split}
&\textcolor{black}{\Bigl\langle 
\bigl[  \nabla_{Z} \textrm{F}\bigl(X(\omega),Y(\cdot)\bigr) 
- 
\nabla_{Z} \textrm{F}\bigl(X'(\omega),Y'(\cdot) 
\bigr) 
\bigr]_{s,t}\Bigr\rangle_{4/3}}   
\le \gamma \Big\{ \big\vert \Delta X_{s,t}(\omega) \big\vert +  \big\langle \Delta Y_{s,t}(\cdot)\big\rangle_{4}   \\
&\quad + \vert X_{s,t}(\omega) \vert \Bigl(\vert \Delta X_{s}(\omega) \vert + \big\langle \Delta Y_{s}(\cdot)\big\rangle_{2} + \big\vert \Delta X_{s,t}(\omega) \big\vert + \big\langle   \Delta Y_{s,t}(\cdot)  \big\rangle_{2} \Bigr)   
\\
&\quad+ \big\langle Y_{s,t}(\cdot) \big\rangle_{4} \, \Bigl(\big\vert \Delta \textcolor{black}{X_{s}(\omega)}\big\vert + \big\langle \Delta Y_{s}(\cdot)\bigr\rangle_{2} + \big\vert \Delta X_{s,t}(\omega) \big\vert + \big\langle \Delta Y_{s,t}(\cdot)\big\rangle_{2} \Bigr) \Big\},
\end{split}
\end{equation*}
where, to get the first line, we used the {boundedness and continuity assumptions of} {the functions} $\partial_{x} D_{\mu}\textrm{F}$, $\partial_z D_{\mu}\textrm{F}$ and ${D^2_{\mu}  \textrm{F}}$. {Up to the exponent $4$ appearing on the first and last lines of the right-hand side, we end up with the same bound as in the analysis of $[ \Delta F(\omega,\cdot)]_{s,t}$ in the first step}, namely
\begin{equation*}
\begin{split}
\textcolor{black}{\bigl\langle \delta_{\mu} [\Delta} \textcolor{black}{\textrm{F}(\omega,\cdot)] \bigr\rangle_{[t_{i},t_{i+1}],w,p,4/3}} 
&\leq \gamma \,\Bigl( \vvvert \Delta X(\omega) \vvvert_{[t_{i},t_{i+1}],w,p} + \big\langle \vvvert \Delta Y(\cdot) \vvvert_{[t_{i},t_{i+1}],w,p}\big\rangle_{8}\Bigr)   \\
&\hspace{-20pt}+ \gamma \, w(0,t_{i},\omega)^{1/p}\, \Bigl( \vvvert \Delta X(\omega) \vvvert_{[0,t_{i}],w,p} + \big\langle \vvvert \Delta Y(\cdot) \vvvert_{[0,t_{i}],w,p} \big\rangle_{8} \Bigr).
\end{split}
\end{equation*}

\textbf{\textbf{Step 4 -- }} We use \eqref{eq:chaining:remainder} to write the remainder term 
\textcolor{black}{$R^{\Delta \textrm{F}}$ in} the form
\begin{equation*}
\begin{split}
R^{\Delta \textrm{F}}_{s,t} &= \Big( \partial_{x} \textrm{F}\bigl(X_{s}(\omega),Y_{s}(\cdot)\bigr) - \partial_{x} \textrm{F}\bigl(X_{s}'(\omega),Y_{s}'(\cdot)\bigr)\Big) R_{s,t}^{X}(\omega)   \\
&\hspace{15pt}+ \partial_{x} \textrm{F}\bigl(X_{s}'(\omega),Y_{s}'(\cdot)\bigr) \Bigl(R_{s,t}^X(\omega) - R_{s,t}^{X'}(\omega)\Bigr)   \\ 
&\hspace{15pt}+ \EE \Bigl[ \Bigl( \nabla_{Z}\textrm{F}\bigl(X_{s}(\omega),Y_{s}(\cdot)\bigr) - \nabla_{Z}\textrm{F}\bigl(X_{s}'(\omega),Y_{s}'(\cdot)\bigr)\Bigr) R_{s,t}^{Y}(\cdot) \Bigr]   \\
&\hspace{15pt}+ \EE \Bigl[\nabla_{Z} \textrm{F}\bigl(X_{s}'(\omega),Y_{s}'(\cdot)\bigr)\Bigl( R_{s,t}^Y(\cdot) -  R_{s,t}^{Y'}(\cdot) \Bigr)\Bigr]   \\  
&\hspace{15pt}+ \textbf{(2)}- \textbf{(2')} + \textbf{(3)}-\textbf{(3')}+ \textbf{(5)} - \textbf{(5')},
\end{split}
\end{equation*}
with
\begin{equation*}
\begin{split}
\textbf{(2)} &:= \int_{0}^1 \Big\{ \partial_{x} \textrm{F}\Bigl( \textcolor{black}{X_{s;(s,t)}^{(\lambda)}}(\omega),Y_{t}(\cdot) \Bigr) - \partial_{x} \textrm{F}\Bigl( \textcolor{black}{X_{s;(s,t)}^{(\lambda)}}(\omega),Y_{s}(\cdot) \Bigr) \Big\} X_{s,t}(\omega)\,d\lambda,   
\\
\textbf{(3)} &:= \int_{0}^1 \Big\{ \partial_{x} \textrm{F}\Bigl( \textcolor{black}{X_{s;(s,t)}^{(\lambda)}}(\omega),Y_{s}(\cdot) \Bigr) - \partial_{x} \textrm{F}\bigl( X_{s}(\omega),Y_{s}(\cdot) \bigr) \Big\} X_{s,t}(\omega)\,d\lambda,    
\\
\textbf{(5)} &:= \int_{0}^1 \Big\langle 
\Bigl\{
\nabla_{\textcolor{black}{Z}} \textrm{F}\bigl( X_{s}(\omega) ,Y_{s;(s,t)}^{(\lambda)}(\cdot)\bigr) - \nabla_{Z} \textrm{F}\bigl( X_{s}(\omega) ,Y_{s}(\cdot)\bigr) \Bigr\} Y_{s,t}(\cdot) \Big\rangle \,d\lambda,   
\end{split}
\end{equation*}   
and similarly for \textbf{(2')}, \textbf{(3')} and \textbf{(5')}, putting a \textit{prime} on all the occurrences of $X$ and $Y$. 

We start with the first four lines in $R^{\Delta \textrm{F}}$. Doing as before, the first line is less than 
\begin{equation*}
\begin{split}
&\Bigl\vert \Big[ \partial_{x} \textrm{F}\bigl(X_{s}(\omega),Y_{s}(\cdot)\bigr) - \partial_{x} \textrm{F}\bigl(X_{s}'(\omega),Y_{s}'(\cdot)\bigr)\Big]  R_{s,t}^{X}(\omega) \Bigr\vert   \\
&\leq \gamma \, w(s,t,\omega)^{2/p} \, \Big\{w(0,t_{i})^{1/p} \, \Bigl( \vvvert \Delta X(\omega) \vvvert_{[0,t_{i}],w,p} + \big\langle \vvvert \Delta Y(\cdot) \vvvert_{[0,t_{i}],w,p} \big\rangle_{8} \Bigr)   \\
&\hspace{110pt}+ \Bigl( \vvvert \Delta X(\omega) \vvvert_{[t_{i},t_{i+1}],w,p} + \big\langle \vvvert \Delta Y(\cdot) \vvvert_{[t_{i},t_{i+1}],w,p} \big\rangle_{8} \Bigr) \Big\}.
\end{split}
\end{equation*}
We also have
\begin{equation*}
\Bigl\vert \partial_{x} \textrm{F}\bigl(X_{s}'(\omega),Y_{s}'(\cdot)\bigr) \Bigl(R_{s,t}^X(\omega) - R_{s,t}^{X'}(\omega) \Bigr) \Bigr\vert \leq {\Lambda} w(s,t,\omega)^{2/p} \, \vvvert \Delta X(\omega) \vvvert_{[t_{i},t_{i+1}],w,p}.
\end{equation*}   
\textcolor{black}{Similarly,}
\begin{equation*}
\begin{split}
& \Bigl\vert \EE\Bigl[ \Bigl( \nabla_{Z}\textrm{F}\bigl(X_{s}(\omega),Y_{s}(\cdot)\bigr) - \nabla_{Z}\textrm{F}\bigl(X_{s}'(\omega),Y_{s}'(\cdot)\bigr)\Bigr) R_{s,t}^{Y}(\cdot) \Bigr] \Bigr\vert   \\
 &\leq \gamma\, w(s,t,\omega)^{2/p} \, \Big\{w(0,t_{i})^{1/p} \, \Bigl( \vvvert \Delta X(\omega) \vvvert_{[0,t_{i}],w,p} + \big\langle \vvvert \Delta Y(\cdot) \vvvert_{[0,t_{i}],w,p} \big\rangle_{8} \Bigr)   \\
&\hspace{110pt}+ \Bigl( \vvvert \Delta X(\omega) \vvvert_{[t_{i},t_{i+1}],w,p} + \big\langle \vvvert \Delta Y(\cdot) \vvvert_{[t_{i},t_{i+1}],w,p}
\big\rangle_{8} \Bigr) \Big\},
\\
&\Bigl\vert \EE \Bigl[ \nabla_{Z} \textrm{F}\bigl(X_{s}(\omega),Y_{s}(\cdot)\bigr) \Bigl(R_{s,t}^Y(\cdot) -  R_{s,t}^{Y'}(\cdot)\Bigr) \Bigr] \Bigr\vert \leq \textcolor{black}{2} {\Lambda} w(s,t,\omega)^{2/p} \,\big\langle \vvvert \Delta Y(\cdot) \vvvert_{[t_{i},t_{i+1}],w,p} \big\rangle_{8}.
\end{split}
\end{equation*}
Now, $\bigl\vert \textbf{(2)} - \textbf{(2')} \bigr\vert$ is bounded above by
\begin{equation*}
\begin{split}
&\gamma \,w(s,t,\omega)^{2/p} \, \big\vvvert  \Delta X (\omega)\big\vvvert_{[t_{i},t_{i+1}],w,p}   
\\
&\hspace{15pt}+ \gamma \,w(s,t,\omega)^{1/p} \, \int_{0}^1 \int_{0}^1 \Bigl\vert \Big\langle 
\nabla_{\textcolor{black}{Z}}\partial_x{\textrm{F}}\Bigl( \textcolor{black}{X_{s;(s,t)}^{(\lambda)}}(\omega) ,Y_{s;(s,t)}^{({\lambda'})}(\cdot)\Bigr)  Y_{s,t}(\cdot)\Big\rangle   
\\
&\hspace{130pt} - \Bigl\langle \nabla_{\textcolor{black}{Z}}  \partial_x{\textrm{F}}\Bigl( \textcolor{black}{X_{s;(s,t)}^{(\lambda) \prime}(\omega) }, Y_{s;(s,t)}^{({\lambda'})\prime}(\cdot) \Bigr) Y_{s,t}'(\cdot) \Bigr\rangle \Bigr\vert d\lambda d\lambda',
\end{split}
\end{equation*}
so $\bigl\vert \textbf{(2)} - \textbf{(2')} \bigr\vert$ is bounded above by
\begin{equation*}
\begin{split}
&\gamma \, 
w(s,t,\omega)^{2/p} \, 
\Big\{
\big\vvvert \Delta X(\omega)\big\vvvert_{[t_{i},t_{i+1}],w,p} + \big\langle \vvvert \Delta Y(\cdot)\vvvert_{[t_{i},t_{i+1}],w,p} \big\rangle_{8}
  \\
&\hspace{130pt}+ w(0,t_{i},\omega)^{1/p} \, \Bigl( \vvvert \Delta X(\omega)\vvvert_{[0,t_{i}],w,p} + \big\langle \vvvert \Delta Y(\cdot)\vvvert_{[0,t_{i}],w,p} \big\rangle_{8}\Bigr) \Big\}.
\end{split}
\end{equation*}
The difference $\textbf{(3)}-\textbf{(3')}$ can be handled in the same way. We end up with the term $\textbf{(5)}-\textbf{(5')}$. As $Y_{s,t}$ and $Y_{s,t}'$ may be estimated in $\LL^4$, it suffices to control 
\begin{equation*}
\begin{split}
\textbf{(5a)} &:= \nabla_{\textcolor{black}{Z}}\textrm{F}\bigl( X_{s}(\omega),
Y_{s;(s,t)}^{(\lambda)}(\cdot)
 \bigr) - \nabla_{Z}\textrm{F}\bigl( X_{s}(\omega) ,Y_{s}(\cdot)\bigr),
\\ 
\textbf{(5a)}-\textbf{(5a')} &:=\Bigl( \nabla_{Z} \textrm{F}\bigl( X_{s}(\omega), Y_{s;(s,t)}^{(\lambda)}(\cdot)\bigr) -
\nabla_{Z} \textrm{F}\bigl( X_{s}(\omega), Y_{s}(\cdot)\bigr) \Bigr)   \\
&\hspace{15pt}- \Bigl( \nabla_{Z}\textrm{F}\bigl( X_{s}'(\omega), Y_{s;(s,t)}^{(\lambda)\prime}(\cdot)\bigr) - \nabla_{Z}\textrm{F}\bigl( X_{s}'(\omega), Y_{s}'(\cdot)\bigr)\Bigr),
\end{split}
\end{equation*}
in $\LL^{4/3}$. We have first
$\bigl\langle \textbf{(5a)} \bigr\rangle_{\LL^{4/3}} \leq \bigl\langle \textbf{(5a)} \bigr\rangle_{\LL^{2}} \leq \gamma \, w(s,t,\omega)^{1/p}$.
In order to estimate \textbf{(5a)-(5a')}, we rewrite \textbf{(5a)} \textcolor{black}{in} the form
\begin{equation*}
\begin{split}
\textbf{(5a)} &= D_{\mu}\textrm{F}\Bigl( X_{s}(\omega), \textcolor{black}{Y_{s;(s,t)}^{(\lambda)}(\cdot)}\Bigr)\Bigl( \textcolor{black}{Y_{s;(s,t)}^{(\lambda)}}(\cdot) \Bigr) - D_{\mu}\textrm{F}\Bigl( X_{s}(\omega) ,Y_{s}(\cdot)\Bigr)\bigl( Y_{s}(\cdot)\bigr)   \\
&= \lambda \int_{0}^1 \partial_z D_{\mu}\textrm{F}\Bigl( X_{s}(\omega), \textcolor{black}{Y_{s;(s,t)}^{(\lambda \lambda')}(\cdot)}  \Bigr)\Bigl(
\textcolor{black}{
Y_{s;(s,t)}^{(\lambda \lambda')}(\cdot)} \Bigr) Y_{s,t}(\cdot) d\lambda'    \\
&\qquad+ \lambda \int_{0}^1 \widetilde{\mathbb E}\Bigl[ D_{\mu}^2\textrm{F}\Bigl( X_{s}(\omega), 
\textcolor{black}{Y_{s;(s,t)}^{(\lambda \lambda ')}(\cdot)}\Bigr) \Bigl(\textcolor{black}{Y_{s;(s,t)}^{(\lambda \lambda ')}(\cdot)}
, \textcolor{black}{\widetilde  Y_{s;(s,t)}^{(\lambda \lambda ')}(\cdot)}\Bigr) \widetilde Y_{s,t}(\cdot) \Bigr] d\lambda'.
\end{split}
\end{equation*}   
Then, using H\"older inequality with exponents $3$ and $3/2$ as in {\eqref{eq:variation:step:3}}, we obtain that $\bigl\langle\textbf{(5a)-(5a')}\bigr\rangle_{\LL^{4/3}}$ is bounded above by
\begin{equation*}
\begin{split}
&\gamma \, w(s,t,\omega)^{1/p} \,  \Big\{  \big\vvvert \Delta X(\omega)\big\vvvert_{[t_{i},t_{i+1}],w,p} + \big\langle \vvvert \Delta Y(\cdot)\vvvert_{[t_{i},t_{i+1}],w,p} \big\rangle_{8}
\\
&\hspace{130pt}
+ w(0,t_{i},\omega)^{1/p} \, \Bigl( \vvvert \Delta X(\omega)\vvvert_{[0,t_{i}],w,p} + \big\langle \vvvert \Delta Y(\cdot)\vvvert_{[0,t_{i}],w,p} \big\rangle_{8}\Bigr)   
 \Big\}.
\end{split}
\end{equation*}
and end up with the bound
\begin{equation*}
\begin{split}
\Big\| R^{\Delta \textrm{F}}(\omega) \Big\|_{[t_{i},t_{i+1}],w,p/2} 
&\leq \gamma \, \Big\{ w(0,t_{i},\omega)^{1/p} \, \Bigl( \vvvert \Delta X(\omega)\vvvert_{[0,t_{i}],w,p} + \big\langle \vvvert \Delta Y(\cdot)\vvvert_{[0,t_{i}],w,p} \big\rangle_{8}\Bigr)   \\
&\hspace{60pt} + \big\vvvert \Delta X(\omega)\big\vvvert_{[t_{i},t_{i+1}],w,p} + \big\langle \vvvert \Delta Y(\cdot)\vvvert_{[t_{i},t_{i+1}],w,p} \big\rangle_{8}\Big\}.
\end{split}
\end{equation*}   \vskip 4pt

\textbf{\textbf{Conclusion.}} Plugging the conclusion of the previous steps ({including the analysis of the various initial conditions}) into equation \eqref{eq:contraction:proof:0}, we get
\begin{equation}
\label{eq:conclusion:proposition:15}
\begin{split}
&\biggl\vvvert \int_{t_{i}}^{\cdot} \textrm{F}\bigl(X_{r}(\omega),Y_{r}(\cdot)\bigr) d {\boldsymbol W}_{r}(\omega) - \int_{t_{i}}^{\cdot} \textrm{F}\bigl(X_{r}'(\omega),Y_{r}'(\cdot)\bigr) d {\boldsymbol W}_{r}(\omega) \biggr\vvvert_{[t_{i},t_{i+1}],w,p}   \\
&\leq \gamma \, \Bigl({  \big\vert \Delta X_{t_{i}}(\omega)\big\vert +
\big\vert \delta_{x} \Delta  X_{t_{i}}(\omega)\big\vert
+ 
 \bigl\langle \Delta Y_{t_{i}}(\cdot)  \bigr\rangle_{4}
 +
 \bigl\langle  \delta_{x} \Delta Y_{t_{i}}(\cdot)  \bigr\rangle_{4}}
 \\ 
&\quad+ \gamma\, w(t_{i},t_{i+1},\omega)^{1/p} \,\bigl\vvvert \textrm{F}(X(\omega),Y(\cdot)) - \textrm{F}(X'(\omega),Y'(\cdot)) \bigr\vvvert_{[t_{i},t_{i+1}],w,p} \Bigr)     
 \\
&\leq \gamma \, w(0,t_{i},\omega)^{1/p}\, \Bigl(\vvvert \Delta X(\omega) \vvvert_{[0,t_{i}],w,p} + \big\langle \vvvert \Delta Y(\cdot)\vvvert_{[0,t_{i}],w,p} \big\rangle_{8} \Bigr)   \\
&\quad+ \gamma \, w(t_{i},t_{i+1},\omega)^{1/p} \,\Big\{\Bigl( \vvvert \Delta X(\omega) \vvvert_{[t_{i},t_{i+1}],w,p} + \big\langle
\vvvert \Delta Y(\cdot) \vvvert_{[t_{i},t_{i+1}],w,p}\big\rangle_{8}\Bigr)   \\
&\hspace{60pt} + w(0,t_{i},\omega)^{1/p} \, \Bigl(\vvvert \Delta X(\omega) \vvvert_{[0,t_{i}],w,p} + \big\langle \vvvert \Delta Y(\cdot) \vvvert_{[0,t_{i}],w,p} \big\rangle_{8}\Bigr) \Big\}.
\end{split}
\end{equation}
{Recalling} that $w(t_{i},t_{i+1},\omega)^{1/p}\leq 1/(4 L)$, we finally get
\begin{equation*}
\begin{split}
&\biggl\vvvert \int_{t_{i}}^{\cdot} \textrm{F}\bigl(X_{r}(\omega),Y_{r}(\cdot)\bigr) d {\boldsymbol W}_{r}(\omega) - \int_{t_{i}}^{\cdot} \textrm{F}\bigl(X_{r}'(\omega),Y_{r}'(\cdot)\bigr) d {\boldsymbol W}_{r}(\omega) \biggr\vvvert_{[t_{i},t_{i+1}],w,p}
\\
&\leq \gamma \, w(0,t_{i},\omega)^{1/p} \, \left( 1 + \frac{1}{4L} \right) \, \Bigl( \big\vvvert \Delta X(\omega) \big\vvvert_{[0,t_{i}],w,p} + \big\langle \vvvert \Delta Y(\cdot) \vvvert_{[0,T],w,p} \big\rangle_{8} \Bigr)   \\
&\quad+ \frac{\gamma}{4L} \,\Big\{ \vvvert \Delta X(\omega) \vvvert_{[t_{i},t_{i+1}],w,p} + \big\langle \vvvert \Delta Y(\cdot) \vvvert_{[0,T],w,p}  \big\rangle_{8} \Big\}.
\end{split}
\end{equation*}
This completes the proof.  \end{proof}

%%---------------------------%%
\subsection{Well-posedness}
\label{SubsectionWellPosedness}
%%---------------------------%%

We first prove a well-posedness result in small time from which Theorem \ref{ThmMain} follows. Recall 
%from \eqref{eq:v} and \eqref{eq:w:s:t:omega} the definition of $w(0,T,\cdot)$, and from 
from Definition \ref{DefnGamma} the fact that the map $\Gamma$ depends on $X_0(\omega)$.
%; recall also from Lemma \ref{LemmEqual1} that there is no loss of generality in assuming $\textcolor{black}{\Lambda}=1$ in \eqref{EqRegularityF} -- this explains the bound for $\partial_xX(\omega)$ in the statement below.

\begin{thm}
\label{main:thm:existence:small:time}
Let \emph{F} satisfy \textbf{\textbf{Regularity assumptions 1}}  and \textbf{\textbf{Regularity assumptions 2}}
{and $w$ be a control 
satisfying \eqref{eq:w:s:t:omega:ineq}
and 
\eqref{eq:useful:inequality:wT}. Assume there exists a positive time horizon $T$ such that the random variables $w(0,T,\textcolor{black}{\cdot})$ and 
$\bigl(N\big([0,T],\textcolor{black}{\cdot},\alpha\big)\bigr)_{\alpha >0}$ have sub and super exponential tails respectively, 
{namely}
\begin{equation}
\label{EqTailAssumptions}
\begin{split}
\PP \bigl( w(0,T,\cdot) \geq t \bigr) &\leq c_{1} \exp \bigl( - t^{\varepsilon_1} \bigr),   
\quad 
\PP \bigl( N([0,T],\cdot,\alpha) \geq t \bigr) \leq c_{2}(\alpha) \exp \bigl( - t^{1+ \varepsilon_2(\alpha)} \bigr),
\end{split}
\end{equation}
for some positive \textcolor{black}{constants $c_{1}$} and $\epsilon_1$, and possibly $\alpha$-dependent positive constants $c_{2}(\alpha)$ and $\epsilon_2(\alpha)$. Then, there exists 
%a positive random variable $A$ satisfying
%$
%\bigl\langle A(\cdot)^{N([0,T],\cdot,1/(4L))} \bigr\rangle_{1} < \infty,
%$
four positive reals $\gamma$,  $L_0$, $L$ and $\eta$, only depending on {$\Lambda$} and $T$, with the following property. For $0\leq S\leq T$ such that
\begin{equation}
\label{eq:constraint:1:E!:small:time}
\textcolor{red}{\Bigl\langle N\bigl([0,S],\cdot,1/(4L_{0})\bigr)  \Bigr\rangle_{8}}  \leq {1},
\end{equation}
and 
\begin{equation}
\label{eq:constraint:2:E!:small:time}
\Bigl\langle \Bigl[ \gamma \Bigl( 1 + w(0,T,\cdot)^{1/p} \Bigr)\Bigr]^{N([0,S],\cdot,1/(4L))} \Bigr\rangle_{32} \textcolor{black}{\leq \eta},
\end{equation}
and for any $d$-dimensional random square-integrable variable $X_{0}$, there exists a random controlled path $X(\cdot)=(X(\omega))_{\omega \in \Omega}$ defined on the time interval $[0,S]$ satisfying  
$\big\langle \delta_{x} X(\cdot) \big\rangle_{\infty} \leq {\Lambda}$, 
and 
$\bigl\langle \vvvert X(\cdot) \vvvert_{[0,S],w,p} \bigr\rangle_{8} < \infty$ (the bound for the latter only depending on
{$\Lambda$}
and the parameters in \eqref{EqTailAssumptions}),
such that, for every $\omega \in \Omega$, the paths $X(\omega)$ and $\Gamma(\omega , X(\omega) , X(\cdot))$ coincide on $[0,S]$. Any other random controlled path $X'(\cdot)$ with $X'_0=X_0$ almost surely, and such that the paths $X'(\omega)$ and $\Gamma\big(\omega,X'(\omega),X'(\cdot)\big)$ coincide almost surely, satisfies 
$$
\PP\Big(\vvvert X(\cdot) - X'(\cdot) \vvvert_{\star,[0,S],w,p} = 0 \Big) = 1.
$$}
\end{thm}

\begin{proof}
We construct a fixed point of  $\Gamma$, see Definition \ref{DefnGamma}, as the limit of the Picard sequence
\begin{equation}
\label{eq:Picard:sequence}
\begin{split}
&\bigl(X^{n+1}(\omega)
 ; \delta_{x} X^{n+1}(\omega) ; 0\bigr) 
 \\
&\hspace{15pt} := \Gamma\Bigl(\omega,\bigl(X^{n}(\omega) ; \delta_{x} X^{n}(\omega)
 ; 0\bigr),\bigl(X^n(\omega') ; \delta_{x} X^{n}(\omega') ; 0\bigr)_{\omega' \in \Omega} \Bigr), 
 \end{split}
\end{equation}
started from 
$\bigl(X^{0}(\omega);\partial_{x} X^{0}(\omega); 0\bigr) \equiv 
\bigl(X_{0}(\omega) ; 0 ; 0\bigr)$, 
for each $\omega\in\Omega$.
{By induction, for any $n \geq 0$, the pair $(X(\omega),Y(\cdot))=(X^n(\omega),X^n(\cdot))$
satisfies 
\eqref{eq:invariant:0000}
in the statement of 
Proposition 
\ref{thm:main:1}. 
Moreover,
by the first bullet point in the conclusion of 
Proposition \ref{thm:main:1}, 
$X(\omega)=X^n(\omega)$ satisfies 
\eqref{eq:invariant:2} for any $n \geq 1$, provided that $L$ therein is taken large enough (independently on $n$). By 
\eqref{eq:invariant:3} and from the tail estimates \eqref{EqTailAssumptions}}, we deduce that, for any $n \geq 0$, 
$\vvvert X^n(\cdot) \vvvert_{[0,T],w,p}$ has finite moments of any order: 
According to Definition \ref{definition:random:controlled:trajectory}, each $X^n(\cdot)=(X^n(\omega))_{\omega \in \Omega}$, $n \geq 1$, is a random controlled trajectory.
%; moreover, by \textbf{\textbf{Regularity assumptions 1}} and Proposition \ref{prop:chaining} on the one hand and by \eqref{eq:invariant:3} and \eqref{EqTailAssumptions} on the other hand, each $(X^{n}(\omega),X^{n}(\cdot))$ satisfies the assumptions of Proposition \ref{thm:main:1} for some $L>0$.

\textbf{\textbf{Step 1.}} Instead of working with $S$ such that \textcolor{red}{$\big\langle N([0,S]\cdot,1/(4L_{0}))  \big\rangle_{8} \leq 1$}, we directly assume that {\textcolor{red}{$\big\langle N(\textcolor{black}{[0,T]},\cdot,1/(4L_{0}))  \big\rangle_{8}  \leq 1$}},  with $L_{0}$ as in  Proposition \ref{thm:main:1}. 
{Recalling that we may take $L_{0}$ large enough so that 
\eqref{eq:invariant:2} holds true {with $L=L_{0}$ and $X=X^n$} for any $n \geq 0$, we} deduce that, for any $n \geq 1$, both $X^n$ and $X^{n-1}$ satisfy  \eqref{eq:fixed:1} and \eqref{eq:fixed:2}:
{\eqref{eq:fixed:1} follows from the third item in the conclusion of 
Proposition \ref{thm:main:1}, whilst 
\eqref{eq:fixed:2} follows from the first item}. Hence,  by Proposition \ref{thm:fixed:1}, $\bigl\vvvert \Delta X^{n}(\omega) \bigr\vvvert_{[t_{i},t_{i+1}],w,p}$, with 
$\Delta X^{n}:=X^{n+1} - X^n$
 is bounded above by 
\begin{equation*}
\begin{split}
&\gamma \, w(0,t_{i},\omega)^{1/p} \Bigl( 1 + \frac{1}{4L} \Bigr) \Big\{ \bigl\vvvert \Delta X^{n-1}(\omega) \bigr\vvvert_{[0,t_{i}],w,p}   + \Bigl\langle \vvvert\Delta X^{n-1}(\cdot) \vvvert_{[0,T],w,p} \Bigr\rangle_{8} \Big\}   \\
&+ \frac{\gamma}{4L} \, \Big\{ \bigl\vvvert \Delta X^{n-1}(\omega) \bigr\vvvert_{[t_{i},t_{i+1}],w,p} + \Bigl\langle \vvvert \Delta X^{n-1}(\cdot) \vvvert_{[0,T],w,p}\Bigr\rangle_{8} \Big\},
\end{split}
\end{equation*}
for any $n \geq 1$, {where 
$\gamma$ depends on $L_{0}$ and $\Lambda$, 
$L$ is greater than $L_{0}$, 
and the sequence $(t_{i})_{\textcolor{black}{i=0,\cdots,N}}$ is as in the statement of Proposition \ref{thm:fixed:1}.
The precise value of $L$ will be fixed later on; the key fact is that it may be taken as large as needed.} We start with the case $i=0$. The above bound yields, for all $n \geq 1$,\begin{equation*}
\begin{split}
&\bigl\vvvert \Delta X^{n}(\omega) \bigr\vvvert_{[0,t_{1}],w,p}  \leq \frac{3\gamma}{4L} \, \Big\{ \bigl\vvvert \Delta X^{n-1}(\omega) \bigr\vvvert_{[0,t_{1}],w,p} + \Bigl\langle \vvvert \Delta X^{n-1}(\cdot) \vvvert_{[0,T],w,p} \Bigr\rangle_{8} \Big\}.
\end{split}
\end{equation*}   
So, {recalling that $\Delta X^0(\omega)=X^1(\omega)$,} we have, for any $n \geq 1$, 
\begin{equation}
\label{eq:contraction:t1}
\begin{split}
&\bigl\vvvert \Delta X^{n}(\omega) \bigr\vvvert_{[0,t_{1}],w,p}
\\ 
&\hspace{15pt} \leq \Bigl( \frac{3 \gamma}{4 L} \Bigr)^{n} \bigl\vvvert X^{1}(\omega) \bigr\vvvert_{[0,t_{1}],w,p}   
 + \sum_{k=1}^{n}\Bigl( \frac{3 \gamma}{4 L}\Bigr)^{n+1-k} \Bigl\langle \vvvert \Delta X^{k-1}(\cdot) \vvvert_{[0,T],w,p}\Bigr\rangle_{8}.
\end{split}
\end{equation}
We proceed with a similar computation when $i \geq 1$. By induction, we have, for $n \geq 1$,
\begin{equation*}
\begin{split}
&\bigl\vvvert \Delta X^{n}(\omega)
\bigr\vvvert_{[t_{i},t_{i+1}],w,p}
\leq 
\Bigl( \frac{\gamma}{4 L}\Bigr)^{n} \bigl\vvvert X^{1}(\omega) \bigr\vvvert_{[t_{i},t_{i+1}],w,p}   \\
&\quad+ \sum_{k=1}^n \Bigl( \frac{\gamma}{4L}\Bigr)^{n+1-k} \Bigl[\gamma w(0,t_{i},\omega)^{1/p} \, \Bigl( 1 + \frac{1}{4L} \Bigr) \, \bigl\vvvert \Delta X^{k-1}(\omega) \bigr\vvvert_{[0,t_{i}],w,p} \Bigr]   \\
&\quad +\sum_{k=1}^n  \Bigl( \frac{\gamma}{4L} \Bigr)^{n+1-k} \Bigl[ \gamma\, \Big\{ \frac1{4L} + w(0,t_{i},\omega)^{1/p} \bigl( 1 + \frac{1}{4L} \bigr) \Big\}   \Bigl\langle \vvvert\Delta X^{k-1}(\cdot) \vvvert_{[0,T],w,p} \Bigr\rangle_{8}  \Bigr].
\end{split}
\end{equation*}
Following 
{footnote${}^{\ref{footnote:6}}$}, we get, for a new value of $\gamma$, 
\begin{equation*}
\begin{split}
&\bigl\vvvert \Delta X^{n}(\omega) \bigr\vvvert_{[0,t_{i+1}],w,p}   
 \leq \gamma \bigl\vvvert \Delta X^{n}(\omega) \bigr\vvvert_{[0,t_{i}],w,p} + \gamma \bigl\vvvert \Delta X^{n}(\omega) \bigr\vvvert_{[t_{i},t_{i+1}],w,p},
\end{split}
\end{equation*}
so
\begin{equation*}
\begin{split}
&\bigl\vvvert \Delta X^{n}(\omega) \bigr\vvvert_{[0,t_{i+1}],w,p}   \leq \gamma \, \bigl\vvvert\Delta X^{n}(\omega) \bigr\vvvert_{[0,t_{i}],w,p} + \gamma \, \Bigl( \frac{\gamma}{4 L} \Bigr)^{n} \bigl\vvvert X^{1}(\omega) \bigr\vvvert_{[t_{i},t_{i+1}],w,p}   \\
&\hspace{15pt}+ \gamma \sum_{k=1}^n 	\Bigl( \frac{\gamma}{4L} \Bigr)^{n+1-k} \Bigl[\gamma  \, w(0,t_{i},\omega)^{1/p}\, \Bigl( 1 + \frac{1}{4L}\Bigr) \, \bigl\vvvert \Delta X^{k-1}(\omega) \bigr\vvvert_{[0,t_{i}],w,p} \Bigr]   \\
&\hspace{15pt}+ \gamma \sum_{k=1}^n  \Bigl( \frac{\gamma}{4L} \Bigr)^{n+1-k} \Bigl[ \gamma \, \Big\{ \frac1{4L} + w(0,t_{i},\omega)^{1/p} \bigl( 1 + \frac{1}{4L} \bigr)  \Big\}  
 \Bigl\langle  \vvvert \Delta X^{k-1}(\cdot) \vvvert_{[0,T],w,p} \Bigr\rangle_{8}  \Bigr],
\end{split}
\end{equation*}
which we can rewrite as
\begin{equation*}
\begin{split}
&\bigl\vvvert \Delta X^{n}(\omega) \bigr\vvvert_{[0,t_{i+1}],w,p}  
\leq \gamma^2 \zeta(\textcolor{black}{\omega}) \biggl\{\sum_{k=1}^{n+1}\Bigl( \frac{\gamma}{4L} \Bigr)^{n+1-k} \, \bigl\vvvert \Delta X^{k-1}(\omega) \bigr\vvvert_{[0,t_{i}],w,p}  
\\
&\quad 
+ \Bigl( \frac{\gamma}{4 L} \Bigr)^{n} \bigl\vvvert X^{1}(\omega) \bigr\vvvert_{[t_{i},t_{i+1}],w,p} 
 + \sum_{k=1}^n \Bigl( \frac{\gamma}{4L} \Bigr)^{n+1-k} \, \Bigl\langle \vvvert \Delta X^{k-1}(\cdot) \vvvert_{[0,T],w,p}\Bigr\rangle_{8} \biggr\},
\end{split}
\end{equation*}
\textcolor{black}{provided we choose $\gamma \geq 1$, and}
with 
$\zeta(\omega) := 1 + w(0,T,\omega)^{1/p} \, \Bigl( 1 + \frac{1}{4L} \Bigr)$.

\textbf{\textbf{Step 2.}} Combine the above estimate together with \eqref{eq:contraction:t1} to get
\begin{equation*}
\begin{split}
&\bigl\vvvert \Delta X^{n}(\omega)
\bigr\vvvert_{[0,t_{2}],w,p}
\leq 
\gamma^2 \zeta(\omega) \sum_{k=1}^{n+1}
\Bigl( \frac{\gamma}{4L} \Bigr)^{n+1-k} \, 
\Bigl( \frac{3 \gamma}{4 L}\Bigr)^{k-1} \, 
\bigl\vvvert X^{1}(\omega)
\bigr\vvvert_{[0,t_{1}],w,p}
\\
&\hspace{5pt}+ \gamma^2 \zeta(\omega) \sum_{k=1}^{n} \Bigl( \frac{ \gamma}{4 L}\Bigr)^{{n}-k}  \sum_{i=1}^k \Bigl(\frac{3 \gamma}{4L} \Bigr)^{k+1-i} \Big\langle \vvvert \Delta X^{i-1}(\cdot) \bigr\vvvert_{[0,T],w,p}\Big\rangle_{8}   \\
&\hspace{5pt}+ \gamma^2 \zeta(\omega) \sum_{k=1}^n  \Bigl( \frac{\gamma}{4L} \Bigr)^{n+1-k} \Bigl\langle \vvvert \Delta X^{k-1}(\cdot) \bigr\vvvert_{[0,T],w,p} \Big\rangle_{8} + \gamma^2 \zeta(\omega) \Bigl( \frac{\gamma}{4 L} \Bigr)^{n}\, \bigl\vvvert X^{1}(\omega)\bigr\vvvert_{{[t_{1},t_{2}]},w,p}   .
\end{split}
\end{equation*}
Hence we have
\begin{equation*}
\begin{split}
&\bigl\vvvert \Delta X^{n}(\omega) \bigr\vvvert_{[0,t_{2}],w,p} \leq \gamma^2 \zeta(\omega) \, \Bigl( \frac{3 \gamma}{4 L} \Bigr)^{n} \Bigl( 1  +  \sum_{k=1}^{n+1} \bigl( \frac{1}{3} \bigr)^{n+1-k} \Bigr)\,\bigl\vvvert X^{1}(\omega) \bigr\vvvert_{[0,t_{2}],w,p}   \\
&\quad+ \gamma^2 \zeta(\omega)   
\Bigl(\frac{\gamma}{4L} \Bigr)^{{n}}
\sum_{i=1}^{n} \Bigl( \frac{3\gamma}{4L} \Bigr)^{1- i} \Bigl\langle 
\vvvert \Delta X^{i-1}(\cdot) \vvvert_{[0,T],w,p} \Bigr\rangle_{8} \sum_{k=i}^n 3^k   \\
&\quad+ \gamma^2 \zeta(\omega) \sum_{k=1}^n  \Bigl( \frac{\gamma}{4L} \Bigr)^{n+1-k} \Bigl\langle \vvvert 
\Delta X^{k-1}(\cdot) \Bigr\vvvert_{[0,T],w,p} \Big\rangle_{8}.
\end{split}
\end{equation*}
Therefore, using the bound $\sum_{k=i}^n 3^k \leq 3^{n+1}/2$, we deduce
\begin{equation*}
\begin{split}
\bigl\vvvert \Delta X^{n}(\omega) \bigr\vvvert_{[0,t_{2}],w,p}   
&\leq  3\gamma^2 \zeta(\omega) \Bigl( \frac{3 \gamma}{4 L} \Bigr)^{n} \, \bigl\vvvert X^{1}(\omega) \bigr\vvvert_{[0,t_{2}],w,p}  
 \\
&\quad + 3 \gamma^2 \zeta(\omega)  \sum_{i=1}^{n} \Bigl( \frac{3\gamma}{4L} \Bigr)^{n+1-i} \Bigl\langle \vvvert \Delta X^{i-1}(\cdot) \vvvert_{[0,T],w,p} \Bigr\rangle_{8}.
\end{split}
\end{equation*}
We here {assume that $L$ is chosen} big enough to have $3 \gamma<4L$. The above inequality may be summed up into
\begin{equation*}
\begin{split}
\bigl\vvvert \Delta X^{n}(\omega) \bigr\vvvert_{[0,t_{2}],w,p}
&\leq  c_{2}(\omega) \Bigl( \frac{3 \gamma}{4 L} \Bigr)^{n} \, \bigl\vvvert X^{1}(\omega) \bigr\vvvert_{[0,t_{2}],w,p}   \\
&\quad + c_{2}(\omega)  \sum_{i=1}^{n} \Bigl( \frac{3\gamma}{4L} \Bigr)^{n+1-i} \, \Bigl\langle \vvvert \Delta X^{i-1}(\cdot) \vvvert_{[0,T],w,p}
\Bigr\rangle_{8},
\end{split}
\end{equation*}
where $c_{2}(\omega) \textcolor{black}{:}= 3 \gamma^2 \zeta(\omega)$. 
Set \textcolor{black}{now}
$c_{i}(\omega) := \bigl(3 \gamma^2 \zeta(\omega)\bigr)^{\textcolor{black}{i-1}}$. 

Comparing the previous estimate of $\bigl\vvvert \Delta X^{n}(\omega) \bigr\vvvert_{[0,t_{2}],w,p}$ with \eqref{eq:contraction:t1} and iterating over the time index $t_{i}$ \textcolor{black}{from the conclusion of the first step}, we obtain, as long as $t_{i} \leq T$, 
\begin{equation*}
\begin{split}
\bigl\vvvert \Delta X^{n}(\omega) \bigr\vvvert_{[0,t_{i}],w,p}
&\leq  c_{i}(\omega) \, \Bigl( \frac{3 \gamma}{4 L} \Bigr)^{n} \, \bigl\vvvert X^{1}(\omega)\bigr\vvvert_{[0,t_{i}],w,p}   \\
&\quad + c_{i}(\omega)  \sum_{k=1}^{n} \Bigl( \frac{3\gamma}{4L} \Bigr)^{n+1-k} \, \Bigl\langle  \vvvert \Delta X^{{k-1}}(\cdot) \vvvert_{[0,T],w,p}
\Bigr\rangle_{8}.
\end{split}
\end{equation*}

\textbf{\textbf{Step 3.}} 
Noting 
that we can take  $N$ in 
Theorem \ref{thm:fixed:1}
less than ${N}\bigl([0,T],\omega,1/(4L_{0})\bigr) + 
{N}\bigl([0,T],\omega,1/(4L)\bigr) \leq 2 
N \bigl([0,T],\omega,1/(4L)\bigr)$, see definition \eqref{eq:N:s:t:omega}, we deduce that
\begin{equation}
\label{eq:contraction:proof:10}
\begin{split}
&\bigl\vvvert \Delta X^{n}(\omega) \bigr\vvvert_{[0,T],w,p}  \leq  \Bigl(3 \gamma^2 \zeta(\omega)\Bigr)^{\textcolor{black}{2N(\omega,1/(4L))}} \,\Bigl( \frac{3 \gamma}{4 L} \Bigr)^{n}\, \bigl\vvvert X^{1}(\omega)\bigr\vvvert_{[0,T],w,p}   \\
&\quad + \Bigl(3 \gamma^2 \zeta(\omega)\Bigr)^{\textcolor{black}{2N(\omega,1/(4L))}}  \sum_{k=1}^{n} \Bigl( \frac{3\gamma}{4L} \Bigr)^{n+1-k}
\Bigl\langle \vvvert \Delta X^{k-1}(\cdot) \vvvert_{[0,T],w,p} \Bigr\rangle_{8},
\end{split}
\end{equation}
where we let $N\big(\omega,1/(4L)\big) := N\big(\textcolor{black}{[0,T],}\omega,1/(4L)\big)$. It follows from the assumed tail behaviour of $N\big(\cdot,1/(4L)\big)$ and $w(0,T,\cdot)$ that we have, for $a > 1$ and any integer $k$ the upper bound
\begin{equation}
\label{eq:wibull:exp}
\begin{split}
\PP \Bigl( \bigl\{ \omega \in \Omega : \zeta^{2N(\omega,1/(4L))}(\omega) \geq a \bigr\} \Bigr) &\leq \PP \bigl( N(\cdot,\textcolor{black}{1/(4L)})\geq k\bigr) + \PP \bigl( \zeta^2 \geq a^{1/k} \bigr)   \\
&\leq \textcolor{black}{c} \exp \bigl( - k^{1+\varepsilon_2} \bigr) + c \exp \left( - {a^{\varepsilon_1 p/(4k)}} \right),
\end{split}
\end{equation}
for a constant $c \geq 1$ depending on $L$ and with $\varepsilon_{2}=\varepsilon_{2}(1/(4L))$.
{In order to derive the last term right above, we used Markov inequality together with 
the fact that ${\mathbb E}[ \exp ( \zeta^{\varepsilon_{1} p/2})]$
is bounded by a constant depending on $c_{1}$, $\epsilon_{1}$ and $L$.}
 For $k= (\ln a)^{1/(1+\textcolor{black}{\varepsilon_{2}}/2)}$,  
\begin{equation*}
\forall \ell \in {\mathbb N} \setminus \{0\}, \quad \PP \Bigl( \Bigl\{ \omega \in \Omega : \zeta^{2N(\omega,1/(4L))}(\omega) \geq a \Bigr\} \Bigr) \leq \textcolor{black}{C_{\ell}} a^{-\ell}, 
\end{equation*}
\textcolor{black}{for a constant $C_{\ell}$ depending on $\ell$}, from which we deduce that 
$\bigl\langle  \bigl( 3 \gamma^2 \zeta \bigr)^{2N(\cdot,1/(4L))} \bigr\rangle_{16} <\infty$.

Set \textcolor{black}{now} 
$
A := ( 3 \gamma^2 \zeta)^{\textcolor{black}{2N(\cdot,1/(4L)}}.$ 
Importantly, $A$ depends on the time horizon $T$ through {$\zeta$} and $N(\cdot,1/(4L))$ \textcolor{black}{(and this on $L$ as well)}. \textcolor{black}{In order to emphasize the dependance upon the time argument, we expand the notation and} write 
$
A_{T} := ( 3 {\gamma^2} \zeta_{T})^{2N([0,T],\textcolor{black}{\cdot},1/(4L))}$.

Clearly, 
$
A_{S} \leq ( 3 {\gamma^2} \zeta_{T})^{2N([0,S],\textcolor{black}{\cdot},1/(4L))}$,  
since ${\gamma}$ and $\zeta_{T}$ are greater than $1$. Since the term $N\big([0,S],\textcolor{black}{\cdot},1/(4L)\big)$ tends to $0$ with $S$, we have
$\lim_{S \searrow 0} \, \bigl\langle \big(3 {\gamma^2} \zeta_{T}\big)^{2N([0,S],\textcolor{black}{\cdot},1/(4L))} \bigr\rangle_{16} = 1$,
so
$\lim_{S \searrow 0} \bigl\langle A_{S} \bigr\rangle_{16} = 1$.
Hence, taking the $\LL^{8}$ norm in \eqref{eq:contraction:proof:10} with $T$ replaced by $S$, 
\begin{equation*}
\begin{split}
\Bigl\langle \vvvert \Delta X^{n}(\cdot) \vvvert_{[0,S],w,p}  \Bigr\rangle_{8}   
&\leq \bigl(1+\delta(S)\bigr) \Bigl( \frac{3\gamma}{4L} \Bigr)^{n}
\Bigl\langle \big\vvvert X^{1}(\cdot) \big\vvvert_{[0,S],w,p} \Bigr\rangle_{{16}}   \\
&\quad+ \bigl(1+\delta(S)\bigr)  \sum_{i=1}^{n} \Bigl( \frac{3\gamma}{4L}\Bigr)^{n+1-i} \Bigl\langle \big\vvvert \Delta X^{i-1}(\cdot) \big\vvvert_{[0,S],w,p} \Bigr\rangle_{8}   \\
&= \bigl(1+\delta(S)\bigr) \Bigl( \frac{3\gamma}{4L} \Bigr)^{n} \Bigl\langle \big\vvvert X^{1}(\cdot) \big\vvvert_{[0,S],w,p} \Bigr\rangle_{{16}}   \\
&\quad+ \bigl(1+\delta(S)\bigr)  \sum_{i=0}^{n-1} \Bigl( \frac{3\gamma}{4L}\Bigr)^{n-i} \Bigl\langle \big\vvvert \Delta X^{i}(\cdot) \big\vvvert_{[0,S],w,p} \Bigr\rangle_{8},
\end{split}
\end{equation*}
where {$\delta(S)>0$} tends to $0$ with $S$. 
So, we have
 \begin{equation*}
\begin{split}
&\sum_{k=0}^{n} \Bigl( \frac{3\gamma}{4L} \Bigr)^{(n-k)/2} \Bigl\langle \big\vvvert \Delta X^{k}(\cdot) \big\vvvert_{[0,S],w,p} \Bigr\rangle_{8}   \\
&\leq \bigl( 1 + \delta(S) \bigr)  \sum_{k=0}^{n} \Bigl( \frac{3\gamma}{4L}\Bigr)^{(n-k)/2} \Bigl( \frac{3\gamma}{4L}\Bigr)^{k} \Bigl\langle \big\vvvert X^{1}(\cdot) \big\vvvert_{[0,S],w,p} \Bigr\rangle_{{16}}   \\
&\quad+ \bigl(1+\delta(S)\bigr)  \sum_{i=0}^{n-1} \Bigl( \frac{3\gamma}{4L}\Bigr)^{(n-i)/2} \Bigl\langle \big\vvvert \Delta X^{\textcolor{black}{i}}(\cdot) \big\vvvert_{[0,S],w,p}
\Bigr\rangle_{8} \sum_{k=i+1}^{n}  \Bigl(\frac{3\gamma}{4L} \Bigr)^{(k-i)/2}   \\
&\le \bigl( 1 + \delta(S) \bigr)  \Bigl( \frac{3\gamma}{4L}\Bigr)^{n/2} \sum_{k=0}^{n}  \Bigl(\frac{3\gamma}{4L}\Bigr)^{k/2}
 \Bigl\langle \big\vvvert X^{1}(\cdot) \big\vvvert_{[0,S],w,p} \Bigr\rangle_{{16}}   \\
&\quad+ \frac{1+\delta(S) }{1-\sqrt{3\gamma/(4L)}} \, \Bigl(\frac{3\gamma}{4L} \Bigr)^{1/2} \sum_{i=0}^{n}  \Bigl(\frac{3\gamma}{4L} \Bigr)^{(n-i)/2}  \Bigl\langle \big\vvvert \Delta X^{i}(\cdot) \big\vvvert_{[0,S],w,p} \Bigr\rangle_{8}.
\end{split}
\end{equation*}
Assuming that $3\gamma/(4L) \leq 1/16$
and choosing $S$ small enough, we may assume that
\begin{equation}
%\frac{1}{1-\sqrt{3\gamma/(4L)}} \Bigl(\frac{3\gamma}{4L} \Bigr)^{1/2} < 1,
%\quad 
\label{eq:condition:a}
a := \frac{1+\delta(S)}{1-\sqrt{3\gamma/(4L)}} \Bigl( \frac{3\gamma}{4L} \Bigr)^{1/2} < 1,
\end{equation}
we can find a positive constant $C$ such that 
\begin{equation*}
\begin{split}
&\sum_{k=0}^{n} 
\Bigl(\frac{3\gamma}{4L} \Bigr)^{(n-k)/2} \Bigl\langle \big\vvvert \Delta X^{k}(\cdot) \big\vvvert_{[0,S],w,p} \Bigr\rangle_{8}   \\
&\le C
\Bigl( \frac{3\gamma}{4L} \Bigr)^{n/2} \Bigl\langle \big\vvvert X^{1}(\cdot) \big\vvvert_{[0,S],w,p} \Bigr\rangle_{{16}}
+ a \, \sum_{i=0}^{n} \Bigl( \frac{3\gamma}{4L} \Bigr)^{(n-i)/2} \, \Bigl\langle \big\vvvert \Delta X^{i}(\cdot) \big\vvvert_{[0,S],w,p} \Bigr\rangle_{8}.
\end{split}
\end{equation*}
Changing the value of $C$ if necessary, we obtain
\begin{equation*}
\begin{split}
&\sum_{k=0}^{n} \Bigl(\frac{3\gamma}{4L} \Bigr)^{(n-k)/2} \, \Bigl\langle \vvvert \Delta X^{k}(\cdot) \vvvert_{[0,S],w,p}\Bigr\rangle_{8}
\le C\,\Bigl(\frac{3\gamma}{4L} \Bigr)^{n/2} \,\Bigl\langle \big\vvvert X^{1}(\cdot)\big\vvvert_{[0,S],w,p} \Bigr\rangle_{{16}},
\end{split}
\end{equation*}
%which leads to  
%\begin{equation*}
%\begin{split}
%&\Bigl\langle \big\vvvert \Delta X^{n}(\cdot) \big\vvvert_{[0,S],w,p}  \Bigr\rangle_{8}  \leq C \, a^n \Bigl\langle \big\vvvert X^{1}(\cdot)\big\vvvert_{[0,S],w,p} \Bigr\rangle_{{16}}.
%\end{split}
%\end{equation*}
Using \eqref{eq:contraction:proof:10}, we eventually have, for a new value of $C$,
\begin{equation}
\label{eq:conclusion:preuve:contraction}
\begin{split}
&\bigl\vvvert \Delta X^{n}(\omega) \bigr\vvvert_{[0,S],w,p} 
\leq C  \bigl(3 \gamma^2 \zeta(\omega)\bigr)^{2N([0,T],\omega,1/(4L))}
\\
&\hspace{15pt}
\times 
\Bigl[ \Bigl( \frac{3 \gamma}{4 L} \Bigr)^{n} \, \bigl\vvvert X^{1}(\omega) \bigr\vvvert_{[0,T],w,p}  
+   \Bigl( \frac{3 \gamma}{4 L} \Bigr)^{n/2}  \textcolor{black}{\bigl\langle \vvvert X^{1}(\cdot)\vvvert_{[0,S],w,p} \bigr\rangle_{{16}}}\Bigr] .
\end{split}
\end{equation}
In order to conclude, we notice the following two facts. First, the above estimate remains true if we replace $\bigl\vvvert \Delta X^{n}(\omega) \bigr\vvvert_{[0,S],w,p}$ by $\bigl\vvvert \Delta X^{n}(\omega) \bigr\vvvert_{\star,[0,S],w,p}$ in the left-hand side. Second, Proposition \ref{thm:main:1} guarantees that $\big\langle \vvvert X^{1}(\cdot)\vvvert_{[0,S],w,p}\big\rangle_{{16}}< \infty$. Using a Cauchy like argument, we deduce that, for any $\omega \in \Omega$, the sequence $\big(X^n(\omega),\partial_{x} X^n(\omega),R^{X^n}(\omega)\big)_{n\geq 0}$ is convergent for the norm $\vvvert \cdot \vvvert_{\star,[0,S],w,p}$. Using Proposition \ref{thm:fixed:1}, the limit is a fixed point of $\Gamma$.
\vspace{4pt}

\noindent \textbf{\textbf{Uniqueness --}} Let $\big(X'(\cdot);\delta_{x} X'(\cdot);0\big)$ stand for another fixed point of $\Gamma$, with
$\delta_{x} X'(\omega) = \textrm{F}\bigl(X'(\omega),X'(\cdot)\bigr)$, {for almost every} $\omega \in \Omega$, together with $\langle \vvvert X'(\cdot) \vvvert_{[0,T],w,p} \bigr\rangle_{8} < \infty$. {
In particular, 
we have $\big\langle    \delta_{x} X'(\cdot)   \big\rangle_{\infty} \le \Lambda$.
Allowing} the value of the constant $L_{0}$ to increase, we can assume that
%$\bigl\langle \| X'(\cdot) \|_{[0,T],w,p} \bigr\rangle_{8}^2 \leq \sqrt{L_{0}}$, 
$\bigl\langle \vvvert X'(\cdot) \vvvert_{[0,T],w,p} \bigr\rangle_{8}^2 \leq L_{0}$.
We can also assume that, {for 
${\mathbb P}$-a.e. $\omega$,} 
$\big\vvvert X'(\omega) \big\vvvert_{[t_{i}^0,t_{i+1}^0],w,p}^2 \leq {L_{0}}$, {
with 
$\bigl(t_{i}^0\bigr)_{i=0,\cdots,N^0+1}$
as in 
the statement of Proposition \ref{thm:fixed:1}. 
The proof of
the latter claim is as follows: For a given 
$\omega$ such that $\vert \delta_{x} X'(\omega) \vert \leq \Lambda$ and for 
a given $i \in \{0,\cdots,N^0\}$, 
call $t_{i+1}'$ the first time when $\vvvert X(\omega) \vvvert_{[t_{i}^0,t_{i+1}'],w,p}^2 = L_{0}$. 
If $t_{i+1}' < t_{i+1}^0$, then 
\eqref{eq:starstarstar}
gives
$L_{0} \leq \vvvert X(\omega) \vvvert_{[t_{i}^0,t_{i+1}'],w,p}^2 \leq \gamma + C_{\Lambda,\Lambda} (2 L_{0}+1)/(4L_{0})$, 
which is indeed impossible if $L_{0}$ is large enough.  }

Therefore, we can 
{apply
Proposition  
\ref{thm:fixed:1}
in order to compare $X$ and $X'$ and then}
duplicate the analysis of the convergence sequence, replacing $\Delta X^{n}$ by $\Delta X:=X-X'$. \textcolor{black}{Similar} to \eqref{eq:contraction:proof:10} ({recalling that $X^1$ therein is understood as $\Delta X^0$}), $\bigl\vvvert \Delta X(\omega) \bigr\vvvert_{[0,T],w,p}$ is bounded above by
\begin{equation*}
\begin{split}
&\Bigl(3 \gamma^2 \zeta(\omega)\Bigr)^{2N(\omega,1/(4L))}
\biggl[ \Bigl( \frac{3 \gamma}{4 L} \Bigr)^{n} \bigl\vvvert  \Delta X(\omega) \bigr\vvvert_{[0,T],w,p}   + 
  \sum_{i=1}^{n} \Bigl( \frac{3\gamma}{4L} \Bigr)^{n+1-i} \bigl\langle  \vvvert  \Delta X(\cdot) \vvvert_{[0,T],w,p} \bigr\rangle_{8} \biggr].
\end{split}
\end{equation*}
Letting $n$ tend to $\infty$, this yields
\begin{equation*}
\begin{split}
&\bigl\vvvert \Delta X(\omega) \bigr\vvvert_{[0,T],w,p}
\hspace{15pt}\leq  \bigl(3 \gamma^2 \zeta(\omega)\bigr)^{2N(\omega,1/(4L))} \frac{3\gamma/(4L)}{1 - 3\gamma/(4L)}\, \Bigl\langle \vvvert \Delta X(\cdot) \vvvert_{[0,T],w,p} \Bigr\rangle_{8}.
\end{split}
\end{equation*}
Taking the $\LL^8$ norm, {replacing $T$ by $S$ as in the third step and recalling from 
\eqref{eq:condition:a}
that $\frac{\sqrt{3\gamma/(4L)}}{1-\sqrt{3\gamma/(4L)}} %\Bigl( \frac{3\gamma}{4L} \Bigr)^{1/2} 
\bigl\langle \big(3 {\gamma^2} \zeta_{T}\big)^{2N([0,S],\textcolor{black}{\cdot},1/(4L))} \bigr\rangle_{16}
<1$}, we get uniqueness in small time.
\end{proof}

\noindent \textit{Application to the proof of Theorem \ref{ThmMain}.} Applying iteratively Theorem \ref{main:thm:existence:small:time} along a sequence $(S_{0}=0,\cdots,S_{\ell}=T)$ (shifting in an obvious way $[0,S_{1}]$ into 
$[S_{1},S_{2}]$, $\cdots$) satisfying 
\begin{equation*}
\begin{split}
&\Bigl\langle N([S_{j-1},S_{j}],\cdot,1/(4L_{0}))   \Bigr\rangle_{8}^{2(p-1)/p} \leq {1}, 
\\
\textrm{\rm and}
\qquad &\Bigl\langle \Bigl[ \gamma (1+ w(0,T,\cdot)^{1/p}) \Bigr]^{N([S_{j-1},S_{j}],\cdot,1/(4L))} \Bigr\rangle_{32} \textcolor{black}{\leq \eta},
\end{split}
\end{equation*}
we get {existence and} uniqueness on the whole interval $[0,T]$. 
{We notice that, at each node
$(S_{j})_{j=1,\cdots,\ell}$
 of the subdivision, $\langle X_{S_{j}}(\cdot) \rangle_{2} \leq 
 \langle X_{S_{j-1}}(\cdot) \rangle_{2}
 +
2 \langle \vvvert X \vvvert_{[S_{j-1},S_{j}],w,p} \rangle_{4}
\langle w(0,T,\cdot) \rangle_{4}$, which is finite by a straightforward induction.
By sticking the paths constructed on each subinterval of the subdivision, we indeed 
obtain a random controlled path on the entire $[0,T]$. }
This is Theorem \ref{ThmMain}.
{Importantly, uniqueness holds whatever the choice of $w$ in 
\eqref{eq:w:s:t:omega:ineq}
and
\eqref{eq:useful:inequality:wT}: If $X$ and $X'$ are two solutions, driven by different $w$ and $w'$, 
then we may easily work with $w+w'$, which also satisfies 
\eqref{eq:w:s:t:omega:ineq}
and
\eqref{eq:useful:inequality:wT}.
The control $(w+w')^{1/p}$ also satisfies 
\eqref{EqTailAssumptions}, see for instance 
\eqref{eq:N1:N2} for a simple bound on the local accumulation associated to the sum of 
two different controls $w$ and $w'$. 
}

\section{Uniqueness and Convergence in Law}
\label{weak}

%%----------------------------------------------------------------%%
\subsection{Uniqueness in Law on Strong Rough Set-Ups}
\label{subse:geom}
%%----------------------------------------------------------------%%

Since the solution given by Theorem \ref{main:thm:existence:small:time} is constructed by Picard iteration on each interval $[S_{j-1},S_{j}]$, for $j=1,\cdots, \ell$, we should expect its law to be somehow independent of the probability space used to build the rough set-up ${\boldsymbol W}$.  Recall indeed from \eqref{eq:remainder:integral} the following expansion, which 
 holds true for 
{any rank $n$ in the Picard iteration 
 \eqref{eq:Picard:sequence}}
 and
for any subdivision $0=t_{0}<\cdots<t_{K}=T$, 
\begin{equation}
\label{exp:measurability}
\begin{split}
&X_{t_{i}}^{n+1}(\omega) 
= X_{0}(\omega) + \sum_{j=1}^i \textrm{\rm F}\bigl(X_{t_{j-1}}^n(\omega),X_{t_{j-1}}^n(\cdot)\bigr) W_{t_{j-1},t_{j}}
(\omega)
\\
&\hspace{15pt} + \sum_{j=1}^i 
\partial_{x} \textrm{\rm F}\bigl(X_{t_{j-1}}^n(\omega),X_{t_{j-1}}^n(\cdot)\bigr)
\Bigl( 
\textrm{\rm F}\bigl(X_{t_{j-1}}^n(\omega),X_{t_{j-1}}^n(\cdot)\bigr)
 \WW_{t_{j-1},t_{j}}(\omega)\Bigr)
\\
&\hspace{15pt} + \sum_{j=1}^i \Bigl\langle D_{\mu} \textrm{\rm F}\bigl(X_{t_{j-1}}^n(\omega),X_{t_{j-1}}^n(\cdot)\bigr)\bigl(
X_{t_{j-1}}^n(\cdot)\bigr)
\Bigl(
\textrm{\rm F}\bigl(X_{t_{j-1}}^n(\cdot),X_{t_{j-1}}^n(\cdot)\bigr)
\WW^{\indep}_{t_{j-1},t_{j}}(\cdot,\omega) \Bigr) \Bigr\rangle
\\
&\hspace{15pt} +\sum_{j=1}^i S_{t_{j-1},t_{j}}^{n+1}(\omega);
\end{split}
\end{equation}
the last term converging to $0$ as the step size of the subdivision tends to $0$. 
{In the second line, the matrix product 
$\partial_{x} \textrm{\rm F}\bigl(X^n_{s}(\omega),X_{s}^n(\cdot)\bigr)
\bigl( \textrm{\rm F}\bigl( X^n_{s}(\omega),X^n_{s}(\cdot)\bigr) 
{\mathbb W}_{s,t}(\omega) \bigr)$
should be understood 
as $\bigl(\sum_{\ell=1}^d \sum_{j,k=1}^m \partial_{x_{\ell}}
{\rm F}^{i,j}\bigl(X^n_{s}(\omega),X^n_{s}(\cdot) \bigr) 
\bigl( \textrm{\rm F}^{\ell,k}
\bigl(X^n_{s}(\omega),X^n_{s}(\cdot) \bigr) {\mathbb W}_{s,t}^{k,j}(\omega)
\bigr)\bigr)_{i=1,\cdots,d}$ and similarly for the term on the third line.
Our guess is that the above expansion should permit to identify the law of $X^{n+1}$
and, passing to the limit, to express in a somewhat canonical manner the law of the solution of the mean field rough equation in terms of the law of the rough set-up.} 

However, although it seems to be a relevant concept in our context, uniqueness in law requires some care as the rough set-up explicitly depends upon the underlying probability space $(\Omega,{\mathcal F},\PP)$; recall indeed that the random variables $\Omega \ni \omega \mapsto {\mathbb W}^{\indep}(\omega,\cdot)$ and $\Omega \ni \omega \mapsto {\mathbb W}^{\indep}(\cdot,\omega)$ are not only defined on $(\Omega,{\mathcal F},\PP)$ but also take values in $\LL^q(\Omega,{\mathcal F},\PP;\RR^m)$. The fact that the arrival spaces of both random variables  explicitly depend upon the probability space is a serious drawback to get a form of weak uniqueness. It is thus relevant to identify the canonical information in the rough set-up that is needed to determine the law of the solution. Somehow, we encountered a similar problem in the example of a rough set-up given by Proposition \ref{prop:example:construction}. The difficulty therein is indeed to reconstruct the iterated integral ${\mathbb W}^{\indep}(\omega',\omega)$ from the observation of $W(\omega)$, $W(\omega')$ and ${\mathbb W}(\omega)$; in the proof of Proposition \ref{prop:example:construction}, this is made at the price of an extra source of randomness. Interestingly, things become trivial when ${\mathbb W}^{\indep}(\omega',\omega)$ can be (almost surely) written as the image of $\big(W(\omega),W(\omega')\big)$ by a measurable function. 
Fortunately, all the examples we {may} have {in mind in practice} enter in fact this simpler setting.
For instance, both Examples \ref{ex:brownian} and \ref{example:3} fall within this case. More generally, in the framework of Proposition \ref{prop:example:construction}, we can write $W^{2,1}$ as the almost sure image of $\big(W^1,W^2\big)$ by a measurable function from ${\mathcal C}\big([0,T];\RR^m\big)^2$ into ${\mathcal C}\big({\mathcal S}_{2}^T;\RR^m \otimes \RR^m\big)$, when, for a.e. $\xi \in \Xi$, the quantity $W^{2,1}(\xi)$ can be approximated by the iterated integral of mollified versions of $W^1(\xi)$ and $W^2(\xi)$, provided the mollification procedure defines a measurable map from ${\mathcal C}([0,T];\RR^m)$ into itself. 
The following proposition makes it clear.

\begin{prop}
\label{prop:strong}
Within the framework of \textrm{\rm Proposition} \ref{prop:example:construction},  define, for $1\leq i\leq 2$ and $n\ge 0$, the linear interpolation $W^{i,n}$ of $W^i$ at dyadic points $\big(t_{n}^k = kT/2^n\big)_{k=0,\cdots,2^n-1}$ of $[0,T]$:
\begin{equation*}
W^{i,n}_t(\xi)
:= 
\sum_{k=0}^{2^n-1} 
\left(
W^i_{t_{n}^k}(\xi)
+
W^i_{t_{n}^{k},t_{n}^{k+1}}(\xi)
\frac{2^{n} (t-t_{n}^k)}{T}
\right)
{\mathbf 1}_{[t_{n}^k,t_{n}^{k+1})}(t).
\end{equation*}
If for ${\mathbb Q}$-a.e. $\xi \in \Xi$, for all $(s,t) \in {\mathcal S}_{2}^T$, 
\begin{equation*}
W^{2,1}_{s,t}(\xi) = \lim_{n \rightarrow \infty}
\int_{s,t} 
\Bigl( W^{2,n}_{r}(\xi) - W^{2,n}_{s}(\xi) \Bigr) \otimes d W^{1,n}_{r}(\xi),
\end{equation*}
then there exists a measurable function ${\mathcal I}$ from ${\mathcal C}([0,T];\RR^m)^2$ into ${\mathcal C}\big({\mathcal S}_{2}^T;\RR^m \otimes \RR^m\big)$ such that 
\begin{equation*}
{\mathbb Q} \Bigl( \Bigl\{ \xi \in \Xi : W^{2,1}(\xi) = {\mathcal I}\bigl(W^2(\xi),W^1(\xi)\bigr) \Bigr\} \Bigr) = 1.
\end{equation*}
\end{prop}

The scope of Proposition \ref{prop:strong} is limited to so-called \textit{geometric rough paths}, but the underlying principle is actually more general. This prompts us to introduce the following definition.

\begin{defn}
\label{def:strong}
A \textbf{\textbf{rough set-up}}, as defined in Section \ref{SectionRoughStructure}, is called \textbf{\textbf{strong}} if there exists a measurable mapping ${\mathcal I}$ from ${\mathcal C}\big([0,T];\RR^m\big)^{2}$ into ${\mathcal C}\big({\mathcal S}_{2}^T;\RR^m \otimes \RR^m\big)$ such that
\begin{equation}
\label{eq:representation}
\PP^{\otimes 2} \Bigl( \bigl\{ (\omega,\omega') \in \Omega^2 : {\mathbb W}^{\indep}(\omega,\omega')={\mathcal I}\bigl(W(\omega),W(\omega')\bigr) \bigr\} \Bigr) = 1.
\end{equation}
\end{defn}

So, Proposition \ref{prop:strong} provides a typical instance of strong set-up, which covers in particular  Examples \ref{ex:brownian} and \ref{example:3}. However, it is worth mentioning that strong set-ups may not fall within the scope of Proposition \ref{prop:strong}, since the latter is limited to geometric rough paths, see footnote\footnote{\label{foo:weak} A trivial example of rough set-up is given by the collection of real-valued rough paths
$W^1(\xi) = W^2(\xi) \equiv 0$,
$W^{1,1}(\xi) \equiv 0$, 
$W^{2,1}_{s,t}(\xi) = a(\xi) (t-s)$, 
$(s,t) \in {\mathcal S}_{2}^T$, 
for $\xi$ in a probability space $(\Xi,{\mathcal G},{\mathbb Q})$, where $a$ is a real-valued random variable on $(\Xi,{\mathcal G},{\mathbb Q})$. 
If $a$ is deterministic and non-zero, the set-up is strong but is not geometric. 
If the support of $a$ does not reduce to one point, then the set-up induced by $\big(W^1(\cdot),W^2(\cdot),W^{1,1}(\cdot),W^{2,1}(\cdot)\big)$ is {not strong}.}.

Proposition \ref{prop:example:construction} sheds a light on the rationale for the word \textit{strong} in Definition \ref{def:strong}. Here \textit{strong} has the same meaning as in the theory of strong solutions to stochastic differential equations: The second level $W^{2,1}$ of the rough-path is a measurable function of $(W^2,W^1)$. In contrast, the general set-up considered in the statement of Proposition  \ref{prop:example:construction} may not be strong as $W^{2,1}$ may carry, in addition to $(W^1,W^2)$, an additional external independent randomization. If this additional randomization is not trivial, the set-up should be called \textit{weak}, see again
footnote${}^{\ref{foo:weak}}$ {for a typical instance}. 
 Also, we refer the reader to Deuschel and al. \cite{FrizMaurelli} for a related use of the notion of strong set-up, although the terminology \textit{strong} does not appear therein.

We now have all the ingredients to formulate a weak uniqueness property.

%%%%%%%%%%%%%%%%%%%%%%%%%%%%%%

\begin{thm}
Let 
$X_{0}(\cdot):= \bigl(X_{0}(\omega)\bigr)_{\omega \in \Omega}$, 
$X_{0}'(\cdot):= \bigl(X_{0}'(\omega)\bigr)_{\omega \in \Omega'}$
and
\begin{equation*}
\begin{split}
&{\boldsymbol W}(\cdot):=\bigl(W(\omega),{\mathbb W}(\omega),{\mathbb W}^{\indep}(\omega,\omega')\bigr)_{\omega \in \Omega,\omega' \in \Omega},   \\
&{\boldsymbol W}'(\cdot):=\bigl(W'(\omega),{\mathbb W}'(\omega),{\mathbb W}^{\indep,\prime}(\omega,\omega')\bigr)_{\omega \in \Omega',\omega' \in \Omega'},
\end{split}
\end{equation*} 
be two square integrable initial conditions and two strong rough set-ups with the same parameters $m$, $p$ and $q$, defined on two probability spaces $(\Omega,{\mathcal F},\PP)$ and $(\Omega',{\mathcal F}',{\mathbb P}')$, such that the random variables
\begin{equation*}
\begin{split}
&\Omega^2 \ni (\omega,\omega') \mapsto \bigl(X_{0}(\omega), W(\omega),{\mathbb W}(\omega),W^{\indep}(\omega,\omega')\bigr),   \\
&(\Omega')^2 \ni (\omega,\omega') \mapsto \bigl(X_{0}'(\omega), W'(\omega),{\mathbb W}'(\omega),W^{\indep,\prime}(\omega,\omega')\bigr),
\end{split}
\end{equation*}
have the same law on $\RR^d \times {\mathcal C}([0,T];\RR^m) \times {\mathcal C}({\mathcal S}_{2}^T;\RR^m \otimes \RR^m) \times {\mathcal C}({\mathcal S}_{2}^T;\RR^m \otimes \RR^m)$. Then, the corresponding two solutions $\big(X(\omega)\big)_{\omega \in \Omega}$ and $\big(X'(\omega)\big)_{\omega \in \Omega'}$ to \eqref{EqRDE} have the same law on ${\mathcal C}([0,T];\RR^m)$. 
\end{thm}

As the two set-ups have the same law, we can use the same mapping ${\mathcal I}$ in the representations \eqref{eq:representation} of ${\mathbb W}^{\indep}$ and of ${\mathbb W}^{\indep,\prime}$. Iterating {on $n$} in 
\eqref{exp:measurability}, the result easily follows {by proving, at each rank, that the law of $(W,{\mathbb W},X^n)$ is uniquely determined}.

%%---------------------------------------------------%%
\subsection{Continuity of the It\^o-Lyons Map}
\label{SubsectionContinuity}
%%---------------------------------------------------%%

As expected from a robust solution theory of differential equations, we have continuity of the solution with respect to the parameters in the equation, most notably the rough set-up itself. The next statement quantifies that fact.

\begin{thm}
\label{ThmContinuity}
Let \emph{F} satisfy the same assumptions as in \emph{Theorem \ref{main:thm:existence:small:time}}. 
Given a time interval $[0,T]$ and a sequence of probability spaces $(\Omega_{n},{\mathcal F}_{n},\PP_{n})$, indexed by $n \in {\mathbb N}$, 
let, for any $n$, 
$X^n_{0}(\cdot) := (X^n_{0}(\omega_{n}))_{\omega_{n} \in \Omega_{n}}$
be an $\RR^d$-valued square-integrable initial condition
and
$$
{\boldsymbol W}^n(\cdot) := \Big(W^n(\omega_{n}),\WW^n(\omega_{n}),\WW^{n,\indep}(\omega_{n},\omega_{n}')\Big)_{\omega_{n},\omega_{n}' \in \Omega_{n}}
$$
be an $m$-dimensional rough set-up with corresponding control $w^n$, {as given by 
\eqref{eq:w:s:t:omega},} and local accumulated variation $N^n$, for fixed values of $p \in [2,3)$ and $q > 8$. Assume that
\begin{itemize}
\item[\textcolor{gray}{$\bullet$}] the collection 
$\bigl(\PP_{n} \circ (\vert X^n_{0}(\cdot) \vert^2)^{-1}\bigr)_{n \geq 0}$ is uniformly integrable;

   \item[\textcolor{gray}{$\bullet$}] for positive constants $\varepsilon_{1},c_{1}$ and $(\varepsilon_{2}(\alpha),c_{2}(\alpha))_{\alpha >0}$,  the tail assumption \eqref{EqTailAssumptions} holds for $w^n$ and $N^n$, for all $n\geq 0$;

   \item[\textcolor{gray}{$\bullet$}] associating a control $v^n$ with each ${\boldsymbol W}^n(\cdot)$ as in \eqref{eq:v}, the functions 
   $
   \bigl({\mathcal S}_{2}^T \ni (s,t) \mapsto \langle v^n(s,t,\cdot) \rangle_{2q}\bigr)_{n \geq 0}$ 
   are uniformly Lipschitz continuous. %      \item[\textcolor{gray}{$\bullet$}] the sequence 
%$\Bigl( \Bigl\llangle 
%\bigl\| {\mathbb W}^{n,\indep}(\cdot,\cdot) 
%\bigr\|_{[0,T],p/2-\textrm{\rm var}}
%\Bigr\rrangle_{q}
%\Bigr)_{n \geq 0}$
%is bounded. 

\end{itemize}
Assume also that there exist, on another probability space
$(\Omega,{\mathcal F},\PP)$, 
a square integrable initial condition $X_{0}(\cdot)$
with values in $\RR^d$ and 
 a strong rough set-up
$$
{\boldsymbol W}(\cdot) := \Big(W(\omega),\WW(\omega),\WW^{\indep}(\omega,\omega')\Big)_{\omega,\omega' \in \Omega}
$$
with values in $\RR^m$, such that
the law under the probability measure $\PP_{n}^{\otimes 2}$ 
of the random variable
$\Omega_{n}^2 \ni (\omega_{n},\omega_{n}')
\mapsto \bigl(X_{0}^n(\omega_{n}), W^n(\omega_{n}),{\mathbb W}_{n}(\omega_{n}),{\mathbb W}^{\indep}_{n}(\omega_{n},\omega_{n}')\bigr)$, 
seen as a random variable with values in the space ${\RR^d \times} {\mathcal C}([0,T];\RR^m) \times 
\bigl\{ {\mathcal C}({\mathcal S}_{2}^T;\RR^m \otimes \RR^m)
\bigr\}^2$, 
converges in the weak sense to the law of $
\Omega^2 \ni (\omega,\omega')
\mapsto \bigl(X_{0}(\omega), W(\omega),{\mathbb W}(\omega_{n}),{\mathbb W}^{\indep}(\omega,\omega')\bigr)$.

Then, ${\boldsymbol W}(\cdot)$ satisfies the requirements of \emph{Theorem \ref{main:thm:existence:small:time}}
for some $p' \in (p,3)$ and $q' \in [8,q)$, {with 
$w$ therein
being given by 
\eqref{eq:w:s:t:omega}}. Moreover, if $X^n(\cdot)$, resp. $X(\cdot)$, is the solution of the mean field rough differential equation driven by ${\boldsymbol W}^n(\cdot)$, resp. ${\boldsymbol W}(\cdot)$, then $X^n(\cdot)$ converges in law to $X(\cdot)$ on ${\mathcal C}([0,T];\RR^d)$.
\end{thm}

%The above assumptions may seem stronger than needed. Indeed, in our basic definition of a rough set-up, the regularity of the block ${\mathbb W}^{\indep}$ is mostly considered through the regularity of the sole random variables $\Omega \ni \omega \mapsto {\mathbb W}^{\indep}(\omega,\cdot)$ and $\Omega \ni \omega \mapsto {\mathbb W}^{\indep}(\cdot,\omega)$; in contrast, the regularity of the \textit{full-fledged} random variable $\Omega^2 \ni (\omega,\omega') \mapsto {\mathbb W}^{\indep}(\omega,\omega')\in {\mathcal C}({\mathcal S}_{2}^T;\RR^m \otimes \RR^m)$ does not show up explicitly. So, the fact that we require, in the above statement, the weak convergence of the sequence $({\mathbb W}^{n,\indep})_{n \geq 0}$ may be too much. 
%According to the usual version of the It\^o-Lyons continuity theorem, one could indeed expect the conclusion to hold under the weaker assumption that the $\RR^m \times {\mathbb L}^q(\Omega,{\mathcal F},\PP;\RR^m)$-valued random variables  $({\mathbb W}^{n,\indep}(\omega,\cdot))_{n \geq 0}$ converge almost surely in the space of $p$-variation paths, and similarly for $({\mathbb W}^{n,\indep}(\cdot,\omega))_{n \geq 0}$. 

The rationale for the framework and the assumptions used in the statement of Theorem \ref{ThmContinuity} is two-fold. First, it allows for a 
proof based on compactness arguments; in particular, the proof completely bypasses any lengthy stability estimate of the paths with respect to the rough structure, which, in our extended framework, would be especially cumbersome. Also, this compactness argument is pretty interesting in itself and complements quite well Subsection \ref{subse:geom} on weak uniqueness; noticeably, it allows the 
set-ups to be supported by different probability spaces.  Second, our formulation of the continuity of the It\^o-Lyons map turns out to be well-fitted to the applications  addressed in our {forthcoming} companion paper 
\cite{BCDParticleSystem}, {see also Section 4 in the earlier version \cite{BCDArxiv}}. 

The assumption that the limiting rough set-up is strong is tailored-made to the compactness arguments we use below as it permits to 
pass quite simply to the weak limit along the laws of the rough set-ups $({\boldsymbol W}^n(\cdot))_{n \geq 0}$
and to identify the limiting law.

\begin{proof} Throughout the proof, we call $p \in [2,3)$ 
and $q >8$
the fixed indices used to define the set-ups and, in particular, to control the variations in 
the definition 
\eqref{EqTailAssumptions} of each $w^n$, $n \geq 0$, 
{$w^n$ being associated with $v^n$ through 
\eqref{eq:w:s:t:omega}}. This is important because, at some points of the proof, we will use other values $p'>p$ and $q'<q$.

\textbf{\textbf{Step 1.}} 
We prove key properties on the tightness of the sequence 
$({\boldsymbol W}^{n}(\cdot))_{n \geq 0}$. 
\smallskip

\textbf{\textbf{1a.}}
For any $n \geq 0$, we introduce the modulus of continuity of $(W^n(\cdot),{\mathbb W}^n(\cdot),{\mathbb W}^{n,\indep}(\cdot))$, namely we let, for any $\delta >0$,
\begin{equation*}
\begin{split}
&\varsigma^{n}\bigl(\delta,\omega_{n},\omega_{n}' \bigr) 
 := 
\sup_{\vert s-t \vert \leq \delta}
\vert 
W_{t}^n(\omega_{n}) - 
W_{s}^n(\omega_{n}) \vert 
\\
&\hspace{15pt} + 
\sup_{\vert s-{s'} \vert + \vert {t}-t' \vert \leq \delta}
 \bigl\vert 
{\mathbb W}_{s',t'}^n(\omega_{n}) 
-
{\mathbb W}_{s,t}^n(\omega_{n}) 
\bigr\vert 
+
\sup_{\vert s-{s'} \vert + \vert {t}-t' \vert \leq \delta}
\bigl\vert 
{\mathbb W}_{s',t'}^{n,\indep}(\omega_{n},\omega_{n}') 
- 
{\mathbb W}_{s',t'}^{n,\indep}(\omega_{n},\omega_{n}') 
\bigr\vert,
\end{split}
\end{equation*}
where $(\omega_{n},\omega_{n}') \in \Omega_{n}^2$. 
Since 
the laws of the processes 
$(W^n(\cdot),{\mathbb W}^n(\cdot),
{\mathbb W}^{n,\indep}(\cdot,\cdot))_{n \geq 0}$
are tight in the space 
${\mathcal C}([0,T];\RR^m) \times 
\bigl\{
{\mathcal C}({\mathcal S}_{2}^T;\RR^m \otimes \RR^m) \bigr\}^2$, we deduce that 
\begin{equation*}
\forall \varepsilon >0,
\quad \lim_{\delta \searrow 0}
\sup_{n \geq 0}
{\mathbb P}_{n}^{\otimes 2}
\Bigl( \bigl\{ (\omega_{n},\omega_{n}') 
\in \Omega_{n}^2 : 
\varsigma_{n}\bigl(\delta,\omega_{n},\omega_{n}' \bigr) \geq \varepsilon 
\bigr\}
\Bigr) = 0.
\end{equation*}

\textbf{\textbf{1b.}}
We now prove that, for any $q'\in [8,q)$, the laws of the processes $\bigl(\Omega_{n} \ni \omega_{n}
\mapsto \langle {\mathbb W}^{n,\indep}(\omega_{n},\cdot) \rangle_{q'} \bigr)_{n \geq 0}$
are tight\footnote{{In the notation $\langle \cdot \rangle_{q'}$, the expectation is implicitly taken under 
${\mathbb P}_{n}$}.}, and similarly for the laws of the processes 
$\bigl(\Omega_{n} \ni \omega_{n}
\mapsto \langle {\mathbb W}^{n,\indep}(\cdot,\omega_{n}) \rangle_{q'} \bigr)_{n \geq 0}$. 
By \eqref{eq:w:s:t:omega}, we have, for any $\omega_{n}
\in \Omega_{n}$,
\begin{equation*}
\sup_{(s,t) \in {\mathcal S}_{2}^T}
  \bigl\langle {\mathbb W}_{s,t}^{n,\indep}(\omega_{n},\cdot) \bigr\rangle_{q}
\leq \bigl( w^{n}(0,T,\omega_{n}) \bigr)^{{2/p}}.
\end{equation*}
By the {second} bullet point in the assumption, the tails of the right-hand side are uniformly dominated. So,
\begin{equation}
\label{eq:tightness:1}
\lim_{A \rightarrow \infty}
\sup_{n \geq 0}
\PP_{{n}} \Bigl( \bigl\{ \omega_{n} \in \Omega_{n}
: 
\sup_{(s,t) \in {\mathcal S}_{2}^T} \bigl\langle {\mathbb W}_{s,t}^{n,\indep}(\omega_{n},\cdot) \bigr\rangle_{q}
\geq A \bigr\}
\Bigr) =0,
\end{equation}
which is one first step in the proof of tightness. 
For any $a>0$, we now consider the  
event
\begin{equation*}
E_{n}(\delta,a) := \Bigl\{ \omega_{n}
\in \Omega_{n} : 
\PP_{n}
\Bigl( \bigl\{ 
\omega_{n}'\in \Omega_{n} : 
\varsigma_{n}(\delta,\omega_{n},\omega_{n}')
\geq \varepsilon 
\bigr\} 
\Bigr) \geq a\Bigr\}.
\end{equation*}
By Markov's inequality and then Fubini's theorem, 
\begin{equation*}
\PP_{n}\bigl(E_{n}(\delta,a)\bigr) 
\leq a^{-1}
\PP_{n}^{\otimes 2}
\Bigl( 
\bigl\{
(\omega_{n},\omega_{n}') \in 
\Omega_{n}^2 : 
\varsigma_{n}(\delta,\omega_{n},\omega_{n}')
\geq \varepsilon 
\bigr\}
\Bigr),
\end{equation*}
{the right-hand side converging to $0$ as $n$ tends to $\infty$.}
Clearly, for any $\varepsilon >0$, we can
find a collection of positive reals $(a_{\varepsilon}(\delta))_{\delta >0}$ such that 
\begin{equation*}
\lim_{\delta \searrow 0} a_{\varepsilon}(\delta) = 0,
\quad \textrm{\rm and}
\quad
\lim_{\delta \searrow 0}
\PP_{n} \Bigl(E_{n}\bigl(\delta,a_{\varepsilon}(\delta)\bigr) \Bigr) =0.
\end{equation*}
Take now $\omega_{n} \in E_{n}(\delta,a_{\varepsilon}(\delta))^{\complement}$ such that 
$
\sup_{(s,t) \in {\mathcal S}_{2}^T} \bigl\langle {\mathbb W}_{s,t}^{n,\indep}(\omega_{n},\cdot) \bigr\rangle_{q} \leq A$, 
for a given $A>0$. Then, for any $q' \in [8,q)$ and $(s,t),(s',t') \in {\mathcal S}_{2}^T$ {with 
$\vert s-s' \vert + \vert t-t' \vert \leq \delta$},
\begin{equation*}
\begin{split}
&\Bigl\vert 
\bigl\langle {\mathbb W}_{s',t'}^{n,\indep}(\omega_{n},\cdot) \bigr\rangle_{q'}
-
\bigl\langle {\mathbb W}_{s,t}^{n,\indep}(\omega_{n},\cdot) \bigr\rangle_{q'}
\Bigr\vert
\\
&\leq 
\Bigl\langle {\mathbb W}_{s',t'}^{n,\indep}(\omega_{n},\cdot) 
-
 {\mathbb W}_{s,t}^{n,\indep}(\omega_{n},\cdot) \Bigr\rangle_{q'}
\leq \varepsilon + A a_{\varepsilon}(\delta)^{1-q'/q}. 
\end{split}
\end{equation*}
For $A$ fixed and $\delta$ small enough, the right-hand side is less than $2 \varepsilon$. We easily deduce that, 
for any $\varepsilon >0$,  
\begin{equation*}
\lim_{\delta \searrow 0}
\sup_{n \geq 0}
\PP_{n} 
\biggl(
\Bigl\{ 
\omega_{n} \in 
\Omega_{n} : 
{\sup_{\vert s-s' \vert + \vert t-t' \vert \leq \delta}} \,
\Bigl\vert 
\bigl\langle {\mathbb W}_{s',t'}^{n,\indep}(\omega_{n},\cdot) \bigr\rangle_{q'}
-
\bigl\langle {\mathbb W}_{s,t}^{n,\indep}(\omega_{n},\cdot) \bigr\rangle_{q'}
\Bigr\vert
\geq \varepsilon
\Bigr\}
\biggr)=0,
\end{equation*}
{which, together with 
\eqref{eq:tightness:1}, proves tightness}.
Clearly, the same holds for the family $\bigl(\Omega_{n} \ni \omega_{n} \mapsto \langle {\mathbb W}^{n,\indep}(\cdot,\omega_{n}) \rangle_{q'}\bigr)_{n \geq 0}$. {Similarly}, the two deterministic functions 
$\bigl(\langle W^n(\cdot) \rangle_{q'}\bigr)_{n \geq 0}$
and 
$\bigl(\llangle W^{n,\perp}(\cdot,\cdot) \rrangle_{q'}\bigr)_{n \geq 0}$ are relatively compact in ${\mathcal C}([0,T];\RR)$ and 
${\mathcal C}({\mathcal S}_{2}^T;\RR)$. 
\smallskip

\textbf{\textbf{1c.}}
For each coordinate of the family of processes
$$\Bigl(
   \Omega_{n} \ni \omega_{n}
   \mapsto
   \big( \vert W^n_{s,t}(\omega_{n})
   \vert , \vert {\mathbb W}^n_{s,t}(\omega_{n})
   \vert,
\bigl\langle {\mathbb W}^{n,\indep}_{s,t}(\omega_{n},\cdot)
\bigr\rangle_{q'}, 
\bigl\langle {\mathbb W}^{n,\indep}_{s,t}(\cdot,\omega_{n})
\bigr\rangle_{q'}
\bigr)_{(s,t) \in {\mathcal S}^2_{T}}\Big)_{n \geq 0},$$
we know that the corresponding family of laws is tight in
${\mathcal C}({\mathcal S}_{2}^T;\RR)$
and that the associated family of $p$-variations over $[0,T]$ has tight laws in $\RR$ (because of the second item in the assumption).
Hence, we can apply Lemma \ref{lem:aux:1} below, with any $p' \in (p,3)$ instead of $p$ itself, and with $Z^n_{s,t}(\omega)$ equal to one of the coordinate of the above process.

We proceed in the same way with the coordinates of the deterministic sequence   
$\bigl(z^n_{s,t}= \bigl(\bigl\langle W^n_{s,t}(\cdot)
\bigr\rangle_{q'}, \llangle {\mathbb W}^{n,\indep}_{s,t}(\cdot,\cdot)
\rrangle_{q'} \bigr)_{(s,t) \in {\mathcal S}_{2}^T}\bigr)_{n \geq 0}$.
We deduce that, for any $p' \in (p,3)$, the sequence of probability measures $\Big(\PP \circ ({\mathcal S}_{2}^T \ni (s,t) \mapsto v^{n,\prime}(s,t,\cdot))^{-1}\Big)_{n \geq 0}$ is tight in ${\mathcal C}({\mathcal S}^2_{T};\RR)$ and hence that
\begin{equation*}
\forall \varepsilon >0, \quad \lim_{\delta \searrow 0} \sup_{n \geq 0} \, \PP_{n}\biggl( \sup_{(s,t) \in {\mathcal S}_{2}^T : 
t-s \le \delta} v^{n,\prime}(s,t,\cdot) > \varepsilon \biggr) = 0,
\end{equation*}
where $v^{n,\prime}$ is associated with ${\boldsymbol W}^n(\cdot)$ through \eqref{eq:v} {using} the pair of parameters $(p',q')$ instead of $(p,q)$.
\smallskip

\textbf{\textbf{1d.}}
 Obviously, $v^{n,\prime}(s,t,\cdot) \leq (v^n(s,t,\cdot))^{p'/p}$. Since $p'/p \leq 2$ and   
%\begin{itemize}
%   \item[\textcolor{gray}{$\bullet$}] the tails of $w^n \geq v^n$ decay faster than any polynomial function, uniformly in $n \geq 0$;    
  % \item[\textcolor{gray}{$\bullet$}] 
  the function ${\mathcal S}_{2}^T \ni (s,t) \mapsto \langle v^n(s,t,\cdot) \rangle_{2q}$ is Lipschitz continuous, uniformly in $n \geq 0$,
%   \end{itemize}
we deduce that $(s,t) \mapsto \langle v^{n,\prime}(s,t,\cdot) \rangle_{q}$ is Lipschitz continuous, uniformly in $n \geq 0$. Hence, 
\begin{equation*}
\forall \varepsilon >0, \quad \lim_{\delta \rightarrow 0} \sup_{n \geq 0} \, \PP_{n} \biggl( \sup_{(s,t) \in {\mathcal S}_{2}^T : 
t-s \le \delta}
w^{n,\prime}(s,t,\cdot) > \varepsilon \biggr) = 0,
\end{equation*}
where, as above, {$w^{n,\prime}$ 
is associated with 
$v^{n,\prime}$
and $(p',q')$ through
\eqref{eq:w:s:t:omega}}. Importantly, we deduce from the bound $(v^{n,\prime}(0,T,\cdot))^{1/p'} \leq (v^n(0,T,\cdot))^{1/p}$ that, similar to $w^n$ and $N^n$ (the latter is associated with $w^n$ through 
\eqref{eq:N:s:t:omega}), the function $w^{n,\prime}$ and the corresponding local accumulated variation $N^{n,\prime}$
(given by 
\eqref{eq:N:s:t:omega}
with $\varpi = w^{n,\prime}$)
 satisfy the tail assumption \eqref{EqTailAssumptions}, uniformly in $n \geq 0$. The bound on the tails of $N^{n,\prime}$ is easily obtained by comparison with the tails of 
 $N^{n}$. 
 \medskip

\textbf{\textbf{Step 2.}}

\textbf{\textbf{2a.}} 
The next step is to observe, as a corollary of the proof of Theorem \ref{main:thm:existence:small:time}, see \eqref{eq:conclusion:preuve:contraction}, that there exist a constant $C$ and a real $S>0$ such that, for all $n \geq 0$,
\begin{equation*}
\Bigl\langle \vvvert X^n(\cdot) \vvvert_{[0,S],w^{n,\prime},p'} \Bigr\rangle_{8} \leq C. 
\end{equation*}
The fact that $C$ and $S$ can be chosen independently of $n$ is a consequence of the fact that the tails of $N^n$ and $w^{n}$ are controlled uniformly in $n \geq 0$. Here $S$ is chosen small enough so that  \eqref{eq:constraint:1:E!:small:time} and \eqref{eq:constraint:2:E!:small:time} in the statement
{of Theorem 
\ref{main:thm:existence:small:time}}
 are satisfied, uniformly in $n \geq 0$. 
\smallskip

\textbf{\textbf{2b.}} Arguing as in the derivation of Theorem \ref{ThmMain} from the statement of Theorem \ref{main:thm:existence:small:time}, we can iterate the argument and construct a sequence of deterministic times $0=S_{0}<S=S_{1}<\ldots<S_{K}=T$, for some deterministic $K \geq 1$, such that, for all $n \geq 0$ and all $j \in \{0,\cdots,K-1\}$, 
$\bigl\langle \vvvert X^n(\cdot) \vvvert_{[S_{j},S_{j+1}],w^{n,\prime},p'} \bigr\rangle_{8} \leq C$. Up to a modification of the constant $C$, we deduce that, for all $n \geq 1$, $
\bigl\langle \vvvert X^n(\cdot) \vvvert_{[0,T],w^{n,\prime},p'} \bigr\rangle_{8} \leq C$. 
Recalling that $\bigl( \PP_{n} \circ (\vert X_{0}^n(\cdot) \vert^2)^{-1} \bigr)_{n \geq 0}$
is uniformly integrable, it is easily checked that $\bigl(\PP_{n} \circ (\sup_{0 \leq t \leq T}
\vert X_{t}^n(\cdot) \vert^2)^{-1}\bigr)_{n \geq 0}$ is also uniformly integrable. 
\smallskip

\textbf{\textbf{2c.}} As another result of the previous step, for any $\epsilon >0$, we can find $a>0$ such that 
\begin{equation*}
\sup_{n \geq 0}\,
{\mathbb P}_{n}
\Bigl( 
\vvvert X^n(\cdot) \vvvert_{[0,T],w^{n,\prime},p'}
> a \Bigr) \leq \varepsilon,
\end{equation*}
from which,  we deduce that
 \begin{equation*}
\forall a>0, \quad \exists \varepsilon >0 \ : \ \sup_{n \geq 0} \, {\mathbb P}_{n} \Bigl( \forall (s,t) \in {\mathcal S}_{2}^T, \vert X^n_{s,t} \vert^{p'} > a w^{n,\prime}(s,t) \Bigr) \leq \varepsilon.
\end{equation*}
Combining with 
\textbf{\textbf{1d}}, this yields
\begin{equation*}
\forall \varepsilon >0, \quad \lim_{\delta \rightarrow 0} \sup_{n \geq 0}\, \PP_{n} \biggl( \sup_{(s,t) \in {\mathcal S}_{2}^T :
t-s \le \delta} \vert X^n_{s,t} \vert > \varepsilon \biggr) = 0.
\end{equation*}
From the conclusion of \textbf{\textbf{2b}}, the sequence $\big(\PP_{n} \circ (X^n(\cdot))^{-1}\big)_{n \geq 0}$ is tight in ${\mathcal C}\big([0,T];\RR^d\big)$. 
\medskip

\textbf{\textbf{Step 3.}}

\textbf{\textbf{3a.}} As a consequence of the assumptions of Theorem \ref{ThmContinuity} and of Step 2, we have the following tightness properties:    

\begin{enumerate}
   \item[\textcolor{gray}{$\bullet$}] $\bigl(\PP_{n} \circ (W^n(\cdot))^{-1}\bigr)_{n \geq 0}$ and $\bigl(\PP_{n} \circ (X^n(\cdot))^{-1}\bigr)_{n \geq 0}$ are tight in the spaces ${\mathcal C}\big([0,T];\RR^m\big)$ and ${\mathcal C}\big([0,T];\RR^d\big)$;    
   \item[\textcolor{gray}{$\bullet$}] $\bigl(\PP_{n} \circ ({\mathbb W}^n)^{-1}(\cdot)\bigr)_{n \geq 0}$ is tight in ${\mathcal C}\big({\mathcal S}_{2}^T;\RR^m \otimes \RR^m\big)$;    
   \item[\textcolor{gray}{$\bullet$}] 
   $
   \Bigl(\PP_{n}^{\otimes 2} \circ \Big(\Omega_{n}^2 \ni (\omega_{n},\omega_{n}') \mapsto {\mathbb W}^{n,\indep}(\omega_{n},\omega_{n}') \in {\mathcal C}({\mathcal S}_{2}^T; \RR^m \otimes \RR^m)\Big)^{-1}\Bigr)_{n \geq 0}
   $ 
   is tight in ${\mathcal C}\big({\mathcal S}_{2}^T;\RR^m \otimes \RR^m\big)$;

   \item[\textcolor{gray}{$\bullet$}]    
   $\Bigl(\PP_{n} \circ \Big(v^{n,\prime} (\omega_{n}) : \Omega_{n} \ni
   \omega_{n}
   \mapsto 
   \bigl( {\mathcal S}_{2}^T \ni (s,t)
    \mapsto v^{n,\prime}(s,t,\omega_{n})
    \bigr)  \in {\mathcal C}({\mathcal S}_{2}^T; \RR )\Big)^{-1}\Bigr)_{n \geq 0}
   $
   is tight in ${\mathcal C}\big({\mathcal S}_{2}^T;\RR \big)$;

\end{enumerate}

\textbf{\textbf{3b.}} By Skorokhod's representation theorem, we can find  an auxiliary Polish probability space $\big(\widehat{\Omega},\widehat{\mathcal F},\widehat{\mathbb P}\big)$, such that, up to a subsequence, for $\widehat{\PP}$-a.e. $\widehat \omega \in \widehat{\Omega}$,
\begin{flalign}
&\lim_{n \rightarrow \infty}\, \Bigl( \widehat W^{n,1}(\widehat \omega),  \widehat W^{n,2}(\widehat \omega), \widehat W^{n,1,1}(\widehat \omega),\widehat {W}^{n,2,1}(\widehat \omega), \widehat{v}^{n,1,\prime}(\widehat \omega),
\widehat{v}^{n,2,\prime}(\widehat \omega), \widehat X^{n,1}(\widehat \omega),
\widehat X^{n,2}(\widehat \omega)
\Bigr)  \nonumber
 \\
&= \Bigl( \widehat W^1(\widehat \omega), \widehat W^2(\widehat \omega), \widehat {W}^{1,1}(\widehat \omega),\widehat {W}^{2,1}(\widehat \omega), \widehat{v}^{1,\prime}(\widehat \omega),\widehat{v}^{2,\prime}(\widehat \omega) 
, \widehat{X}^1(\widehat \omega), \widehat{X}^2(\widehat \omega)
\Bigr),
\label{eq:cv:n}
\end{flalign} 
where $\big(\widehat{W}^{n,1},\widehat{W}^{n,2},\widehat{W}^{n,1,1},\widehat{W}^{n,2,1},  \widehat{v}^{n,1,\prime}(\widehat \omega),
\widehat{v}^{n,2,\prime}(\widehat \omega),
\widehat X^{n,1}(\widehat \omega),
\widehat X^{n,2}(\widehat \omega)
\big)$ has the same law as the random variable
\begin{equation*}
\begin{split}
&\Omega_{n}^2 \ni (\omega_{n},\omega_{n}')
\\
&\hspace{5pt} \mapsto \Bigl(W^n(\omega_{n}),W^n(\omega_{n}'),{\mathbb W}^n(\omega_{n}),{\mathbb W}^{n,\indep}(\omega_{n}',\omega_{n}), v^{n,\prime}(\omega_{n}),
v^{n,\prime}(\omega_{n}'),
X^n(\omega_{n}),X^n(\omega_{n}') \Bigr),
\end{split}
\end{equation*}
which takes values in the space $
\bigl\{
{\mathcal C}([0,T];\RR^m) \bigr\}^2 \times \bigl\{ {\mathcal C}({\mathcal S}_{2}^T;\RR^m \otimes \RR^m)
\bigr\}^2  
\times \bigl\{ {\mathcal C}({\mathcal S}_{2}^T;\RR)\bigr\}^2 \times \bigl\{ {\mathcal C}([0,T];\RR^d)  \bigr\}^2$, 
and where $\bigl(\widehat{W}^1(\cdot),\widehat{W}^2(\cdot),\widehat{W}^{1,1}(\cdot),\widehat{W}^{2,1}(\cdot),X_{0}^1(\cdot)\bigr)$ has the same law as the random variable
\begin{equation}
\label{eq:limit:law}
\Omega^2 \ni (\omega,\omega') \mapsto \Bigl(W(\omega),W(\omega'),{\mathbb W}(\omega),{\mathbb W}^{\indep}(\omega',\omega),X_{0}(\omega)\Bigr).
\end{equation}
%By the third item in the assumption, we know that the sequence 
%$$
%\left(\PP_{n}^{\otimes 2} \circ \Big( \| W^n(\cdot)\|_{[0,T],p-\textrm{\rm var}} + \| {\mathbb W}^n(\cdot)\|_{[0,T],p/2-\textrm{\rm var}} + \| {\mathbb W}^{n,\indep}(\cdot,\cdot)\|_{[0,T],p/2-\textrm{\rm var}} \Big)^{-1}\right)_{n \geq 0}
%$$ 
%is tight in $\RR$. So, 
%recalling that $p$- and $p/2$-variations on 
%${\mathcal C}([0,T];\RR^m)$
%and 
%${\mathcal C}({\mathcal S}_{2}^T; 
%\RR^m \otimes \RR^m)$ are lower-semicontinuous
%and thus measurable 
%and 
%following the proof of Proposition 5.5 in Friz and Victoir's book \cite{FrizVictoirBook}, we can assume that all the coordinates related to the rough set-up in \eqref{eq:cv:n} converge in $p'$-variation. The last two coordinates in \eqref{eq:cv:n} converge for the uniform topology. 

%Last, observe that, since each set-up ${\boldsymbol W}^n$, for $n \geq 0$, is strong, $v_{p'}^n$ may be expressed as a measurable function of $(W^n,{\mathbb W}^n)$, say $\PP_{n}\big(\big\{ \omega_{n} \in \Omega_{n} : v_{p'}^n(\omega_{n}) = M^n(W^n(\omega_{n}),{\mathbb W}^n(\omega_{n}))\big\}\big)=1$, for some measurable $M^n : {\mathcal C}([0,T];\RR^m) \times {\mathcal C}([0,T];{\mathcal S}_{2}^T) 
%\rightarrow {\mathcal C}({\mathcal S}_{2}^T; \RR^m \otimes \RR^m)$. In particular, $\widehat{v}^{\prime}^n(\widehat{\omega}) := M^n(\widehat{W}^{n,1}(\widehat{\omega}),\widehat{W}^{n,1,1}(\widehat{\omega}))$, for
%almost every $\widehat{\omega} \in \widehat{\Omega}$.  
%

\textbf{\textbf{3c.}} At this point of the proof, the difficulty is that $\big(\widehat{W}^1(\cdot),\widehat{W}^2(\cdot),\widehat{W}^{1,1}(\cdot),\widehat{W}^{2,1}(\cdot)\big)$ does not form a rough set-up. Still, we have the following two properties. First, using the fact that the limiting set-up is strong, we have
\begin{equation*}
\widehat{\PP}
\Bigl( \Bigl\{ \widehat{\omega} \in \widehat{\Omega} : {\widehat{W}}^{2,1}(\widehat{\omega}) = {\mathcal I}\bigl({W}^{2}(\widehat{\omega}),
W^1(\widehat{\omega}) \bigr) \Bigr\} \Bigr) = 1,
\end{equation*}
for a measurable mapping ${\mathcal I} : {\mathcal C}([0,T];\RR^m)^2 \rightarrow {\mathcal C}({\mathcal S}_{2}^T;\RR^m \otimes \RR^m)$, which follows from the identification with the law of 
\eqref{eq:limit:law}. Also, passing to the limit in Chen's relations satisfied by each ${\boldsymbol W}^n$, we have, for $\widehat{\mathbb P}$-a.e. $\widehat{\omega} \in \widehat{\Omega}$, and all $0 \leq r \leq s \leq t \leq T$,
\begin{equation*}
\begin{split}
&\widehat W^{1,1}_{r,t}(\widehat \omega) = \widehat W^{1,1}_{r,s}(\widehat \omega) + \widehat W^{1,1}_{s,t}(\widehat \omega) + \widehat W^1_{r,s}(\widehat \omega) \otimes \widehat W_{s,t}^1(\widehat \omega),   
\\
&\widehat W_{r,t}^{2,1}(\widehat\omega) = \widehat W^{2,1}_{r,s}(\widehat \omega) + \widehat W_{s,t}^{2,1}(\widehat \omega) + \widehat W_{r,s}^{2}(\widehat \omega) \otimes \widehat W_{s,t}^1(\widehat \omega).
\end{split}
\end{equation*}
Obviously, $(\widehat W^2,\widehat X^2)$ is independent of $\big(\widehat W^1,\widehat W^{1,1},\widehat{X}^1,\widehat{v}^{1,\prime}\big)$. Following the proof of Proposition \ref{prop:example:construction}, {but in} a simpler setting here since the limiting rough set-up is strong, we can find  

\begin{itemize}
   \item[\textcolor{gray}{$\bullet$}] four random variables $\widehat{W}(\cdot)$, $\widehat{\mathbb W}(\cdot)$,
 $\widehat{v}^{\prime}(\cdot)$     
 and
 $\widehat{X}(\cdot)$
    from $\big(\widehat{\Omega},\widehat{\mathcal F},\widehat{\PP}\big)$ into  ${\mathcal C}([0,T];\RR^m)$, ${\mathcal C}\big({\mathcal S}_{2}^T;\RR^m \otimes \RR^m\big)$,
    ${\mathcal C}\big({\mathcal S}_{2}^T;\RR \big)$
    and ${\mathcal C}([0,T];\RR^d)$
      such that 
\begin{equation*}
\begin{split}
&\widehat{\PP}\Bigl( \Bigl\{ \widehat{\omega} \in \widehat{\Omega} :\bigl( \widehat{W},
\widehat{\mathbb W},\widehat{v}^{\prime},\widehat X\bigr)(\widehat \omega)
 = \bigl( W^1 , W^{1,1},\widehat{v}^{1,\prime},
 \widehat X^1\bigr)(\widehat \omega)
\Bigr\} \Bigr) = 1;
\end{split}
\end{equation*}

   \item[\textcolor{gray}{$\bullet$}] a random variable $\widehat{\mathbb W}^{\indep}(\cdot,\cdot)$ from $\big(\widehat{\Omega}^2,\widehat{\mathcal F}^{\otimes 2},\widehat{\PP}^{\otimes 2}\big)$ into ${\mathcal C}\big({\mathcal S}_{2}^T;\RR^m \otimes \RR^m\big)$ such that
\begin{equation}
\label{eq:cv:strong}
\begin{split}
&\widehat{\PP}^{\otimes 2}\Bigl( \Bigl\{ (\widehat{\omega},\widehat{\omega}') \in \widehat{\Omega}^2 : \widehat{\mathbb W}^{\indep} (\widehat{\omega},\widehat{\omega}') = {\mathcal I}\bigl( \widehat{W}(\widehat{\omega}), \widehat{W}(\widehat{\omega}') \bigr) \Bigr\} \Bigr) = 1;
\end{split}
\end{equation}
\end{itemize}
the rough set-up $\widehat{\boldsymbol W}(\cdot) := \big(\widehat{W}(\cdot),\widehat{\mathbb W}(\cdot),\widehat{\mathbb W}^{\indep}(\cdot,\cdot)\big)$ satisfying \eqref{eq:chen} with probability 1 and  
$\widehat \Omega^2 \ni (\widehat \omega,\widehat \omega') \mapsto \big(
\widehat{W}(\widehat \omega),
\widehat{W}(\widehat \omega'),
\widehat{\mathbb W}(\widehat \omega),\widehat{\mathbb W}^{\indep}(\widehat \omega',\widehat \omega),
\widehat{v}^{\prime}(\widehat \omega),
\widehat{v}^{\prime}(\widehat \omega'),
\widehat{X}(\widehat \omega),\widehat X(\widehat \omega')\big)$ having the same law as $\big(
\widehat W^1(\cdot),
\widehat W^2(\cdot),
\widehat W^{1,1}(\cdot),\widehat W^{2,1}(\cdot),
\widehat{v}^{1,\prime}(\cdot),
\widehat{v}^{2,\prime}(\cdot),
\widehat{X}^1(\cdot),\widehat{X}^2(\cdot)\big)$ on the product space 
$$\bigl\{
{\mathcal C}\big([0,T];\RR^m\big)\bigr\}^2 \times \bigl\{ {\mathcal C}\big({\mathcal S}_{2}^T;\RR^m \otimes \RR^m\big)
\bigr\}^2
 \times \bigl\{ {\mathcal C}\big({\mathcal S}_{2}^T;\RR\big)
 \bigr\}^2
 \times \bigl\{{\mathcal C}\big([0,T];\RR^d\big) \bigr\}^2.
$$ 
{Pay attention that, at this stage, we do not whether $\widehat{X}$ solves the mean field rough equation.}
\smallskip

\textbf{\textbf{3d.}}
{We know from the previous step that the limiting set-up satisfies (at least outside an exceptional event) the required algebraic conditions.} We now check that $\widehat{\boldsymbol W}(\cdot)$ satisfies the required regularity properties.

We start with the variations of $\widehat{W}(\widehat{\omega})$, $\langle \widehat{W}(\cdot) \rangle_{q'}$, $\widehat{\mathbb W}(\widehat{\omega})$, 
$\langle \widehat{\mathbb W}^{\indep}(\widehat{\omega},\cdot) \rangle_{q'}$, 
$\langle \widehat{\mathbb W}^{\indep}(\cdot,\widehat{\omega}) \rangle_{q'}$
and 
$\llangle \widehat{\mathbb W}^{\indep}(\cdot,\cdot) \rrangle_{q'}$. To do so,  we recall that, 
for a.e. $\widehat{\omega} \in \widehat{\Omega}$, 
$\hat{v}^{\prime}(\widehat{\omega})$
is the limit of 
$\hat{v}^{n,\prime}(\widehat{\omega})$. 
By passage to the limit, $\hat{v}'$ inherits the super-additive 
property of the $(v^{n,\prime})_{n \geq 0}$'s, {see
step
  \textbf{\textbf{1d}}}, and its tails satisfy
(uniformly in $n \geq 0$) a bound similar to that 
satisfied by 
the $(v^n)_{n \geq 0}$'s 
 in the first item of the assumption.
Also, ${\mathcal S}_{2}^T \ni (s,t) \mapsto 
\langle v'(s,t,\cdot)\rangle_{q'}$ is Lipschitz.

Passing once more to the limit, we get that, for a.e. $\hat{\omega} \in \widehat{\Omega}$, for any $(s,t) \in {\mathcal S}_{2}^T$, $\vert \widehat{W}_{s,t}(\widehat{\omega}) \vert^{p'} \leq v'(s,t,\omega)$, from which we deduce that 
the $p'$-variation of $\widehat{W}(\widehat \omega)$ is dominated (in an obvious sense) by 
$\widehat{v}^{\prime}$. A similar augment applies for $\langle \widehat{W}(\cdot) \rangle_{q'}$, $\widehat{\mathbb W}(\widehat{\omega})$ and 
$\llangle \widehat{\mathbb W}^{\indep}(\cdot,\cdot) \rrangle_{q'}$.

It thus remains to handle $\bigl\langle \widehat{\mathbb W}^{\indep}(\widehat{\omega},\cdot) \bigr\rangle_{q'}$
and
$\bigl\langle \widehat{\mathbb W}^{\indep}(\cdot,\widehat{\omega}) \bigr\rangle_{q'}$. 
%By Fatou's lemma,
%\begin{equation}
%\label{eq:bound:ll:rr}
%\Bigl\llangle \sup_{(s,t) \in {\mathcal S}_{2}^T} \bigl\vert \widehat{\mathbb W}_{s,t}^{\indep}(\cdot,\cdot) \bigr|
%\Bigr\rrangle_{q'}
% < \infty.
%\end{equation}
%Hence, arguing as in the presentation of a rough set-up, 
%see Section 
%\ref{SectionRoughStructure}, we can consider 
%\begin{equation*}
%\begin{split}
%&\widehat{\Omega} \ni \widehat{\omega} \mapsto 
%\Bigl( 
%\widehat{W}^{\indep}(\widehat{\omega},\cdot)
%{\mathbf 1}_{\{\langle 
%\sup_{t \in [0,T]}
%\vert \hat{\mathbb W}^{\indep}(\hat \omega,\cdot) \vert \rangle_{q'}
%<
%\infty
%\}},
%\widehat{W}^{\indep}(\cdot,\widehat{\omega})
%{\mathbf 1}_{\{\langle 
%\sup_{t \in [0,T]}
%\vert \hat{\mathbb W}^{\indep}(\cdot,\hat \omega ) \vert \rangle_{q'}
%<
%\infty
%\}} \Bigr),
%\end{split}
%\end{equation*}
%as random variables with values in 
%$
%{\mathcal C}\bigl({\mathcal S}_{2}^T;\RR^m \otimes \LL^q(\widehat{\Omega},\widehat{\mathcal F},\widehat{\PP};\RR^m)\bigr) \times {\mathcal C}\bigl({\mathcal S}_{2}^T;\LL^q(\widehat{\Omega},\widehat{\mathcal F},\widehat{\PP};\RR^m) \otimes \RR^m\bigr).$ 
%Continuity of the preceding two paths follows from the fact that $\widehat{W}^{\indep}$ has continuous paths and from the bound
%\eqref{eq:bound:ll:rr},
%which makes licit the application of Lebesgue's dominated convergence theorem to prove continuity.
In order to control their variations, 
we proceed as follows. For any non-negative valued bounded continuous function
$g$
on ${\mathcal C}([0,T];\RR^m) \times 
{\mathcal C}({\mathcal S}_{2}^T;\RR)$
and for every $(s,t) \in {\mathcal S}_{2}^T$, we have
\begin{equation*}
\begin{split}
&\int_{\widehat{\Omega}}
\Bigl[ g \bigl( \widehat{W}(\widehat{\omega}),
\widehat{v}^{\prime}(\widehat{\omega}) 
\bigr) \bigl\langle 
\widehat{ {\mathbb W}}_{s,t}^{\indep}(\widehat{\omega},\cdot) \bigr\rangle_{q'}^{q'}
\Bigr] d \widehat{\mathbb P}(\widehat{\omega})
\\
&=
\int_{\widehat{\Omega}^2}
\Bigl[ g \bigl( \widehat{W}(\widehat{\omega}'),
\widehat{v}^{\prime}(\widehat{\omega}') 
\bigr) \bigl\vert 
\widehat{  {\mathbb W}}_{s,t}^{\indep}(\widehat{\omega}',\widehat{\omega}) 
\bigr\vert^{q'}
\Bigr] d \widehat{\mathbb P}^{\otimes 2}(\widehat{\omega},\widehat{\omega}')
\\
&= \lim_{n \rightarrow \infty}
\int_{{\Omega}_{n}^2}
\Bigl[ g \bigl( {W}^n({\omega}_{n}'),
{v}^{n,\prime}({\omega}_{n}') 
\bigr) \bigl\vert 
{  {\mathbb W}}_{s,t}^{n,\indep}({\omega}_{n},{\omega}_{n}') 
\bigr\vert^{q'}
\Bigr] d {\mathbb P}_{n}^{\otimes 2}({\omega}_{n}',{\omega}_{n}),
\end{split}
\end{equation*}
where we used Fubini's theorem
to pass from the first to the second term
together with 
\eqref{eq:cv:n} to pass from the first to the second line. Now, we use the very definition of 
$v^{n,\prime}$ and 
the second item in the assumption
to deduce that 
\begin{equation*}
\begin{split}
&\int_{\widehat{\Omega}}
\Bigl[ g \bigl( \widehat{{W}}(\widehat{\omega}),
\widehat{v}^{\prime}(\widehat{\omega}) 
\bigr) \langle 
\widehat{ {\mathbb W}}_{s,t}^{\indep}(\widehat{\omega},\cdot) \bigr\rangle_{q'}^{q'}
\Bigr] d \widehat{\mathbb P}(\widehat{\omega})
\\
&\leq \lim_{n \rightarrow \infty}
\int_{{\Omega}_{n}}
\Bigl[ g \bigl( {W}^n({\omega}_{n}),
{v}^{n,\prime}({\omega}_{n}) 
\bigr) 
\bigl(
{v}^{n,\prime}(s,t,{\omega}_{n}) 
\bigr)^{q'/p'}
\Bigr] d {\mathbb P}_{n}({\omega}_{n})
\\
&= \int_{\widehat{\Omega}}
\Bigl[ g \bigl( \widehat{W}(\widehat{\omega}),
\widehat{v}^{\prime}(\widehat{\omega}) 
\bigr) 
\bigl(
\widehat{v}^{\prime}(s,t,\widehat{\omega}) 
\bigr)^{q'/p'}
\Bigr] d \widehat{\mathbb P}(\widehat{\omega}).
\end{split}
\end{equation*}
Recalling from
\eqref{eq:cv:strong}
that
$\widehat{\Omega}\ni \widehat{\omega}
\mapsto \langle 
\widehat{ {\mathbb W}}_{s,t}^{\indep}(\widehat{\omega},\cdot) \bigr\rangle_{q'}$
is $\sigma\{\widehat{W}(\cdot)\}$-measurable, 
we get, 
for 
any $(s,t) \in {\mathcal S}_{2}^T$
and for
a.e. $\widehat{\omega} \in
\widehat{\Omega}$, 
$\langle 
\widehat{ {\mathbb W}}_{s,t}^{\indep}(\widehat{\omega},\cdot) \bigr\rangle_{q'}^{p'}
\leq \widehat v^{\prime}(s,t,\widehat{\omega})$. 
{Obviously, the latter is true for a.e. $\hat{\omega}$, for any $(s,t) \in {\mathcal S}_{2}^T \cap {\mathbb Q}^2$.
By almost sure (in $(\widehat \omega,\widehat \omega')$) continuity of the paths ${\mathcal S}_{2}^T \ni (s,t)  \mapsto \widehat{\mathbb W}_{s,t}^{\indep}(\widehat{\omega},\widehat{\omega}')$ and by Fatou's lemma, 
we deduce that it holds true for a.e. $\widehat \omega$, for any $(s,t) \in {\mathcal S}_{2}^T$.} 
The same holds for 
$\langle 
\widehat{ {\mathbb W}}_{s,t}^{\indep}(\cdot,\widehat{\omega}) \bigr\rangle_{q'}$.

Associating 
with the rough set-up $\widehat{\boldsymbol W}$  a (random) control function $\widebar{v}'$ through the definition \eqref{eq:v} with $(p,q)$ replaced by $(p',q')$, we deduce that, for $\widehat{\PP}$-a.e. $\widehat{\omega} \in \widehat{\Omega}$,
for all $(s,t) \in {\mathcal S}_{2}^T$, 
$\widebar{v}'(s,t,\widehat{\omega})$ is less than 
$\widehat{v}^{\prime}(s,t,\widehat{\omega})$.

Modifying the definition of the set-up on the possibly non-empty null event where 
one of the aforementioned properties fails (see the proof of Proposition \ref{prop:example:construction} for details), we can assume without 
any loss of generality that, for any $\widehat{\omega} \in \widehat{\Omega}$, 
the variation of $\widehat{\boldsymbol W}(\widehat{\omega})$ is 
dominated by
 $\widehat{v}^{\prime}(\widehat{\omega})$ 
 and that the latter 
 is finite for all $\widehat{\omega} \in \widehat{\Omega}$. Also, 
 we can assume that Chen's relationship, see
 \eqref{eq:chen}, is satisfied for every $\widehat{\omega} \in \widehat{\Omega}$.
\smallskip

\textbf{\textbf{3e.}}
We let $\widehat{w}^{\prime}(s,t,\widehat{\omega}) :=
\widehat{v}^{\prime}(s,t,\widehat{\omega}) 
 + C (t-s)$, where $C$ is the Lipschitz constant in 
 the second item of the assumption. 
 Clearly, 
$\widehat{w}^{\prime}$ satisfies the first tail estimate in \eqref{EqTailAssumptions}. Moreover, if we associate with $\widehat{w}^{\prime}$ the (random) local accumulation $\widehat{N}'(\cdot,\alpha):=N_{(\widehat{w}^{\prime})^{1/p}}([0,T],\alpha)$ as in \eqref{eq:N:s:t:omega}, then {we must have{\footnote{  The proof is as follows.
Call $N'=\widehat{N}'(\cdot,\alpha)$.
 Without any loss of generality, we may assume $N' \geq 2$.
Define $(t_{i}:=\tau_{i}^{\varpi}(0,\alpha))_{i=0,\cdots,N'-1}$
 as in 
 \eqref{eq:stopping:times}, 
 with $\varpi = (\widehat{w}^{\prime})^{1/p}$,  
 and let $t_{N'} := T$. 
 We also let $K:=\lfloor N'/2\rfloor \geq 1$. 
 By super-additivity, we have, for any $k \in \{0,\cdots,K-1\}$, 
 $w(t_{2k},t_{2k+2}) \geq 2 \alpha^p$. 
Recall now that, almost surely, 
$\widehat{w}^{n,\prime}$ converges uniformly to $\widehat{w}^{\prime}$ on ${\mathcal S}_{2}^T$. 
Hence, almost surely, 
for $n$ large enough, we must have 
 $\widehat{w}^{n,\prime}(t_{2k},t_{2k+2}) > \alpha^p$, which implies that 
 $N_{(\widehat{w}^{n,\prime})^{1/p}}([0,T],\alpha) \geq K$.}}} 
$\widehat{N}'([0,T],\alpha) \leq 2 \liminf_{n \rightarrow \infty}
N_{(\widehat{w}^{n,\prime})^{1/p}}([0,T],\alpha)+1$, where 
$\widehat w^{n,\prime}(s,t,\widehat \omega) = \widehat v^{n,\prime}(s,t,\widehat \omega) + C (t-s)$. In particular,
$\widehat{N}'(\cdot,\alpha)$ satisfies the second tail estimate in \eqref{EqTailAssumptions} {(for possible new constants 
$c_{2}(\alpha)$ and $\epsilon_{2}(\alpha)$)}. Obviousy,
the same holds for the counter $\widebar{N}'(\cdot,\alpha)$ associated with $\widebar{v}'(\cdot)$. In the end, $\widehat{\boldsymbol W}(\cdot)$ satisfies all the requirements of {Theorems \ref{main:thm:existence:small:time} and \ref{ThmMain}}.
\medskip

\textbf{\textbf{Step 4.}}

\textbf{\textbf{4a.}} For each $n \geq 0$, we define $\delta_{x} \widehat X^n(\cdot)$ and $R^{\widehat X^n}(\cdot)$ as
\begin{equation*}
\begin{split}
&\delta_{x} \widehat X^n_{t}(\widehat \omega) := \textrm{\rm F}\bigl(\widehat X^n_{t}(\widehat \omega),{\mathcal L}(X^n_{t})\bigr),
 \ \widehat R^{\widehat X^n}_{s,t}(\widehat \omega) := \widehat X^n_{t}(\widehat \omega) - \widehat X^n_{s}(\widehat \omega) - \delta_{x} \widehat X^n_{s}(\widehat \omega) \widehat W^n_{s,t}(\widehat \omega), \end{split}
\end{equation*} 
$(s,t) \in {\mathcal S}_{2}^T$,
$\widehat{\omega}
\in \widehat{\Omega}$,
from which we easily deduce that $\big(\delta_{x} \widehat X^n(\cdot),\widehat{R}^{\widehat X^n}(\cdot)\big)_{n \geq 0}$ converges with probability to $1$ to $\big(\delta_{x} \widehat X(\cdot),\widehat{R}^{\widehat{X}}(\cdot)\big)$ defined as 
\begin{equation*}
\begin{split}
&\delta_{x} \widehat X_{t}(\widehat \omega) := \textrm{\rm F}\bigl(\widehat X_{t}(\widehat \omega),{\mathcal L}(\widehat X_{t})\bigr),
\quad \widehat R^{\widehat X}_{s,t}(\widehat \omega) := \widehat X_{t}(\widehat \omega) - \widehat X_{s}(\widehat \omega) - \delta_{x} \widehat X_{s}(\widehat \omega) \widehat W_{s,t}(\widehat \omega), 
\end{split} 
\end{equation*} 
$(s,t) \in {\mathcal S}_{2}^T$,
$\widehat{\omega}
\in \widehat{\Omega}$.
In order to pass to the limit in the measure argument of $\textrm{\rm F}$, we use the fact 
that, for any $t \in [0,T]$, $({\mathcal L}(X^n_{t}))_{n \geq 0}$ converges in the weak sense to
${\mathcal L}(\widehat X_{t})$. By 
the uniform integrability property 
{\textbf{2b}}, the convergence also holds in $2$-Wasserstein distance $d_{2}$.
By continuity of $\textrm{\rm F}$ with respect to $d_{2}$, we easily conclude.
\smallskip

\textbf{\textbf{4b.}}
By the second step,   $\big(\PP_{n} \circ (\vvvert X^n(\cdot) \vvvert_{[0,T],w^{n,\prime},p'})^{-1}\big)_{n \geq 0}$ is tight in $\RR$, where  we take 
 $w^{n,\prime}(s,t,\omega_{n}) = 
 v^{n,\prime}(s,t,\omega_{n}) + C(t-s)$, 
 for the same $C$ as in
 \textbf{\textbf{3e}}. {Hence, we can add a new coordinate
 to the almost surely converging subsequence 
 \eqref{eq:cv:n}
 inherited from 
 Skorokhod theorem.
 This new coordinate represents 
$(\vvvert X^n(\cdot) \vvvert_{[0,T],w^{n,\prime},p'})_{n \geq 0}$. In fact, 
since $\PP_{n} \circ \bigl(X^n(\cdot), \delta_{x} X^n(\cdot),R^{X^n}(\cdot),v^{n,\prime}(\cdot)\bigr)^{-1}$ coincides with $\widehat{\PP} \circ \bigl(\widehat{X}^n(\cdot),
\delta_{x} \widehat{X}^n(\cdot),\widehat{R}^{\widehat{X}^n}(\cdot),\widehat{v}^{n,\prime}(\cdot)
\bigr)^{-1}$ for each $n \geq 0$, the new coordinate in the Skorokhod subsequence may be chosen as $\big(\vvvert  \widehat X^n(\cdot) \vvvert_{[0,T],\widehat w^{n,\prime},p'}\big)_{n \geq 0}$ itself, where, as before,
$\widehat w^{n,\prime}(s,t,\widehat \omega) = \widehat v^{n,\prime}(s,t,\widehat \omega) + C (t-s)$. We thus assume that the latter sequence is almost surely convergent.}
Moreover, identity in law of $\big(W^n(\cdot),X^n(\cdot)\big)$ under $\PP_{n}$ and of $\big(\widehat{W}^n(\cdot),\widehat{X}^n(\cdot)\big)$ under $\widehat{\PP}$ {also says that}, for $\widehat{\PP}$-a.e. $\widehat{\omega} \in \widehat{\Omega}$ and any $(s,t) \in {\mathcal S}_{2}^T$, 
$\vert \widehat X^n_{s,t}(\widehat \omega) \vert \leq \bigl\vvvert \widehat X^n(\widehat \omega) \bigr\vvvert_{[0,T],\widehat w^{n,\prime},p'} \, \bigl( \widehat w^{n,\prime}(s,t,\widehat \omega) \bigr)^{1/p'}$. 
By  {\eqref{eq:cv:n}} and \textbf{\textbf{3c}}, we get, for $\widehat{\mathbb P}$-a.e. $\widehat{\omega} \in \widehat{\Omega}$, for all $(s,t) \in {\mathcal S}_{2}^T$,
$$\vert \widehat X_{s,t}(\widehat \omega) \vert \leq
\bigl( \lim_{n \rightarrow \infty} \,
 \bigl\vvvert \widehat X^n(\widehat \omega) \bigr\vvvert_{[0,T],\widehat w^{n,\prime},p'}
 \bigr) \bigl( \widehat w'(s,t,\widehat \omega) \bigr)^{1/p'}.$$ 
Proceeding similarly for $\delta_{x}\widehat X^n(\cdot)$ and $R^{\widehat X^n}(\cdot)$, we deduce that, for $\widehat \PP$-a.e. $\widehat \omega \in \widehat \Omega$, 
\begin{equation*}
\vvvert \widehat X(\widehat \omega) \vvvert_{[0,T],\widehat w',p'} \leq \lim_{n \rightarrow \infty}
\vvvert \widehat X^n(\widehat \omega) \vvvert_{[0,T],\widehat w^{n,\prime},p'},
\end{equation*}
which shows in particular by Fatou's lemma, see step
\textbf{\textbf{2b}}, that
$\bigl\langle \vvvert \widehat X(\cdot) \vvvert_{[0,T],\widehat w',p'}
\bigr\rangle_{8} < \infty$. Although $\widehat{v}^{\prime}(\widehat{\omega})$
(and thus
$\widehat{w}^{\prime}(\widehat{\omega})$) 
 is not associated with $\widehat{\boldsymbol W}(\widehat{\omega})$
 through \eqref{eq:v}, we shall say that, for a.e. $\widehat{\omega} \in \widehat{\Omega}$, 
$\widehat{X}(\widehat{\omega})$ is an  
$\widehat \omega$-controlled trajectory for the rough set-up {$\widehat{\boldsymbol W}(\cdot)$}. 
{(We come back to this point right below.)}
\medskip

\textbf{\textbf{Step 5.}}   

\textbf{\textbf{5a.}} So far, we have constructed $\bigl(\widehat X(\widehat \omega);\textrm{\rm F}\bigl(\widehat X(\widehat \omega),\widehat X(\cdot)\bigr);0\bigr)$ as an $\widehat \omega$-controlled trajectory for the limit rough set-up 
{$\widehat{\boldsymbol W}(\cdot)$}, but for $\widehat \omega$ in a full event $\widehat \Omega' \subset \widehat\Omega$. For free, we can modify the definition of $\widehat X(\widehat \omega)$ for $\widehat \omega \in \widehat \Omega \setminus \widehat \Omega'$ and define $\delta_{x} \widehat X(\widehat \omega)$ accordingly so that $\bigl(\widehat X(\widehat \omega);\delta_{x} \widehat X(\widehat \omega);0\bigr)$ is an $\widehat \omega$-controlled trajectory for any $\widehat \omega$. Then,  $\big(\widehat X(\widehat \omega)\big)_{\widehat \omega \in \widehat \Omega}$ forms a random controlled trajectory. 
\smallskip

\textbf{\textbf{5b.}} In order to conclude, it remains to identify $\bigl(\widehat X(\widehat \omega);\textrm{\rm F}\bigl(\widehat X(\widehat \omega),\widehat X(\cdot)\bigr);0\bigr)$, for $\widehat \PP$-a.e. $\widehat \omega \in \widehat \Omega$, with 
$
\Gamma_{\widehat{\boldsymbol W}}\Bigl(\widehat X(\widehat \omega);\textrm{\rm F}\bigl(\widehat X(\widehat \omega),\widehat X(\cdot)\bigr);0\Bigr),
$ 
where the index $\widehat{\boldsymbol W}$ in $\Gamma_{\widehat{\boldsymbol W}}$ is to emphasize the rough set-up upon which the map $\Gamma$ in Definition \ref{DefnGamma} is constructed. To do so, we recall from \eqref{eq:remainder:integral} the 
expansion {(see also \eqref{exp:measurability})}
\begin{flalign}
&X_{t_{i}}^{n}(\omega_{n}) \nonumber
= X_{0}^n(\omega_{n}) + \sum_{j=1}^i \textrm{F}\bigl(X_{t_{j-1}}^n(\omega_{n}), {\mathcal L}( X_{t_{j-1}}^n)\bigr) W^n_{t_{j-1},t_{j}}
(\omega_{n}) \nonumber
\\
&\hspace{15pt} + \sum_{j=1}^i 
\partial_{x} \textrm{F}\bigl(X_{t_{j-1}}^n(\omega_{n}),{\mathcal L}( X_{t_{j-1}}^n)\bigr)
\Bigl( 
\textrm{F}\bigl(X_{t_{j-1}}^n(\omega_{n}),{\mathcal L}( X_{t_{j-1}}^n)\bigr)
 \WW^n_{t_{j-1},t_{j}}(\omega_{n})\Bigr)
 \label{eq:Xn:3/p}
\\
&\hspace{15pt} + \sum_{j=1}^i \Bigl\langle D_{\mu} \textrm{F}\bigl(X_{t_{j-1}}^n(\omega_{n}),{\mathcal L}( X_{t_{j-1}}^n)\bigr)\bigl(
X_{t_{j-1}}^n(\cdot)\bigr) 
\Bigl(
\textrm{F}\bigl(X_{t_{j-1}}^n(\cdot),{\mathcal L}( X_{t_{j-1}}^n)\bigr)
\WW^{n,\indep}_{t_{j-1},t_{j}}(\cdot,\omega_{n}) \Bigr) \Bigr\rangle
\nonumber
\\
&\hspace{15pt} +\sum_{j=1}^i S_{t_{j-1},t_{j}}^{n}(\omega_{n}), 
\nonumber
\end{flalign}
that holds true for any $\omega_{n} \in \Omega_{n}$, any $n \geq 0$ and any subdivision
$0=t_{0} < t_{1} < \cdots <t_{K}=T$, with $K \geq 1$, and with
(see 
Theorem 
\ref{thm:integral},
Proposition 
\ref{prop:chaining}
and 
\textbf{\textbf{2b}})
\begin{equation*}
\bigl\vert
S_{t_{j-1},t_{j}}^{n}(\omega_{n})
\bigr\vert 
\leq 
C 
\Bigl( 1 + 
\vvvert X^n(\omega_{n}) \vvvert_{\textcolor{black}{[0,T],w^{n,\prime},p'}}^2 
\Bigr)
\, w^{n,\prime}(t_{j-1},t_{j},\omega_{n})^{3/p'}.
\end{equation*}
In order to pass to the limit in 
\eqref{eq:Xn:3/p}, we consider a 
non-negative valued bounded continuous function
$g$
on ${\mathcal C}([0,T];\RR^m) \times
 {\mathcal C}({\mathcal S}_{2}^T;\RR^m \otimes \RR^m)
 \times 
 {\mathcal C}({\mathcal S}_{2}^T;\RR)
\times 
 {\mathcal C}([0,T];\RR^d)$. We then multiply both sides of 
\eqref{eq:Xn:3/p} by $g\bigl(W^n(\omega_{n}),{\mathbb W}^n(\omega_{n}),v^{n,\prime}(\omega_{n}),X^n(\omega_{n})\bigr)$
and integrate $\omega_{n}$ with respect to $\PP_{n}$. 
It is absolutely obvious that
\begin{equation*}
\begin{split}
&\lim_{n \rightarrow \infty}
\EE_{n}
\Bigl[ g\bigl(W^n(\cdot),{\mathbb W}^n(\cdot),v^{n,\prime}(\cdot),X^n(\cdot)\bigr)
X^n_{t_{i}}(\cdot) 
\Bigr] 
= \widehat{\mathbb E}
\Bigl[ g \bigl( \widehat W(\cdot), \widehat{\mathbb W}(\cdot),\widehat v^{\prime}(\cdot), \widehat X(\cdot) \bigr) 
\widehat{X}_{t_{i}}(\cdot)
\Bigr],
\end{split}
\end{equation*}
and similarly with $t_{i}$ replaced by $0$. In the same way,
\begin{equation*}
\begin{split}
&\lim_{n \rightarrow \infty}
\EE_{n}
\Bigl[ g\bigl(W^n(\cdot),{\mathbb W}^n(\cdot),v^{n,\prime}(\cdot),X^n(\cdot)\bigr)
\textrm{F}\bigl(X_{t_{j-1}}^n(\cdot), {\mathcal L}( X_{t_{j-1}}^n)\bigr) W^n_{t_{j-1},t_{j}}(\cdot)
\Bigr] 
\\
&= \widehat{\mathbb E}
\Bigl[ g \bigl( \widehat W(\cdot), \widehat{\mathbb W}(\cdot),\widehat{v}^{\prime}(\cdot), \widehat X(\cdot) \bigr) 
\textrm{F}\bigl(\widehat X_{t_{j-1}}(\cdot), {\mathcal L}(\widehat X_{t_{j-1}})\bigr) \widehat W_{t_{j-1},t_{j}}(\cdot)
\Bigr],
\end{split}
\end{equation*}
and similarly for the terms on the second line. As for the fifth term in the right-hand side, we have
\begin{equation*}
\begin{split}
&\limsup_{n \rightarrow \infty}
\EE_{n}
\Bigl[ g\bigl(W^n(\cdot),{\mathbb W}^n(\cdot),v^{n,\prime}(\cdot),X^n(\cdot)\bigr)
\,  {\bigl\vert 
S^n_{t_{j-1},t_{j}}(\cdot)
\bigr\vert}
\Bigr] 
\\
&\leq C \limsup_{n \rightarrow \infty}
\EE_{n}
\Bigl[ g\bigl(W^n(\cdot),{\mathbb W}^n(\cdot),v^{n,\prime}(\cdot),X^n(\cdot) \bigr)
\\
&\hspace{150pt} \times \Bigl( 1 + 
\vvvert X^n(\cdot) \vvvert_{\textcolor{black}{[0,T],w^{n,\prime},p'}}^2 
\Bigr)
 w^{n,\prime}(t_{j-1},t_{j},\cdot)^{3/p'}
\Bigr].
\end{split}
\end{equation*}
Transferring the right-hand side into an expectation on 
$(\widehat{\Omega},\widehat{\mathcal F},\widehat{\mathbb P})$
and using obvious uniform integrability properties, 
see \textbf{\textbf{2b}}, we deduce from 
\textbf{\textbf{4b}} that
\begin{equation*}
\begin{split}
&\limsup_{n \rightarrow \infty}
\EE_{n}
\Bigl[ g\bigl(W^n(\cdot),{\mathbb W}^n(\cdot),v^{n,\prime}(\cdot),X^n(\cdot)\bigr)
\vert 
S^n_{t_{j-1},t_{j}}(\cdot)
\vert
\Bigr] 
\\
&\leq C 
\widehat \EE
\Bigl[ g\bigl(\widehat W(\cdot),\widehat{\mathbb W}(\cdot),\widehat v'(\cdot),\widehat X(\cdot)\bigr)
\Bigl( 1 + 
\lim_{n \rightarrow \infty}\vvvert \widehat X^n(\cdot) \vvvert_{\textcolor{black}{[0,T],\widehat w^{n,\prime},p'}}^2 
\Bigr)
\, \widehat w'(t_{j-1},t_{j},\cdot)^{3/p'}
\Bigr].
\end{split}
\end{equation*} 
Of course, the most difficult term to treat in 
\eqref{eq:Xn:3/p} is the fourth one in the right-hand side. This can be done by using Fubini's theorem:
\begin{equation*}
\begin{split}
&\int_{\Omega_{n}} d\PP_{n}(\omega_{n}) \Bigl[ 
g\bigl(W^n(\omega_{n}),{\mathbb W}^n(\omega_{n}),v^{n,\prime}(\omega_{n}), X^n(\omega_{n})\bigr)
\\
&\hspace{3pt} \cdot \Bigl\langle D_{\mu} \textrm{F}\bigl(X_{t_{j-1}}^n(\omega_{n}),{\mathcal L}(X_{t_{j-1}}^n)\bigr)\bigl(
X_{t_{j-1}}^n(\cdot)\bigr) 
\Bigl(
\textrm{F}\bigl(X_{t_{j-1}}^n(\cdot),{\mathcal L}(X_{t_{j-1}}^n)\bigr)
\WW^{n,\indep}_{t_{j-1},t_{j}}(\cdot,\omega_{n}) \Bigr) \Bigr\rangle
\Bigr] 
\\
&=\int_{\Omega_{n}^2} d\PP_{n}^{\otimes 2}(\omega_{n},\omega_{n}') \Bigl[ 
g\bigl(W^n(\omega_{n}),{\mathbb W}^n(\omega_{n}),v^{n,\prime}(\omega_{n}),X^n(\omega_{n})\bigr)
\\
&\hspace{3pt} \cdot  D_{\mu} \textrm{F}\bigl(X_{t_{j-1}}^n(\omega_{n}),{\mathcal L}(X_{t_{j-1}}^n)\bigr)\bigl(
X_{t_{j-1}}^n(\omega_{n}')\bigr)   \Bigl(
\textrm{F}\bigl(X_{t_{j-1}}^n(\omega_{n}'),{\mathcal L}(X_{t_{j-1}}^n)\bigr)
\WW^{n,\indep}_{t_{j-1},t_{j}}(\omega_{n}',\omega_{n}) \Bigr) 
\Bigr]
\\
&= \widehat{\EE}\Bigl[ 
g\bigl(\widehat W^{n,1}(\cdot),{\widehat {W}^{n,1,1}}(\cdot),\widehat v^{1,n,\prime}(\cdot),\widehat X^{n,1}(\cdot)\bigr)
\\
&\hspace{3pt} \cdot  D_{\mu} \textrm{F}\bigl(\widehat X_{t_{j-1}}^{n,1}(\cdot),{\mathcal L}(X_{t_{j-1}}^n)\bigr)\bigl(
\widehat X_{t_{j-1}}^{n,2}(\cdot)\bigr)   \Bigl(
\textrm{F}\bigl(\widehat X_{t_{j-1}}^{n,2}(\cdot),{\mathcal L}(X_{t_{j-1}}^n)\bigr)
{\widehat  W^{n,2,1}_{t_{j-1},t_{j}}}(\cdot) \Bigr) 
\Bigr].
\end{split}
\end{equation*}
We now use 
\eqref{eq:cv:n} in order to pass to the limit. The only slight difficulty is that we must ensure that the 
regularity conditions satisfied by $D_{\mu} \textrm{\rm F}$ are compatible with the 
almost sure convergence property \eqref{eq:cv:n}. Recall indeed 
that the continuity property \textbf{\textbf{Regularity assumptions  1}} 
is formulated in $\LL_{2}$. By \cite[Proposition 5.36]{CarmonaDelarue_book_I}, this implies that the mapping $v \mapsto D_{\mu} \textrm{\rm F}(x,\mu)(v)$
is Lipschitz continuous, uniformly in $x$ and $\mu$. The latter guarantees that, for a.e. 
$\widehat \omega \in \widehat \Omega$, 
\begin{equation*}
\lim_{n \rightarrow \infty}
D_{\mu} \textrm{F}\bigl(\widehat X_{t_{j-1}}^{n,1}(\widehat \omega),{\mathcal L}(X_{t_{j-1}}^n)\bigr)\bigl(
\widehat X_{t_{j-1}}^{n,2}(\widehat \omega)\bigr)
=
D_{\mu}\textrm{F}\bigl(\widehat X_{t_{j-1}}^{1}(\widehat \omega),{\mathcal L}(\widehat X_{t_{j-1}})\bigr)\bigl(
\widehat X_{t_{j-1}}^{2}(\widehat \omega)\bigr).
\end{equation*}
So, the limit of the summand on the 
fourth line of 
\eqref{eq:Xn:3/p} is
\begin{equation*}
\begin{split}
&\widehat{\EE}\Bigl[ 
g\bigl(\widehat W^{1}(\cdot),\widehat W^{1,1}(\cdot),\widehat v^{1,\prime}(\cdot),\widehat X^{1}(\cdot)\bigr)
\\
&\hspace{3pt} \cdot  D_{\mu} \textrm{F}\bigl(\widehat X_{t_{j-1}}^{1}(\cdot),{\mathcal L}(\widehat X_{t_{j-1}}^1)\bigr)\bigl(
\widehat X_{t_{j-1}}^{2}(\cdot)\bigr)   \Bigl(
\textrm{F}\bigl(\widehat X_{t_{j-1}}^{2}(\cdot),{\mathcal L}(\widehat X_{t_{j-1}}^1)\bigr)
\widehat \WW^{2,1}_{t_{j-1},t_{j}}(\cdot) \Bigr) 
\Bigr],
\end{split}
\end{equation*}
and our reconstruction of the limiting set-up permits to rewrite it in the form
\begin{equation*}
\begin{split}
&\int_{\widehat \Omega} d\widehat \PP (\widehat \omega) \Bigl[ 
g\bigl(\widehat W(\widehat \omega),\widehat{\mathbb W}(\widehat \omega),\widehat{v}^{\prime}(\widehat \omega),\widehat X(\widehat \omega)\bigr)
\\
&\hspace{3pt} \cdot \Bigl\langle D_{\mu} \textrm{F}\bigl(\widehat X_{t_{j-1}}(\widehat \omega),
{\mathcal L}(\widehat X_{t_{j-1}})\bigr)\bigl(
\widehat X_{t_{j-1}}(\cdot)\bigr) 
\Bigl(
\textrm{F}\bigl(\widehat X_{t_{j-1}}(\cdot),{\mathcal L}(\widehat X_{t_{j-1}})\bigr)
\widehat \WW^{\indep}_{t_{j-1},t_{j}}(\cdot,\widehat \omega) \Bigr) \Bigr\rangle
\Bigr].
\end{split}
\end{equation*} 
Importantly, since the limiting set-up is strong, the term in bracket in the last line is 
$\sigma\{ \widehat W,\widehat X\}$-measurable. 
\smallskip

\textbf{\textbf{5c.}}
Let now 
\begin{equation*}
\begin{split}
&{\mathcal J}(\widehat{\omega}) 
:= 
\widehat X_{t_{i}}(\widehat \omega) 
- \widehat X_{0}(\widehat \omega) 
- \sum_{j=1}^i \textrm{F}\bigl(\widehat X_{t_{j-1}}(\widehat \omega), {\mathcal L}(\widehat X_{t_{j-1}})\bigr) 
\widehat W_{t_{j-1},t_{j}}
(\widehat \omega) \nonumber
\\
&\hspace{5pt} - \sum_{j=1}^i 
\partial_{x} \textrm{F}\bigl(\widehat X_{t_{j-1}}(\widehat \omega),{\mathcal L}(\widehat X_{t_{j-1}})\bigr)
\Bigl( 
\textrm{F}\bigl(\widehat X_{t_{j-1}}(\widehat \omega),{\mathcal L}( \widehat X_{t_{j-1}})\bigr)
 \widehat \WW_{t_{j-1},t_{j}}(\widehat \omega)\Bigr)
\\
&\hspace{5pt} - \sum_{j=1}^i \Bigl\langle D_{\mu} \textrm{F}\bigl(\widehat X_{t_{j-1}}(\widehat \omega),
{\mathcal L}(\widehat  X_{t_{j-1}})\bigr)\bigl(
\widehat X_{t_{j-1}}(\cdot)\bigr) 
\Bigl(
\textrm{F}\bigl(\widehat X_{t_{j-1}}(\cdot),{\mathcal L}(\widehat  X_{t_{j-1}})\bigr)
\widehat \WW^{\indep}_{t_{j-1},t_{j}}(\cdot,\widehat \omega) \Bigr) \Bigr\rangle.
\end{split}
\end{equation*}
By the conclusion of \textbf{\textbf{5b}}, it is $\sigma\{ \widehat W,\widehat{\mathbb W},\widehat X\}$-measurable and it satisfies, for 
any 
$g$
 as in the previous step,
 \begin{equation*}
 \begin{split}
& \widehat{\mathbb E} \Bigl[  g\bigl(\widehat W(\cdot),\widehat {\mathbb W}(\cdot),
\widehat{v}^{\prime}(\cdot),
\widehat X(\cdot)\bigr) \, {\bigl\vert 
 \widehat{\mathcal J}(\cdot)
 \bigr\vert }
 \Bigr] 
 \\
&\hspace{5pt} \leq 
 \widehat \EE
\Bigl[ g\bigl(\widehat W(\cdot),\widehat{\mathbb W}(\cdot),\widehat{v}^{\prime}(\cdot),\widehat X(\cdot)\bigr)
\Bigl( 1 + 
\lim_{n \rightarrow \infty}\vvvert \widehat X^n(\cdot) \vvvert_{\textcolor{black}{[0,T],\widehat w^{n,\prime},p'}}^2 
\Bigr)
\, \sum_{j=1}^i \widehat w'(t_{j-1},t_{j},\cdot)^{3/p'}
\Bigr].
\end{split}
 \end{equation*}
 Therefore, for $\widehat{\PP}$-a.e. $\widehat{\omega}$,
 \begin{equation*}
 \bigl\vert {\mathcal J}(\widehat \omega) \bigr\vert \leq C
 \Bigl( \sum_{j=1}^i
  \widehat w'(t_{j-1},t_{j},\widehat{\omega})^{3/p'}
  \Bigr)
 \widehat{\EE} \Bigl[ 
\lim_{n \rightarrow \infty}\vvvert \widehat X^n(\cdot) \vvvert_{\textcolor{black}{[0,T],\widehat w^{n,\prime},p'}}^2 
\, 
\vert \, \sigma\bigl\{\widehat W,\widehat {\mathbb W},\widehat{v}^{\prime},\widehat X\bigr\}
\Bigr](\widehat{\omega}). 
 \end{equation*}
By super-additivity   of $\widehat{w}^{\prime}$, $\widehat X_{t}(\widehat \omega)$ and 
$\widehat{X}_{0}(\widehat \omega) + \int_{0}^t \textrm{\rm F}(\widehat X_{s}(\omega),\widehat X_{s}(\cdot)) d \widehat{\boldsymbol W}_{s}(\omega)$ coincide. Note that this is true 
although the functionals 
$\widehat{v}^{\prime}(\widehat{\omega})$
and
$\widehat{w}^{\prime}(\widehat{\omega})$
that control the 
variations of $\widehat X$ are not associated with $\widehat{\boldsymbol W}(\widehat{\omega})$
 through \eqref{eq:v}; 
 the sole fact that $\widehat{v}^{\prime}(\widehat{\omega})$ dominates 
 $\widebar{v}^{\prime}(\widehat{\omega})$ (which is associated with 
 $\widehat{\boldsymbol W}(\widehat{\omega})$
 through \eqref{eq:v})
 and {that 
 $\widehat{w}^{\prime}(\widehat{\omega})$
 satisfy 
 \eqref{eq:w:s:t:omega:ineq}
 and 
\eqref{eq:useful:inequality:wT}}
 suffices.

The domination of 
$\widebar{v}^{\prime}(\widehat{\omega})$
 by $\widehat{v}^{\prime}(\widehat{\omega})$, the latter satisfying the tail properties in Theorem 
 \ref{main:thm:existence:small:time},
 suffices to duplicate the uniqueness argument. 
 In words, $\widehat X(\cdot)$ is the solution to the  mean field rough equation 
 driven by $\widehat{\boldsymbol W}$ and, by uniqueness in law, $\widehat X(\cdot)$ has the same law as $X(\cdot)$. 
\end{proof}

%\begin{Remark}
%The formulation of the statement of \emph{Theorem \ref{ThmContinuity}} should be called {\sl strong}, as the convergence of the rough set-ups $({\boldsymbol W}^n(\cdot))_{n \geq 1}$ is assumed to hold in a strong sense, namely in probability. Still, we may easily formulate a \textit{weak} version of it by assuming that only the laws of the set-ups do converge; indeed, in such a case, we can easily exhibit, by means of Skorohod's representation theorem, a sequence of set-ups, with the same laws as the original ones, that converges in the strong sense. This allows to apply \emph{Theorem \ref{ThmContinuity}}. We do not provide a specific statement. We feel simpler to refer Subsection \ref{SubsectionPropagChaos} where we make clear this strategy to derive propagation of chaos. 
%\end{Remark}
%
%

We used the following lemma in the proof of Theorem \ref{ThmContinuity}. 
%It looks like the converse of Prokhorov's tightness theorem; the latter cannot be applied directly because the space ${\mathcal C}^{p-\textrm{\rm var}}({\mathcal S}_{2}^T;E)$ is not separable.

\begin{lem}
\label{lem:aux:1}
For a separable Banach space $(E,| \cdot|)$, call ${\mathcal C}_{0}^{p-\textrm{\rm var}}({\mathcal S}_{2}^T;E)$ the space of continuous paths $G$ from ${\mathcal S}_{2}^T$ into $E$ that are null on the diagonal, i.e. $G_{t,t}=0$ for all $t \in [0,T]$, and have a finite $p$-variation, i.e.
\begin{equation*}
\|G \|_{[0,T],p-\textrm{\rm var}}^p = \sup_{0 \leq t_{1} < \cdots < t_{N}=T} 
\sum_{i=0}^{N-1}
\vert G_{t_{i},t_{i+1}} \vert^p < \infty.
\end{equation*} 
For each $n \geq 0$, let 
$Z^{n}=(Z^{n}_{s,t})_{s,t \in {\mathcal S}_{2}^T}$ be a process defined on $(\Omega_{n},{\mathcal F}_{n},{\mathbb P}_{n})$ with trajectories in ${\mathcal C}_{0}^{p-\textrm{\rm var}}\big({\mathcal S}_{2}^T;E\big)$. Assume that the family of distributions $\big(\PP_{n} \circ (Z^n)^{-1}\bigr)_{n \geq 0}$
   is tight in ${\mathcal C}({\mathcal S}_{2}^T;E)$, and that 
   the family of distributions  $\big(\PP_{{n}} \circ (\| Z^n \|_{[0,T],p-\textrm{\rm var}})^{-1}\big)_{n \geq 0}$ is tight in $\RR$.

Then, for $p'>p$, the family of distributions $\big(\PP_{{n}} \circ ({\mathcal S}_{2}^T \ni (s,t) \mapsto \| Z^n \|_{[s,t],p'-\textrm{\rm var}} \in \RR)^{-1}\big)_{n \geq 0}$ is tight in ${\mathcal C}({\mathcal S}_{2}^T;\RR)$. In particular, for any $\epsilon >0$, there exists $\delta >0$, such that
\begin{equation*}
\PP_{{n}} \biggl( \sup_{ (s,t) \in {\mathcal S}_{2}^T : t-s \leq \delta} \| Z^n \|_{[s,t],p'-\textrm{\rm var}}  > \varepsilon
\biggr) < \varepsilon.  
\end{equation*}
\end{lem}

\begin{proof}
The first part is an adaptation of Proposition 5.28 and Corollary 5.29 in \cite{FrizVictoirBook}. 
The second part is a  consequence of the fact that 
$\| z \|_{[t,t],p'-\textrm{\rm var}} = 0$, for $z \in {\mathcal C}_{0}^{p-\textrm{\rm var}}({\mathcal S}_{2}^T;E)$.
\end{proof}

%------------------------------------------------------%
\appendix

\section{Proof of Theorem \ref{thm:cass:litterer:lyons}}
\label{SectionIntegrability}
%------------------------------------------------------%

\label{SubsectionAppendixCLL}
%%--------------------------------------------------------------------%%

We provide here the proof of Theorem \ref{thm:cass:litterer:lyons}. We follow the proof of Theorem 11.13 in \cite{FrizHairer}, see also the proof of Proposition 6.2 in \cite{CassLittererLyons}. Throughout the proof, we use the same notations as in the statement of Theorem \ref{thm:cass:litterer:lyons}.

{Notice first that 
handling the local accumulation 
of $w^{1/p}$ is the same as handling the local accumulation 
of $w$. This amounts to change the argument $\alpha$ into $\alpha^p$ in 
\eqref{eq:N:s:t:omega}. Recall now that 
$w(s,t,\omega)$ is given by 
\eqref{eq:w:s:t:omega}
and 
$v(s,t,\omega)$ therein consists in six different terms, see \eqref{eq:v}.} It is an easy exercice to check that it suffices to control the local accumulation associated with each of these six terms. To make it clear, we have the following property. For a given threshold $\alpha >0$ and for any two {nondecreasing} continuous functions $v_{1} : {\mathcal S}_{2}^T \rightarrow \RR_{+}$ and $v_{2}: {\mathcal S}^T_{2} \rightarrow \RR_{+}$, set
$N_{i}(\alpha) := N_{v_{i}}\bigl([0,T],\alpha\bigr)$, for $1\leq i\leq 2$, 
{and
$N(\alpha) := N_{v_{1}+v_{2}}\bigl([0,T],\alpha\bigr)$}; see \eqref{eq:N:s:t:omega} for the original definition. Then 
\begin{equation}
\label{eq:N1:N2}
\max \Bigl( N_{1}\left(\frac{\alpha}{2}\Bigr),N_{2}\Bigl(\frac{\alpha}{2}\Bigr) \right) \geq {N(\alpha)}. 
\end{equation}
For sure, the result is true with the first and third terms in \eqref{eq:v} as this fits the original property established in \cite{CassLittererLyons}. Also, it is obviously true for the second and sixth terms since they are completely deterministic. Hence, the only difficulty is to control the local accumulation associated with the fourth and fifth terms.

The strategy is as follows. As we work with Gaussian rough paths, the set-up, as defined in Section \ref{SectionRoughStructure}, is strong. So, we can transfer it to any arbitrarily fixed probability space (provided that the letter is rich enough). Hence, we can choose $\Omega$ as the path space ${\mathcal W}$, see the notation in the statement of Theorem \ref{thm:cass:litterer:lyons}.

We denote by ${\boldsymbol W}(\omega,\omega')$ the enhanced Gaussian rough path associated to 
$
\big(W(\omega),W'(\omega')\big)$ 
along the lines of Example \ref{example:3}, for $\PP^{\otimes 2}$-a.e. $(\omega,\omega') \in \Omega^2$. The second level of ${\boldsymbol W}(\omega,\omega')$ reads
\begin{equation*}
{\boldsymbol W}^{[2]}(\omega,\omega')
:=
\left(
\begin{array}{cc}
\WW(\omega) & {\mathcal I}\bigl(W(\omega),W'(\omega')\big)
\\
{\mathcal I}\bigl(W'(\omega'),W(\omega)\bigr) &\WW(\omega')
\end{array}
\right),
\end{equation*}
where ${\mathcal I}$ is as in Definition \ref{def:strong}, and where we used the same symbol ${\boldsymbol W}$ as in Section \ref{SectionRoughStructure} for the enhanced path although the meaning here is not exactly the same. Here, ${\boldsymbol W}(\omega,\omega')$ is a function of both $\omega$ and $\omega'$ and takes values in $\RR^{2m} \oplus (\RR^{2m})^{\otimes 2}$. Following Section 3 in \cite{CassLittererLyons}, see also (11.5) in \cite{FrizHairer}, we define, for $h \oplus k \in {\mathcal H} \oplus {\mathcal H}$ the translated rough path $(T_{h \oplus k} {\boldsymbol W})(\omega,\omega')$, 
{where, as in Example 
\ref{example:3}, 
${\mathcal H}$ is the underlying Cameron-Martin space}. We then recall that, with probability 1 under $\PP^{\otimes 2}$, 
\begin{equation*}
T_{h \oplus k} {\boldsymbol W}(\omega,\omega') = {\boldsymbol W}(\omega+h,\omega'+k).
\end{equation*}
Following the argument given in Proposition 6.2 in \cite{CassLittererLyons}, see also Lemma 11.4 in \cite{FrizHairer}, we have, for any $h \in {\mathcal H}$ and any $(s,t) \in {\mathcal S}_{2}^T$, 
\begin{equation*}
\begin{split}
\talloblong {\boldsymbol W}(\omega,\omega') \talloblong_{[s,t],p-\textrm{\rm var}}^p &\leq c \, \Bigl( \talloblong T_{h \oplus 0} {\boldsymbol W}(\omega,\omega')
\talloblong^p_{[s,t],p-\textrm{\rm var}} + \| h \|_{[s,t],\varrho-\textrm{\rm var}}^p \Bigr),
\end{split}
\end{equation*}
where we recall that $1/p+1/\varrho >1$ and $c$ only depends on $p$ and $\varrho$, and where
\begin{equation*}
\talloblong {\boldsymbol W}(\omega,\omega') \talloblong_{[s,t],p-\textrm{\rm var}}
:= \| (W,W')(\omega,\omega') \|_{[s,t],p-\textrm{\rm var}}
+ \sqrt{
\| {\boldsymbol W}^{[2]}(\omega,\omega') \|_{[s,t],(p/2)-\textrm{\rm var}}},
\end{equation*}
and similarly for $\talloblong T_{h \oplus 0}{\boldsymbol W}(\omega,\omega') \talloblong_{[s,t],p-\textrm{\rm var}}$. Taking the power $q$, allowing the constant $c$ to depend on $q$ and integrating with respect to $\omega'$, we get
\begin{equation*}
\Bigl\langle \| {\mathbb W}^{\indep}(\omega,\cdot) \|_{[s,t],(p/2)-\textrm{\rm var}}^{p/2} \Bigr\rangle_{q} \leq c \Bigl( 
\Bigl\langle
\talloblong T_{h \oplus 0} {\boldsymbol W}(\omega,\cdot)
\talloblong^p_{[s,t],p-\textrm{\rm var}}
\Bigr\rangle_{q}
 +  \| h \|_{[s,t],\varrho-\textrm{\rm var}}^p \Bigr).
\end{equation*}
We now let
\begin{equation*}
\begin{split}
&\talloblong {\boldsymbol W}(\omega,\omega') \talloblong_{[s,t],(1/p)-\textrm{\rm H\"ol}}
\\
&\hspace{15pt} := \| (W,W')(\omega,\omega') \|_{[s,t],(1/p)-\textrm{\rm H\"ol}}
 + \sqrt{
\| {\boldsymbol W}^{[2]}(\omega,\omega') \|_{[s,t],(2/p)-\textrm{\rm H\"ol}}},
\end{split}
\end{equation*}
for the standard H\"older semi-norm of the rough path, see Theorem 11.9 in \cite{FrizHairer}. Then,
\begin{equation*}
\begin{split}
&\Bigl\langle \| {\mathbb W}^{\indep}(\omega,\cdot) \|_{[s,t],(p/2)-\textrm{\rm var}} \Bigr\rangle_{q}^{p/2} 
 \leq c \Bigl( 
\Bigl\langle
\talloblong T_{h \oplus 0} {\boldsymbol W}(\omega,\cdot)
\talloblong^p_{[0,T],(1/p)-\textrm{\rm H\"ol}}
\Bigr\rangle_{q}
(t-s)
 +  \| h \|_{[s,t],\varrho-\textrm{\rm var}}^p \Bigr).
\end{split}
\end{equation*}
Therefore, if 
$ \| h \|_{[s,t],\varrho-\textrm{\rm var}}  \leq 1$, then
\begin{equation*}
\begin{split}
&\Bigl\langle {\mathbb W}^{\indep}(\omega,\cdot) \Bigr\rangle_{q;[s,t],(p/2)-\textrm{\rm var}}^{p/2}
 \leq c \Bigl( 
\Bigl\langle
\talloblong T_{h \oplus 0} {\boldsymbol W}(\omega,\cdot)
\talloblong^p_{[0,T],(1/p)-\textrm{\rm H\"ol}}
\Bigr\rangle_{q}
(t-s)
 +  \| h \|_{[s,t],\varrho-\textrm{\rm var}}^{\varrho} \Bigr).
\end{split}
\end{equation*}
Observe that if the left-hand side is equal to or less than {$\alpha$}, the above statement remains true even if $ \| h \|_{[s,t],\varrho-\textrm{\rm var}} > 1$; it suffices to change the constant $c$ accordingly. Define now $N([0,T],\omega,\alpha) := N_{\varpi}([0,T],\alpha)$, 
when 
${\varpi(s,t)} = \bigl\langle {\mathbb W}^{\indep}(\omega,\cdot) \bigr\rangle_{q;[s,t],(p/2)-\textrm{\rm var}}^{p/2}$. 
Then, {by super-additivity of $ \| \cdot \|_{\varrho-\textrm{\rm var}}^{\varrho}$},
\begin{equation*}
N([0,T],\omega,\alpha) {\alpha} \leq  c \Bigl( 
\Bigl\langle
\talloblong T_{h \oplus 0} {\boldsymbol W}(\omega,\cdot)
\talloblong^p_{[0,T],(1/p)-\textrm{\rm H\"ol}}
\Bigr\rangle_{q}
T
 +  \| h \|_{[0,T],\varrho-\textrm{\rm var}}^{\varrho} \Bigr).
 \end{equation*}
By Proposition 11.2 in \cite{FrizHairer}, we get  (for a new value of $c$)
\begin{equation*}
N([0,T],\omega,\alpha) {\alpha} \leq  c \Bigl( 
\Bigl\langle
\talloblong T_{h \oplus 0} {\boldsymbol W}(\omega,\cdot)
\talloblong^p_{[0,T],(1/p)-\textrm{\rm H\"ol}}
\Bigr\rangle_{q}
T
 +  \| h \|_{\mathcal H}^{\varrho} \sqrt{T} \Bigr),
 \end{equation*} 
where $\| \cdot \|_{{\mathcal H}}$ is the standard norm on the reproducing Hilbert space ${\mathcal H}$, see again for instance Appendix D in \cite{FrizVictoirBook}.  We conclude by recalling that the quantity $\big\llangle \talloblong \hspace{-1pt}{\boldsymbol W}(\cdot,\cdot) \hspace{-1pt} \talloblong^p_{[0,T],(1/p)-\textrm{\rm H\"ol}} \big\rrangle_{q}$ is finite, by observing that 
\begin{equation*}
E := \Bigl\{ (\omega,\omega') \in \Omega^2 : T_{h \oplus 0} {\boldsymbol W}(\omega,\omega') = {\boldsymbol W}(\omega+h,\omega'), \quad h \in 
{\mathcal H} \Bigr\},
\end{equation*}
is of full $\PP^{\otimes 2}$-probability measure, see Theorems 11.5 and 11.9 in \cite{FrizHairer}, and then by invoking Theorem 11.7 in \cite{FrizHairer}. 

As for the sub exponential integrability of $w(0,T,\cdot)$, we just proceed with the tails 
of $\Omega \ni \omega \mapsto \big\langle \WW^{\indep}(\omega,\cdot) \big\rangle_{q ; [0,T],p/2-\textrm{\rm var}}^{p/2}$. 
To do so, it suffices to prove that the integral $\int_{\Omega}  \exp\bigl( \bigl\langle 
 \big\| \WW^{\indep}(\omega,\cdot) \big\|_{[0,T],(2/p)-\textrm{\rm H\"ol}}^q \bigr\rangle^{\varepsilon/q}
\bigr) d {\mathbb P}(\omega)$ is finite, for some $\varepsilon >0$. 
We then notice that 
the function $(0,+\infty) \ni x \mapsto \exp \bigl(  x^{\varepsilon/q} \bigr)$,
is convex on $[A_{\varepsilon},\infty)$, for some $A_{\varepsilon} >0$. Therefore, 
Jensen's inequality says that 
it suffices to prove that 
\begin{equation*}
\begin{split}
\int_{\Omega^2}
\exp\bigl(  
A_{\varepsilon}^{\varepsilon/q}
\vee
 \big\| \WW^{\indep}(\omega,\omega') \big\|_{[0,T],(2/p)-\textrm{\rm H\"ol}}^{\varepsilon}
\bigr) d {\mathbb P}(\omega)d {\mathbb P}(\omega') < \infty,
\end{split}
\end{equation*} 
which follows from 
Proposition 6.2 in \cite{CassLittererLyons}
and 
Theorem 11.13 in \cite{FrizHairer}, provided we choose $\varepsilon$ small enough.

\section*{Acknowledgement}
{
I. Bailleul thanks the Centre Henri Lebesgue ANR-11-LABX-0020-01 for its stimulating mathematical research programs, and the U.B.O. for their hospitality, part of this work was written there. I. Bailleul also thanks ANR-16-CE40-0020-01. 
F. Delarue thanks the Institut Universitaire de France.}  

{The authors thank William Salkeld (University of Edinburgh) who found several mistakes in the first version \cite{BCDArxiv} of our work.}


\begin{thebibliography}{AAA}

%\bibitem{BailleulCatellier}
%Bailleul, I. and Catellier, R.,
%\newblock Limit theorems for systems of mean field rough differential equations.
%\newblock {\em Preprint}, 2018.


\bibitem{AlfonsiJourdain}
Alfonsi A., and Jourdain, B. 
\newblock 
Lifted and geometric differentiability of the squared quadratic Wasserstein distance. 
\newblock arxiv.org/abs/1811.07787, 2018. 

\bibitem{BailleulRiedel}
Bailleul, I. and Riedel, S.,
\newblock Rough flows.
\newblock arXiv:1505.01692, 2015.

\bibitem{BailleulRMI}
Bailleul, I.,
\newblock Flows driven by rough paths.
\newblock {\em Revista Mat. Iberoamericana}, {31}(3):901--934, 2015.

\bibitem{BCDParticleSystem}
Bailleul, I., Catellier, R., Delarue, F. 
\newblock Propagation of chaos for mean field rough differential equations.
\newblock {\em In preparation}, 2019.

\bibitem{BCDArxiv}
Bailleul, I., Catellier, R., Delarue, F. 
\newblock Mean field rough differential equations.
\newblock arXiv:1802.05882, 2018.


\bibitem{BlackwellDubins}
Blackwell, D. and Dubins, L.
\newblock An extension of Skorohod's almost sure representation theorem.
\newblock {\em Proc. Am. Nat. Soc.}, {\bf 89}(4):691--692, 1983.

\bibitem{BudhirajaDupuisFischer}
Budhiraja, A. and Dupuis, P. and Fischer, M.,
\newblock Large deviation properties of weakly interacting processes via weak convergence methods.
\newblock {\em Ann. Probab.}, {\bf 40}(1):74--102, 2012.

\bibitem{BudhirajaWu}
Budhiraja, A. and Wu, 
\newblock Moderate Deviation Principles for Weakly Interacting Particle Systems.
\newblock {\em Probab. Th. Rel. Fields}, {\bf 168}(2-4):721--771, 2017.

\bibitem{LionsCardialiaguet}
Cardaliaguet, P.,
\newblock Notes on mean field games.
\newblock https://www.ceremade.dauphine.fr/ cardaliaguet/MFG20130420.pdf, 2013.

\bibitem{CarmonaDelarue_book_I}
Carmona, R. and Delarue, F.,
\newblock {\em Probabilistic Theory of Mean Field Games: vol. I, Mean Field FBSDEs, Control, and Games.}
\newblock Probability Theory and Stochastic
Modelling, Springer Verlag, 2018.	

\bibitem{CarmonaDelarue_book_II}
Carmona, R. and Delarue, F.,
\newblock {\em Probabilistic Theory of Mean Field Games: vol. II, Mean Field Games with Common Noise and Master Equations.}
\newblock Probability Theory and Stochastic
Modelling, Springer Verlag, 2018.	

\bibitem{CassLittererLyons}
Cass, T. and Litterer, C. and Lyons, T.,
\newblock Integrability and tail estimates for Gaussian rough differential equations.
\newblock {\em Ann. Probab.}, {\bf 41}(4):3026--3050, 2013.

\bibitem{CassLyons}
Cass, T. and Lyons, T.,
\newblock Evolving communities with individual preferences.
\newblock {\em Proc. Lond. Math. Soc. (3)}, {\bf 110}(1):83--107, 2015.

\bibitem{CassOgrodnik}
Cass, T. and Ogrodnik, M.,
\newblock Tail estimates for Markovian rough paths.
\newblock {\em Ann. Probab.}, {\bf 45}(4):2477--2504, 2017.

\bibitem{CoutinLejay}
Coutin, L. and Lejay, A.,
\newblock Perturbed linear rough differential equations.
\newblock {\em Ann. Math. Blaise Pascal}, {\bf 21}(1):103--150, 2014.

\bibitem{CoutinLejay2}
Coutin, L. and Lejay, A.,
\newblock Sensitivity of rough differential equations: an approach through the Omega lemma.
\newblock {\em J. Diff. Eq.},
\textbf{264}(6):3899--3917, 
2018.

\bibitem{DawsonGartner}
Dawson, D. and G\"artner, J.,
\newblock Large deviations from the McKean-Vtasov limit for weakly interacting diffusions.
\newblock {\em Stochastics}, {\bf 20}:247--308, 1987.

\bibitem{DawsonVaillancourt}
Dawson, D. and Vaillancourt, J.,
\newblock Stochastic McKean-Vlasov equations.
\newblock {\em Nonlinear Diff. Eq. Appl.}, {\bf 2},199--229, 1995.

\bibitem{Dereich}
Dereich, S.,
\newblock Rough paths analysis of general Banach space-valued Wiener process.
\newblock {\em J. Funct. Analysis}, {\bf 258}:2910--2936, 2010.

\bibitem{DereichDimitroff}
Dereich, S. and Dimitroff, G.,
\newblock A support theorem and a large deviation principle for Kunita flows.
\newblock {\em Stoch. Dyn.},{\bf 12}(03):115022, 2012.

\bibitem{FrizMaurelli}
Deuschel, J.-D. and Friz, P. and Maurelli, M. and Slowik, M.
\newblock The enhanced Sanov theorem and propagation of chaos.
\newblock 
{\em Stoc. Proc. App.}, 
{\bf 128}(7):2228-2269, 2018.



%\bibitem{DiehlOberhauserRiedel}
%Diehl, J. and Oberhauser, H. and Riedel, S., 
%\newblock A L\'evy area between Brownian motion and rough paths with applications to robust nonlinear filtering and rough partial differential equations.
%\newblock {\em Stochastic Process. Appl.}, {\bf 125}(1):161--181, 2015.

\bibitem{FeyeldelaPradelle}
Feyel, D. and de la Pradelle, A.,
\newblock Curvilinear Integrals Along Enriched Paths.
\newblock {\em Elec. J. Probab.}, {\bf 11}(34):860--892, 2006.

%\bibitem{FournierGuillin}
%Fournier, N. and Guillin, A.,
%\newblock On the rate of convergence in the {W}asserstein distance of the empirical measure.
%\newblock
%{\em Probab. Theory Related Fields},
%{\bf 162}:707-738, 2015.

\bibitem{FrizHairer}
Friz, P. and Hairer, M.,
\newblock A course on rough paths, with an introduction to regularity structures.
\newblock {\em Universitext, Springer}, 2014.

%\bibitem{FrizRiedel}
%Friz, P. and Riedel, S.,
%\newblock Integrability of (non-)linear rough differential equations and integrals.
%\newblock {\em Stoch. Anal. Appl.}, {\bf 31}(2):336--358, 2013.

\bibitem{FrizVictoirBook}
Friz, P. and Victoir, N.,
\newblock Multidimensional stochastic processes as rough paths.
\newblock {\em Cambridge studies in advanced Mathematics}, {\bf 120}, 2010.

\bibitem{Gartner}
G\"artner, J.,
\newblock On the McKean-Vlasov limit for interacting diffusions.
\newblock {\em Math. Nachr.}, {\bf 137}:197--248, 1988.

\bibitem{Gubinelli}
Gubinelli, M.,
\newblock Controlling rough paths.
\newblock {\em J. Funct. Anal.}, {\bf 216}(1):86--140, 2004.



\bibitem{jourdain-meleard}
B.~Jourdain and S.~M\'el\'eard, \emph{Propagation of chaos and fluctuations for
  a moderate model with smooth initial data},  \textbf{34} (1998), 727--766.

\bibitem{Kac}
Kac, M.,
\newblock Foundations of kinetic theory.
\newblock {\em Third Berkeley Symp. on Math. Stat. and Probab.}, {\bf 3}:171--197, 1956. 

\bibitem{KellyMelbourne}
Kelly, D. and Melbourne, I.,
\newblock Deterministic homogenization for fast-slow systems with chaotic noise.
\newblock {\em J. Funct. Anal.},
{\bf 272}(10):4063--4102,
2017.

\bibitem{LedouxLyonsQian}
Ledoux, M. and Lyons, T. and Qian, Z.,
\newblock L\'evy area of Wiener processes in Banach spaces.
\newblock {\em Ann. Probab.}, {\bf 30}(2):546--578, 2002.

\bibitem{Lions}
Lions, P.-L.,
\newblock Th\'eorie des jeux \`a champs moyen et applications,
\newblock {\em Lectures at the {C}oll\`ege de {F}rance}. http://www.college-de-france.fr/default/EN/all/equ\_der/cours\_et\_seminaires.htm", 2007-2008

\bibitem{Lyons98}
Lyons, T.,
\newblock Differential equations driven by rough paths.
\newblock {\em Rev. Mat. Iberoamericana}, {\bf 14}(2):215--310, 1998.

\bibitem{LyonsQianDiffeo}
Lyons, T. and Qian, Z.,
\newblock Flows of diffeomorphisms induced by geometric multiplicative functionals.
\newblock {\em Probab. Th. and Rel. Fields}, {\bf 112}:91--119, 1998.

\bibitem{McKean}
McKean, H. P.,
\newblock A class of markov processes associated with nonlinear parabolic equations.
\newblock {\em Prov. Nat. Acad. Sci.}, {\bf 56}:1907--1911, 1966.

\bibitem{meleard}
S.~M{\'e}l{\'e}ard, \emph{Asymptotic behaviour of some interacting particle
  systems; {M}c{K}ean-{V}lasov and {B}oltzmann models}, Probabilistic models
  for nonlinear partial differential equations ({M}ontecatini {T}erme, 1995),
  Lecture Notes in Math., vol. 1627, Springer, 1996, pp.~42--95.



\bibitem{Sznitman}
Sznitman, A.-S.,
\newblock Topics in propagation of chaos.
\newblock {\em Lect. Notes Math.}, {\bf 1464}, 1991.

\bibitem{Tanaka}
Tanaka, H.,
\newblock Probabilistic treatment of the Boltzman equation of Maxwellian molecules.
\newblock {\em Probab. Th. Rel. Fields}, {\bf 46}:67--105, 1978.

\bibitem{WuZhang}
Wu, C. and Zhang, J.,
\newblock An elementary proof for the structure of derivatives in probability measures.
\newblock arXiv:1705.08046, 2017.

\end{thebibliography}
\end{document}